%% file: Heat-a4-v4.tex
\documentclass[10pt,a4paper, fleqn]{article}
\usepackage{etex}
\usepackage[utf8]{inputenc}
\usepackage[TS1,T1]{fontenc}
\usepackage{xspace}
\usepackage[a4paper,tmargin=3truecm,bmargin=3truecm,hmargin=1.6truecm]{geometry}

\usepackage[all]{xy}
\usepackage{xcolor}

\usepackage[english]{babel}
\usepackage[bitstream-charter]{mathdesign}

\usepackage{enumitem}
\newlist{enum-hypothesis}{enumerate}{1}
\setlist[enum-hypothesis]{label=(\arabic*),itemsep=0pt, parsep=0pt}

\usepackage{amsmath,mathtools,slashed}
\usepackage[thmmarks,amsmath]{ntheorem}
\usepackage{tensor}

\usepackage{hyperref}

\usepackage{pbox}

\newtheorem{theorem}{Theorem}[section]
\newtheorem{proposition}[theorem]{Proposition}
\newtheorem{lemma}[theorem]{Lemma}
\newtheorem{corollary}[theorem]{Corollary}

\newtheorem{definition}[theorem]{Definition}

\theoremsymbol{$\square$}
\theoremstyle{plain}
\theorembodyfont{\upshape}
\newtheorem{remark}[theorem]{Remark}

\theoremsymbol{}
\theoremstyle{break}
\theorembodyfont{\slshape}

\theoremheaderfont{\scshape}
\theorembodyfont{\upshape} 
\theoremstyle{nonumberplain}
\theoremseparator{} 
\theoremsymbol{\rule{1.3ex}{1.3ex}} 
\newtheorem{proof}{Proof}

\usepackage[square,numbers,sort,compress,merge]{natbib}

\hypersetup{
plainpages=false,
colorlinks=true,
linkcolor=black, 
anchorcolor=black, 
citecolor=black, 
urlcolor=black, 
menucolor=black, 
filecolor=black, 
bookmarksopen=true,
bookmarksnumbered=true}

\numberwithin{equation}{section}

\hypersetup{
pdfauthor={Thierry Masson, Bruno Iochum},
pdftitle={Heat coefficient a_4 for nonminimal Laplace type operators},
pdfsubject={},
pdfcreator={pdflatex},
pdfproducer={pdflatex},
pdfkeywords={}
}

\newcommand{\bbB}{\mathbb{B}}
\newcommand{\bbC}{\mathbb{C}}

\newcommand{\bbN}{\mathbb{N}}
\newcommand{\bbR}{\mathbb{R}}
\newcommand{\bbS}{\mathbb{S}}

\newcommand\bbZ{\mathbb{Z}}

\newcommand{\bbbone}{{\text{\usefont{U}{dsss}{m}{n}\char49}}}

\newcommand{\calA}{\mathcal{A}}
\newcommand{\calB}{\mathcal{B}}
\newcommand{\calC}{\mathcal{C}}

\newcommand{\calE}{\mathcal{E}}
\newcommand{\calH}{\mathcal{H}}
\newcommand{\calK}{\mathcal{K}}

\newcommand{\calR}{\mathcal{R}}

\newcommand{\calU}{\mathcal{U}}

\newcommand{\pmu}{\partial_\mu}  
\newcommand{\pnu}{\partial_\nu}   
\newcommand{\prho}{\partial_\rho}  
\newcommand{\psigma}{\partial_\sigma}  

\newcommand{\nmu}{\nabla_{\!\mu}}   
\newcommand{\nnu}{\nabla_{\!\!\nu}}   
\newcommand{\nrho}{\nabla_{\!\rho}}   

\newcommand{\tonc}{\xrightarrow{\text{n.c.}}} 

\newcommand{\hnabla}{\widehat{\nabla}}

\newcommand{\hnmu}{\hnabla_{\!\mu}}
\newcommand{\hnnu}{\hnabla_{\!\nu}}

\newcommand{\hDelta}{\widehat{\Delta}}

\newcommand{\mul}{{\bf m}}
\newcommand{\gnabla}{{}^g{\nabla}}
\newcommand{\Ric}{\text{Ric}}
\newcommand{\SC}{\mathfrak{R}}
\newcommand{\SO}{X}

\newcommand{\dd}{\text{d}}
\newcommand{\dvolg}{\text{dvol}_g}

\newcommand{\vc}{\vcentcolon =}   
\newcommand{\cv}{=\vcentcolon }   
\newcommand{\OO}{{\text{\usefont{OMS}{cmsy}{m}{n}O}}}   

\def\comFu_#1{[F_{#1}, u]}
\def\comFUu^#1{[F^{#1}, u]}
\def\comFN^#1_#2{[F_{#2}, N^{#1}]}
\def\anticomHLHN_#1{\{\hDelta, \hnabla_{#1}\}}
\def\Sumgigi^#1{G^{#1}}
\def\Sumgg_#1{G_{#1}}

\newcounter{termscounter}
\newcommand\EqLine[1]{
\stepcounter{termscounter}%
}

\DeclareMathOperator{\Tr}{Tr}	   
\DeclareMathOperator{\tr}{tr}	   
\DeclareMathOperator{\End}{End}	 

\DeclareMathOperator*{\sumperm}{%
\mathchoice%
{\sum\kern-7pt\raise-0.5pt\hbox{\small$\mathsf{P}$}}
{\sum\kern-5.5pt\raise-0.4pt\hbox{\scriptsize$\mathsf{P}$}}
{}
{}
} 

\DeclarePairedDelimiter\abs{\lvert}{\rvert}

\newcounter{mnotecount}[section]
\renewcommand{\themnotecount}{\thesection.\arabic{mnotecount}}
\newcommand{\mnote}[1]%
{\protect{\stepcounter{mnotecount}}${}^{\text{\footnotesize$\bullet$\themnotecount}}$%
\reversemarginpar%
\marginpar{\raggedleft\footnotesize$\bullet$\themnotecount: #1}}

\newlength{\mnotewidth}
\allowdisplaybreaks

\begin{document}
{
\makeatletter\def\@fnsymbol{\@arabic}\makeatother 
\title{Heat coefficient $a_4$ for non minimal Laplace type operators}
\author{B. Iochum, T. Masson\\
{\small Centre de Physique Théorique} \footnote{bruno.iochum@cpt.univ-mrs.fr, thierry.masson@cpt.univ-mrs.fr}\\
\small{Aix Marseille Univ, Université de Toulon, CNRS, CPT, Marseille, France}\\[2ex]
}
\date{}
\maketitle
}

\begin{abstract}
Given a smooth hermitean vector bundle $V$ of fiber $\mathbb{C}^N$ over a compact Riemannian manifold and $\nabla$ a covariant derivative on $V$, let $P = -(\lvert g \rvert^{-1/2} \nabla_\mu \lvert g \rvert^{1/2} g^{\mu\nu} u \nabla_\nu + p^\mu \nabla_\mu +q)$ be a non minimal Laplace type operator acting on smooth sections of $V$ where $u,\,p^\nu,\,q$ are $M_N(\mathbb{C})$-valued functions with $u$ positive and invertible. 
For any $a \in \Gamma(\text{End}(V))$, we consider the asymptotics $\text{Tr} \,a \,e^{-tP} \sim_{t \downarrow 0} \,\sum_{r=0}^\infty a_r(a, P)\,t^{(r-d)/2}$ where the coefficients $a_r(a, P)$ can be written as an integral of the functions $a_r(a, P)(x) = \text{tr}\,[a(x) \,\mathcal{R}_r(x)]$.\\
This paper revisits the previous computation of $\mathcal{R}_2$ by the authors and is mainly devoted to a computation of $\mathcal{R}_4$. The results are presented with $u$-dependent operators which are universal (\textsl{i.e.} $P$-independent) and which act on tensor products of $u$, $p^\mu$, $q$ and their derivatives via (also universal) spectral functions which are fully described. 
\end{abstract}


\section{An introduction to the method}

Let $V$ be a smooth hermitean vector bundle $V$ of fiber $\bbC^N$ over a compact $d$-dimensional boundaryless Riemannian manifold $(M,g)$ and let $P$ be a non minimal Laplace type operator acting on the smooth sections $\Gamma(V)$, written locally as the partial differential operator $P= -[g^{\mu\nu} u(x) \,\partial_\mu\partial_\nu +v^\mu(x)\,\partial_\mu +w(x)]$ where $x\in M$ and $u, v^\mu, w$ are matrices in $M_N(\bbC)$ with $u(x)$ positive and invertible. \\
For a smooth section $a \in \Gamma(\End(V))$, the existence of an asymptotics for the heat-trace $\Tr\, a e^{-tP} \sim_{t \downarrow 0} \,\sum_{r=0}^\infty a_r(a, P)\,t^{(r-d)/2}$ is known (see \cite{Gilk95a,Gilk03a}), with coefficients given by $a_r(a, P) = \int_M \dvolg(x)\, a_r(a, P)(x) $, where $\dvolg(x) \vc \abs{g}^{1/2} \dd x$ and $\abs{g} \vc \det(g_{\mu\nu})$; more precisely given by
\begin{equation*}
a_r(a, P)(x) \vc \tr\,[a(x)\, \calR_r(x)]
\end{equation*}
where $\tr$ is the trace on $M_N(\bbC)$, and $\calR_r$ is a (local) section of $\End(V)$. 

The explicit knowledge of the $a_r$ and $\calR_r$ is important both in mathematics and physics, and several attempts can be found in the literature for many classes of operators $P$, see the books \cite{BerlGetzVerg92a, Gilk95a, Avra15a, Vassilevich}. While here we extend a previous method \cite{IochMass17a, IochMass17c}, this paper is actually self-contained.

To start with, it is convenient to use a covariant derivative $\nabla$ on $V$ and to parameterize the differential operator $P$ as:
\begin{align}
P &\vc -(\abs{g}^{-1/2} \nmu \abs{g}^{1/2} g^{\mu\nu} u \nnu + p^\mu \nmu +q)
= - g^{\mu\nu} u \nmu \!\nnu - \big( p^\nu + g^{\mu\nu} (\nmu u) - \Gamma^\nu u \big) \nnu - q,
\label{eq-def-P-upq}
\end{align}
where $u, p^\mu, q$ are sections of $\End(V)$ (see \cite[Appendix A.4]{IochMass17a} for the swap between $(v^\mu,w)$ and $(p^\mu,q)$) and 
\begin{align*}
\Gamma^\nu \vc g^{\mu\rho} \Gamma^\nu_{\mu\rho} \text{ where the }\Gamma_{\mu \rho}^\nu \text{ are the Christoffel symbols of } g.
\end{align*}

The computation of $\calR_r$ is realized through $\Tr\,a e^{-tP} = \int_M  \!\dvolg(x)\, \tr\,[a(x) \,\calK(t,x,x)]$, where $\calK(t,x,x)$ is the diagonal of the kernel of $e^{-tP}$ defined for any section $s \in \Gamma(V)$ by 
$\int_M \!\dvolg(y) \, \calK(t,x,y) \,s(y) = (e^{-tP} s)(x)$. 
Recall that this kernel can be computed using a compactly supported section $s$ with support in a open subset $U$ of $M$ which gives at the same time a chart of $M$ and a trivialization of $V$ (and $\End(V)$). In that situation, we can look at $s$ as a map $s : U \to \bbC^N$. \\
The use of the Fourier transform  $\widehat{s}\,(\xi) \vc (2 \pi)^{-d/2} \int_U \dd y \, e^{- i y \cdot\xi} s(y)$ of $s$ with inverse $s(x) = (2 \pi)^{-d/2} \int_{\bbR^d} \dd \xi \, e^{i x \cdot\xi}\, \widehat{s}\,(\xi)$ yields to
\begin{align*}
(e^{-tP} s)(x) 
= (2 \pi)^{-d/2} \int_{\bbR^d} \!\!\!\dd \xi \, (e^{-tP} e^{i x \cdot\xi}) \,\widehat{s}\,(\xi) 
= (2 \pi)^{-d} \int_{\bbR^d} \!\!\!\dd \xi \int_U \dd y \,  e^{- i y \cdot\xi}(e^{-tP} e^{i x \cdot\xi}) \,s(y).
\end{align*}
For fixed $(x,y)\in M\times M$, we can look at  $e^{- i y \cdot\xi}(e^{-tP} e^{i x \cdot\xi})$, and then at $\calK(t,x,y) = (2 \pi)^{-d} \abs{g}^{-1/2}(y) \int_{\bbR^d} \!\dd \xi \, e^{- i y \cdot\xi}(e^{-tP} e^{i x \cdot\xi})$, as maps $\bbC^N \to \bbC^N$, because the function $x \mapsto e^{i x \cdot\xi}$ “absorbs” all the derivatives operators in $e^{-tP}$. We now give another expression for the matrix $\calK(t,x,y) \in M_N(\bbC)$. Since for any $s : U \to \bbC^N$, 
\begin{align}
&- (P \,e^{i x \xi} s)(x)  \nonumber
\\
&\qquad=
e^{i x \xi} \big( \big[ 
-g^{\mu\nu} u \xi_\mu \xi_\nu 
+ i \xi_\mu \big( p^\mu + g^{\mu\nu} (\nnu \,u) - \Gamma^\nu u + 2 g^{\mu\nu} u \nnu \big) 
+u\Delta 
+ \big( p^\nu + g^{\mu\nu} (\nmu u) - \Gamma^\nu u \big) \,\nnu 
+ q
\big] \,s \big) (x) \label{eq-P and exp}
\nonumber
\\
&\qquad= - e^{i x \xi} ([ H + K + P ] \,s)(x)
\end{align}
where we have introduced
\begin{align}
\Delta &\vc -g^{\mu\nu}\, \nmu \! \nnu, \nonumber\\
K&\vc - i \xi_\mu \big( p^\mu + g^{\mu\nu} (\nnu u) - \Gamma^\mu u + 2 g^{\mu\nu} u \nnu \big),
\label{eq-def-K-up}
\\
H &\vc g^{\mu\nu} u \,\xi_\mu \xi_\nu,
\label{eq-def-Hmunu}
\end{align}
one gets $e^{-i y \cdot \xi} (e^{-tP} \,e^{i x \cdot \xi} \,s)(x) = e^{i (x - y) \cdot \xi} (e^{-t\,[ H + K + P ]} \,s)(x)$. \\
The operators $H$ and $K$ being essential here, they will be also decomposed as $H = \xi_\mu \xi_\nu \,H^{\mu\nu}$ and $K = \xi_\mu \,K^\mu$ where
\begin{align}
&H^{\mu\nu}
\vc g^{\mu\nu} u, \nonumber
\\
\nonumber
&K^\mu \vc -i \big( p^\mu + g^{\mu\nu} (\nnu u) - \Gamma^\mu u + 2 g^{\mu\nu} u \nnu \big)
=  - i \big( L^\mu +  2 H^{\mu\nu} \nnu \big),\\
\nonumber
&L^\mu
\vc N^\mu
- \Gamma^\mu u,
 \\
\label{eq-def-Nmu}
&N^\mu  \vc p^\mu + g^{\mu\nu} (\nnu\, u),
\end{align}
and thus 
\begin{equation}
\label{def-PHLq}
P = - H^{\mu\nu}\, \nmu \nnu - L^\mu\, \nmu - q.
\end{equation}
The introduction of the variable $N^\mu$ will be justified in Lemma \ref{lem-N=0}.

Then, for any $v \in \bbC^N$ and at fixed values $x$ and $y$,
\begin{align*}
\calK(t,x,y)\, v 
&= 
(2 \pi)^{-d} \abs{g}^{-1/2}(y) \int_{\bbR^d} \!\!\!\dd \xi \, e^{i (x - y) \cdot \xi} e^{-t\,[ H + K + P ]}\, v 
= 
t^{- d/2} (2 \pi)^{-d} \abs{g}^{-1/2}(y) \int_{\bbR^d} \!\!\!\dd \xi \, e^{i (x - y) \cdot \xi} e^{- H - \sqrt{t} \,K - t\, P}\, v,
\end{align*}
and the diagonal part of the kernel is
$\bbC^N \ni v \mapsto \calK(t,x,x)\, v = t^{- d/2} (2 \pi)^{-d} \abs{g}^{-1/2} \int_{\bbR^d} \dd \xi \, e^{- H - \sqrt{t}\, K - t\, P} v$.\\
From now on, $v \in \bbC^N$ will be considered as a locally constant section of $V$. 

The computation of this matrix-valued function  $\calK$ is based on the Volterra series  
\begin{align*}
e^{\,A+B}
=
e^{\,A} +\sum_{k=1}^\infty \,\int_{\Delta_k} \dd s\, e^{\,(1-s_1)\,A}\,B\,e^{\,(s_1-s_{2})\,A} \cdots e^{\,(s_{k-1}-s_k)\,A} \,B \,e^{\,s_k \,A}\,,
\end{align*}
where
\begin{equation*}
\Delta_k \vc \{s = (s_1,\dots, s_k) \in \bbR_+^{k} \,\, \vert \,\, 0 \leq s_k \leq s_{k-1} \leq \cdots \leq s_2 \leq s _1 \leq 1\}\quad \text{(we also use the convention }s_0\vc 1,\,s_{k+1}\vc 0),
\end{equation*}
and by convention $\Delta_0 \vc \varnothing$. This gives, for $A = -H$ and $B = -t^{1/2} K -tP$ (omitting the $\xi$-dependence),
\begin{align}
e^{-H-\sqrt{t}\, K -tP} v 
= e^{-H } v + \sum_{k=1}^\infty (-1)^k f_k [ (\sqrt{t} K +t P)^{\otimes ^k} ] \,v
\label{Volt}
\end{align}
where for any $k \in \bbN$, the map $f_k (\xi) : M_N[\xi,\nabla]^{\otimes^k} \to M_N[\xi]$ is defined, for any $v \in \bbC^N$, by
\begin{align}
\begin{aligned}
& f_k(\xi) [ B_1 \otimes \cdots \otimes  B_k ] \,v
\vc \int_{\Delta_k}\! \dd s \, e^{\,(s_1-s_0)\, H(\xi)} \, B_1 \,e^{\,(s_2-s_1) \,H(\xi)} \, B_2 \cdots B_k\, e^{(s_{k+1}-s_k) \,H(\xi)}\, v,
\\
& f_0(\xi)[\lambda] v
\vc \lambda\, e^{-H(\xi)} v, \quad \text{for } \lambda\in \bbC \cv M_N(\bbC)^{\otimes^0}.
\end{aligned}
\label{eq-fk-def}
\end{align}
Here the $B_i$ are matrix-valued differential operators in $\nmu$ depending on $x$ and (linearly in) $\xi$, and $\lambda \in \bbC$. 
\begin{remark}[Notation and convention]
\label{rmk fk convention}
By convention, each $\nmu$ in $B_i\in M_N[\xi,\nabla]$ acts on all its right remaining terms. Since $v$ is constant, \emph{this allows to identify $f_k(\xi) [ B_1 \otimes \cdots \otimes  B_k ]$ as matrix valued functions}, $v \mapsto f_k(\xi) [ B_1 \otimes \cdots \otimes  B_k ] \,v$, even if the $B_i$'s contain derivative operators. 
\end{remark}

Proposition~\ref{lem-fk-nabla-vector-propagation} will get rid of these explicit differential operators in the arguments of $f_k(\xi)$ and will produce formulas with matrix-valued arguments only. This is an essential result for the method.

The first terms of \eqref{Volt} are (omitting again the $\xi$-dependence)
\begin{align*}
e^{-H-\sqrt{t} K -tP} v
= 
& \,f_0[1]\,  v
 - t^{1/2} f_1[K] \,v
+ t\big(f_2[K\otimes K]-f_1[P] \big)\, v
 + t^{3/2} \big(f_2[K\otimes P]+f_2[P\otimes K]-f_3[K\otimes K\otimes K]\big) \,v 
 \\
 & + t^2 \big( 
f_2[P\otimes P] - f_3[K\otimes K\otimes P]- f_3[K\otimes P \otimes K] - f_3[P \otimes K \otimes K]
+ f_4[K \otimes K \otimes K \otimes K]\big)\, v  +\OO(t^2).  \nonumber
\end{align*}

Since the $\xi$-integral cancels the non-integers powers of $t$, one recovers the coefficients $a_r$ via the asymptotics behavior $\Tr \,[ a e^{-tP}] \underset{t\downarrow 0}{\sim} t^{- d/2} \sum_{r=0}^\infty a_{r}(a, P) \, t^{r/2}$: at the point $x$,
\begin{align*}
&a_0(a,P)(x)
= \tr \tfrac{\abs{g}^{-1/2}}{(2 \pi)^{\,d}}\, a(x) \!\int_{\bbR^d} \dd \xi \, f_0[1],
\\
&a_2(a,P)(x)
=  \tr \tfrac{\abs{g}^{-1/2}}{(2 \pi)^{\,d}}\, a(x)\! \int_{\bbR^d}  \dd \xi \, ( f_2[K\otimes K] - f_1[P] ),
\\
&a_4(a,P)(x)
= \tr \tfrac{\abs{g}^{-1/2}}{(2 \pi)^{\,d}} \,a(x) \!\int_{\bbR^d}  \!\dd \xi \, 
\big( f_2[P\otimes P] - f_3[K\otimes K\otimes P] - f_3[K\otimes P \otimes K]
- f_3[P \otimes K \otimes K]
+ f_4[K \otimes K \otimes K \otimes K]\big).
\end{align*}
where the convention of Remark~\ref{rmk fk convention} is adopted. \\
While it is tempting to generalize $H^{\mu\nu}=g^{\mu\nu} u$ to an arbitrary strictly positive matrix $H^{\mu\nu}$ (with $H^{\mu\nu}=H^{\nu\mu}$), it is almost impossible to obtain simple expression for the $a_r$. For instance the $\xi$-integral of $a_0(a,P)(x)$ cannot be done explicitly (see for instance \cite{IochMass17a} for details and \cite[Section 3.3]{Avramidi2005} for a link with Finsler metrics).

The computation of $a_r$ goes the following way.

First the calculus of the $\xi$-integral of \eqref{eq-fk-def} leads, via the spectral decomposition of $H(\xi)=g^{\mu\nu} u \,\xi_\mu \xi_\nu$, to the universal functions 
$I_{\alpha,k}:\,(r_0, \dots, r_k)\in  (\bbR_+^*)^{\,k+1} \mapsto \int_{\Delta_k} \dd s\,\,[\,\sum_{\ell=0}^k \,(s_\ell-s_{\ell+1})\,r_\ell \,]^{\,-\alpha}$ , see Section~\ref{Section def of I} or \cite{IochMass17a, IochMass17c}. 

Then, as shown in Section~\ref{Section def of I}, the family of functions $I_{\alpha,k}$ satisfies a few relations. Moreover, in the same section, these relations  are translated in terms of operators $X_{\alpha,k}$ (see Definition~\ref{eq-def-Xalphak}) acting on the tensor product $ B_1 \otimes \cdots \otimes  B_k $ of $M_N(\bbC)$-valued differential operators and $I_{\alpha, k}$ can be seen as the spectral functions associated to the action of $X_{\alpha,k}$. The differential aspect, which is an important part of the game here, is not a difficulty because the family of $X_{\alpha,k}$ is compatible with derivations, see Proposition~\ref{prop-Propagation}. 

The previous described approach and formulae for $a_r(a,P)(x)$ are well-known, see especially \cite{AvraBran01a, AvraBran02a, Avra04a, Avramidi2005, Avra05a}. The originality of this work, compared to previous quoted ones and \cite{IochMass17a, IochMass17c}, is however to perform the computations using operators instead of spectral functions.  One can follow this way more precisely what is the contribution of each term, an information that is \textit{a priori} lost when adding spectral functions coming from different contributions. These operators, respectively $f_k(\xi)$ and $X_{\alpha,k}$, have a universal property since they only depend of a positive invertible matrix-valued which can be either $H(x,\xi)$ or $u(x)$. Thus, in this formal algebraic level, the computations are reduced to a control of the propagation of derivatives inside the arguments, see for instance Proposition~\ref{lem-fk-nabla-vector-propagation}. 

However, to secure the results on the matrices $\calR_r(x)$, we need a covariant formulation of all tools. 
This is the aim of Section~\ref{Total covariant derivative and normal coordinates} where $P$ is presented in a covariant way in equation \eqref{eq-laplacianLCu} and, since we know that $\calR_r(x)$ must then be invariant under a change of coordinates, we choose a normal coordinate system. This implies that the covariant derivative $\nabla$ has to be extended to the total covariant derivative  $\hnabla$ combining both $\nabla$ and the Levi-Civita connection on $M$. Such an extension requires for the sequel to compute beforehand a few formulae on the action of $\hnabla$ on all ingredients.

In particular, in Section~\ref{The method and its simplifications}, we present a formula for the propagation of $\nabla$ within the arguments and show several simplifications due to a few Riemannian contractions which appear all along the computations. This leads to an operational version of the method exposed in Proposition~\ref{prop-simplified-computation}.

In Section~\ref{Results on R2} the whole intermediate steps for the calculation of $\calR_2$ are given and, of course, we recover in this new algebraic setting the previous results of \cite{IochMass17a,IochMass17c}.

While all computations can be done “by hand” for $\calR_2$, the case of $\calR_4$ requires the use of a computer due to the huge number of generated terms. The hard and long part of this work was to develop a code \textit{ab initio}. 
The elaboration of such a code is explained in Section~\ref{The code}. It takes care of all intricate aspects of the computations: the algebraic manipulations quoted before, the use of normal coordinates for higher derivatives and last but not least, the simplification of a large number of terms via a reduction process, see \eqref{eq-reduction}.

Finally, a formula for $\calR_4$ is exhibited in Section~\ref{Results on R4} which is a new result. Of course, it is compatible with old known results like when $u=\bbbone$, see \cite{Gilk03a,Gilk95a}, but is written here in full generality when the section $u$ is parallel for $\nabla$ and the $N^\nu$ are not zero, while the standard presentation always assumes that $N^\nu=0$ (see Lemma \ref{lem-N=0}).

\section{\texorpdfstring{The universal operators $X_{\alpha,k}$}{The universal operators X alpha k}}
\label{Section def of X}

The aim of this section is to define and study the operators $X_{\alpha,k}$ which depends only on a strictly positive matrix-valued function $h:\mathcal{U} \to M_N(\bbC)$ where $\mathcal{U}$ is a given parameter space. Later on, when the differential operator $P$ will play a role, $h$ will be either $u(x)$ with $\mathcal{U}=U$ or $H(x,\xi)=g^{\mu\nu}(x) \,u(x)\,\xi_\mu\xi_\nu$ with $\mathcal{U}=U\times \bbR^d$.

\subsection{\texorpdfstring{The universal spectral functions $I_{\alpha,k}$}{The universal spectral functions I alpha k}}
\label{Section def of I}

Let us first consider the algebra $\calA=M_N(\bbC)$  (although many of the theory could be generalized to a unital $C^\ast$-algebra $\calA$) and $h \in \calA$ which is a positive invertible matrix: 
\begin{align*}
0<h \in \calA.
\end{align*}
For any $k\in \bbN$ and $a_\ell \in \calA$, let $[a_0^R \otimes a_1^R \otimes  \cdots \otimes a_k^R] \in \calB(\calA^{\otimes^{k+1}} ,\calA^{\otimes^{k+1}} )$ (bounded operator from $\calA^{\otimes^{k+1}}$ into itself) be defined by $[a_0^R \otimes a_1^R \otimes  \cdots \otimes a_k^R] [b_0 \otimes \cdots \otimes b_k]\vc b_0 a_0 \otimes b_1 a_1 \otimes \cdots \otimes b_k a_k$. For some reason we want to apply such operators to $\calA^{\otimes^{k}}$ and to do so, we need the injection of $\calA^{\otimes^{k}}$ into $\calA^{\otimes^{k+1}}$:
\begin{align}
\label{def iota}
\kappa_{k}:\,b_1 \otimes \cdots \otimes b_k \mapsto \bbbone \otimes b_1 \otimes \cdots \otimes b_k, \text{ for } k \geq 1, \quad \text{and}\quad  \kappa_{0}(\lambda)\vc \lambda \bbbone,\,\lambda \in \calA^0\vc \bbC.
\end{align}
where $\bbbone$ is the unit of $\calA$. Now we define the operator $(a_0^R \otimes a_1^R \otimes  \cdots \otimes a_k^R) \in \calB(\calA^{\otimes^{k}},\calA^{\otimes^{k+1}})$ as
\begin{align*}
\label{def operator action}
(a_0^R \otimes a_1^R \otimes  \cdots \otimes a_k^R) \vc [a_0^R \otimes a_1^R \otimes  \cdots \otimes a_k^R] \circ \kappa_k.
\end{align*}
We also need, for $0\leq \ell\leq k$, the following family of operators:
\begin{align*}
R_\ell(a) \vc (\bbbone \otimes \cdots \otimes \bbbone\otimes a^R\otimes \bbbone \otimes \cdots \otimes \bbbone) \in \calB(\calA^{\otimes^{k}},\calA^{\otimes^{k+1}}) \text{ where } a^R \text{ is at the }\ell\text{-th place}.
\end{align*}
For any $k$, denote by $\mul : \calA^{\otimes^{k+1}} \to \calA$ the multiplication in $\calA$. Then $\mul \circ (a_0^R \otimes a_1^R \otimes  \cdots \otimes a_k^R) \in \calB(\calA^{\otimes^{k}}, \calA)$ is given by
\begin{equation}
\label{eq- action of tensors}
\big(\mul \circ (a_0^R \otimes a_1^R \otimes  \cdots \otimes a_k^R)\big) [b_1 \otimes \cdots \otimes b_k]
=
a_0 b_1 a_1  \cdots b_k a_k.
\end{equation}

We now consider the functional calculus on $h$, with the shorthand notation
\begin{equation*}
E_\ell \vc E_{r_\ell} \text{ is the spectral projection associated to the spectral value }r_\ell \text{ of } h.
\end{equation*}
Keep in mind that $\ell$ is not the index of a spectral value but is the index of the position in the $(k+1)$-tensor product. \\
The need to compute the $\xi$-integrals of the operators $f_k(\xi)$ drives us to 
\begin{align}
\label{eq-integral of fk}
\int \dd \xi\, \xi_{\mu_1} \cdots \xi_{\mu_{2p}} \,f_k(\xi)
= \mul \circ \int_{\Delta_k} \! \dd s \,\int \dd \xi \, \xi_{\mu_1} \cdots \xi_{\mu_{2p}}\,e^{-C_k(s,\,H(\xi))}
=c_{k,\mu_1\dots\mu_{2p}}\,\mul \circ \int_{\Delta_k} \! \dd s \, C_k(s,u)^{-\,(d/2+p)},
\end{align}
where $C_k(s,a)\vc \sum_{\ell =0}^k (s_\ell-s_{\ell+1})\,R_\ell(a)$, permuting the integrals in the first equality and using a Gaussian integration with spherical coordinates in the second one, see \cite[Eq. (4.4)]{IochMass17a}. Thus the functional calculus for $h$ naturally leads us to the following functions (when $\alpha=d/2+p$):
\begin{definition}
\label{Def of I}
For any $\alpha \in \bbR$ and $k \in \bbN$, let $I_{\alpha,k}: (\bbR_+^*)^{\, k+1}\to \bbR_+$ defined by
\begin{align}
\begin{aligned}
 & I_{\alpha,k}(r_0,\dots,r_k) \vc 
\int_{\Delta_k} \!\dd s\, \,\big[\sum_{\ell=0}^k \,(s_\ell-s_{\ell+1})\,r_\ell \big]^{\,-\alpha} 
=
\int_{\Delta_k}\! \dd s\,\, \big[\sum_{\ell=0}^{k} \,s_\ell (r_\ell -r_{\ell-1}) \big]^{\,-\alpha} 
\quad(\text{with the convention }r_{-1}\vc 0)
\\
& I_{\alpha,0}\, (r_0) \vc r_0^{-\alpha}\,  \text{ for } \alpha \neq 0. 
\end{aligned}
\label{def I0}
\end{align}
\end{definition}
For instance,
\begin{align}
\label{eq-I alpha k (r0, ...,r0)}
& I_{\alpha,k}(\,\underbrace{r_0, \dots, r_0}_{k+1})  = \tfrac{1}{k!}\, r_0^{\,-\alpha},\\
& I_{\alpha,1}(r_0,r_1)= 
\begin{dcases}
\label{eq def I1} 
\, \tfrac{1}{1-\alpha} \tfrac{r_0^{\,1-\alpha} \,-\,r_1^{\,1-\alpha}}{r_0\,-\,r_1} & \text{for } \alpha \neq 1,\\
\,\tfrac{\log(r_0)\,-\,\log(r_1)}{r_0\,-\,r_1} & \text{for } \alpha =1,
\end{dcases}
\end{align}
and see \cite{IochMass17a, IochMass17c} for other explicit expressions for these integrals. \\
We will need the following recursion formulas on the functions  $I_{\alpha, k}$, seen as generalizations of \cite[eq.~(3.1)]{IochMass17a}:
\begin{lemma}
\label{lem-relations-Ialpha}
For any $\alpha \in \bbR$ and $k \in \bbN$ the family of functions $I_{\alpha,k}$ satisfies 
\begin{align}
\hspace{-0.3cm} i) \,\,
I_{\alpha,k}(r_0, \dots, r_k)
= 
r_0\, I_{\alpha+1,k+1}(r_0, r_0, \dots, r_k)
+ r_1 I_{\alpha+1,k+1}(r_0, r_1, r_1, \dots, r_k)
+ \cdots 
+ r_k I_{\alpha+1,k+1}(r_0, r_1, \dots, r_k, r_k).
\label{eq-sum-Ialpha+1k+1}
\end{align}

ii) Moreover, for any $\alpha \neq 1$, $k,\ell \in \bbN^*$ and $1 \leq \ell \leq k$, 
\begin{align}
\label{eq-recursive-I-ell}
I_{\alpha,k}(r_0, \ldots, r_k) 
= \tfrac{1}{1 - \alpha} \tfrac{1}{r_{\ell} - r_{\ell-1}} \left[
I_{\alpha-1,k-1}(r_0, \ldots, \widehat{r_{\ell - 1}}, \ldots, r_k)
-
I_{\alpha-1,k-1}(r_0, \ldots, \widehat{r_{\ell}}, \ldots, r_k)
\right].
\end{align}
\end{lemma}

\begin{proof}
i) For $0 \leq \ell \leq k$, the integrand in the defining integral of $I_{\alpha+1,k+1}(r_0, \dots, r_\ell, r_\ell, \dots, r_k)$ reduces to
\begin{equation}
\begin{dcases}
\,[(s_0 - s_2) r_0 + (s_2 - s_3) r_1 + \dots + s_{k+1} r_k]^{\,-(\alpha+1)} 
&
\text{for $\ell = 0$,}
\\
\,[(s_0 - s_1) r_0 + \dots + (s_\ell - s_{\ell+2}) r_\ell + \cdots + s_{k+1} r_k]^{\,-(\alpha+1)} 
&
\text{for $0 < \ell< k$,}
\\
\,[(s_0 - s_1) r_0 + \dots + (s_{k-1} - s_k) r_{k-1} + s_k r_k]^{\,-(\alpha+1)} 
&
\text{for $\ell = k$,}
\end{dcases}
\label{eq-integrands-ell}
\end{equation}
and then does not depend anymore of the variable $s_{\ell+1}$. Using
\begin{equation*}
\int_{\Delta_{k+1}} \!\!\!\! \dd s 
= 
\begin{dcases}
    \int_0^{1}\! \dd s_2 
    \cdots 
    \int_0^{s_{k}} \!\dd s_{k+1}
    \int_{s_{2}}^{1} \!\dd s_{1}
    &
    \text{for $\ell = 0$,}
    \\  
    \int_0^1 \dd s_1  
    \cdots 
    \int_0^{s_{\ell-1}} \!\dd s_{\ell} 
    \int_0^{s_{\ell}} \!\dd s_{\ell + 2} 
    \int_0^{s_{\ell + 2}}\! \dd s_{\ell + 3} 
    \cdots 
    \int_0^{s_{k}} \!\dd s_{k+1}
    \int_{s_{\ell +2}}^{s_{\ell}}\! \dd s_{\ell+1}
    &
    \text{for $0<\ell<k$,}
    \\
    \int_0^1 \dd s_1  
    \int_0^{s_1} \!\dd s_2 
    \cdots 
    \int_0^{s_{k}} \!\dd s_{k+1}
    &
    \text{for $\ell = k$,}
\end{dcases}
\end{equation*}
the integration along $s_{\ell+1}$ produces, respectively, the factors $(1 - s_2)$, $(s_{\ell} - s_{\ell +2})$, and $s_{k}$. \\
Let us perform the change of variables 
\begin{equation*}
\begin{dcases}
(s_2, s_3, \dots, s_{k+1}) \to (s_1, s_2, \dots, s_{k})
&
\text{for $\ell=0$,}
\\
(s_1, s_2, \dots, s_\ell, s_{\ell+2}, \dots, s_{k+1}) \to (s_1, s_2, \dots, s_\ell, s_{\ell+1}, \dots, s_{k})
&
\text{for $0 <\ell < k$.}
\end{dcases}
\end{equation*}
There is no need to change the variables for $\ell=k$. Then, for all $0 \leq \ell \leq k$, the remaining integration is over $\Delta_k$ and all the brackets in \eqref{eq-integrands-ell} coincide (with the case $\ell=k$). This implies that the RHS of \eqref{eq-sum-Ialpha+1k+1} can be recombined, through the factorization of this common bracket, as a single integral
\begin{align*}
\int_{\Delta_k} \dd s \,
\big[ (s_0 - s_1) r_0 + \cdots + (s_{\ell} - s_{\ell +1}) r_{\ell} + \cdots + (s_k-s_{k+1}) r_k \big]\,
\big[ (s_0 - s_1) r_0 + \cdots + (s_k -s_{k+1}) r_k \big]^{\,-(\alpha+1)},
\end{align*}
which is nothing else but $I_{\alpha,k}(r_0, \dots, r_k)$.

ii) For any $\ell$, $1 \leq \ell \leq k$,  one has 
$\int_{\Delta_k} \! \dd s = 
\int_0^1 \dd s_1  
\int_0^{s_1} \dd s_2 
\cdots 
\int_0^{s_{\ell-2}} \dd s_{\ell - 1} 
\int_0^{s_{\ell - 1}} \dd s_{\ell + 1} 
\int_0^{s_{\ell + 1}} \dd s_{\ell + 2} 
\cdots 
\int_0^{s_{k - 1}} \dd s_k
\int_{s_{\ell + 1}}^{s_{\ell - 1}} \dd s_\ell$. \\
Performing the last integration along $s_\ell$, one gets
\begin{multline*}
\int_{s_{\ell + 1}}^{s_{\ell - 1}}\! \dd s_\ell \,
[r_0 + s_1 (r_1 -  r_0) + \dots + s_k (r_k -  r_{k-1})]^{\,-\alpha}
\\
\begin{aligned}
& =
\tfrac{1}{1 - \alpha} \tfrac{1}{r_{\ell} - r_{\ell-1}} \big[ [r_0 + s_1 (r_1 -  r_0) + \cdots + s_k (r_k -  r_{k-1})]^{1-\alpha} \big]_{\,s_\ell \,=\, s_{\ell + 1}}^{\,s_\ell \,= \,s_{\ell - 1}}
\\
&
=
\tfrac{1}{1 - \alpha} \tfrac{1}{r_{\ell} - r_{\ell-1}} [r_0 + \cdots + s_{\ell - 1}(r_{\ell} - r_{\ell - 2}) + s_{\ell + 1}(r_{\ell + 1} - r_{\ell}) + \cdots + s_k (r_k -  r_{k-1})]^{1-\alpha}
\\
&\hspace{0.5cm}
- \tfrac{1}{1 - \alpha} \tfrac{1}{r_{\ell} - r_{\ell-1}} [r_0 + \cdots + s_{\ell - 1}(r_{\ell - 1} - r_{\ell - 2}) + s_{\ell + 1}(r_{\ell + 1} - r_{\ell - 1}) + \cdots + s_k (r_k -  r_{k-1})]^{1-\alpha}.
\end{aligned}
\end{multline*}
An integration along the other variables $s_i$'s gives the equation \eqref{eq-recursive-I-ell}.
\end{proof}

\subsection{\texorpdfstring{Definitions and properties of the operators $X_{\alpha,k}$}{Definitions and properties of the operators Xalpha}}

As before, $h$ is a positive invertible element in $\calA=M_N(\bbC)$.
\begin{definition}
For any $k\in \bbN$, given a function $f:\,(\bbR_+^*)^{\,k+1} \to \bbC$, we define the operator $\pi_f$ acting on $\calA^{\otimes^{k}}$ as
\begin{align}
\label{eq-def Pi f}
\pi_f \vc f(r_0, \dots, r_k) \, \,\mul\circ (E_0^R \otimes \cdots \otimes E_k^R) \in \calB(\calA^{\otimes^{k}}, \calA), 
\end{align}
with an (implicit) summation over $k+1$-tuples $(r_0, \ldots, r_k)$ of spectral values of $h$. 
\end{definition}
In particular, using \eqref{eq- action of tensors} for $\lambda\in \bbC,\,b_\ell \in \calA$,
\begin{align}
&\pi_f[\lambda]= f(r_0)\lambda E_0=\lambda f(h)\in \calA \, \text{ for } k=0, \label{eq-def I k=0} \\
&\pi_f[ b_1 \otimes \cdots \otimes b_k] 
= f(r_0, \dots, r_k) \, E_0 b_1 E_1 \cdots E_{k-1} b_k E_k \in \calA \, \text{ for } k\in \bbN^*. \label{eq-def Pi f 1}
\end{align}
If $\widehat{f}(r_0,\dots,r_k)\vc\int_Z \dd z \,f(z;r_0,\dots r_k)$, where $Z$ is a measured space, then
\begin{align}
\label{eq-pi integral}
\pi_{\widehat{f}}[b_1\otimes \cdots \otimes b_k]=\int_Z \dd z\, \pi_{f(z)}[b_1\otimes \cdots \otimes b_k]
\end{align}
since the implicit summation over $r_0,\dots ,r_k$ is finite and $E_0 b_1 E_1 \cdots E_{k-1} b_k E_k$ is $z$-independent.

\bigskip
The spectral function $f$ in the equation \eqref{eq-def Pi f 1} has a peculiar modification if one of the variables $b_\ell$ is a function of $h$:

\begin{lemma}
For any continuous function $F:\bbR_+^* \to \bbC$ and $1\leq \ell \leq k$, we have
\begin{align}
&\pi_f [ b_1 \otimes \cdots \otimes b_{\ell-1} \otimes F(h)\otimes b_{\ell+1}\otimes\cdots\otimes b_k]  
= \pi_{\tilde{f}} [ b_1 \otimes \cdots \otimes b_{\ell-1} \otimes b_{\ell+1}\otimes \cdots\otimes b_k], \label{lem bi=h}
\\
& \text{with }\tilde{f}(r_0,\dots,r_{k-1}) \vc F(r_{\ell-1})\,f(r_0,\dots,r_{\ell-1},r_{\ell-1},r_\ell,\dots,r_{k-1}). \nonumber
\end{align}
\end{lemma}

\begin{proof}
Since $E_{\ell-1}F(h) E_{\ell}=F(r_{\ell-1})E_{\ell-1}E_\ell = \delta_{r_\ell,r_{\ell-1}} F(r_{\ell-1}) E_{\ell-1}$, we get
\begin{align*}
\pi_f [ b_1 \otimes \cdots& \otimes b_{\ell-1} \otimes F(h)\otimes b_{\ell+1}\otimes\cdots\otimes b_k]  \\
& =  f(r_0, \dots, r_{\ell-1}, r_\ell, r_{\ell+1}, \dots, r_k) \,E_0 b_1E_1\cdots b_{\ell-1} E_{\ell-1} F(h)E_\ell b_{\ell+1}E_{\ell+1}\cdots  b_k E_k\\
& = F(r_{\ell-1})f(r_0,\dots,r_{\ell-1},r_{\ell-1},\dots,r_{k}) \, E_0 b_1E_1 \cdots b_{\ell-1} E_{\ell-1} b_{\ell+1}E_{\ell+1} \cdots  b_k E_k,
\end{align*}
which, after a relabeling of the summation indices $i \to i-1$ for $i>\ell$, can be written as in \eqref{lem bi=h}.
\end{proof}
Since this result will be widely used, let us give an example: for $a,b\in \calA$ and $f : (\bbR_+^*)^{5}\to \bbC$
\begin{align*}
\pi_f[h^2\otimes a\otimes h^{1/2}\otimes b]=r_0^2 r_1^{1/2}f(r_0,r_0,r_1,r_1,r_2) \, E_0 aE_1bE_2.
\end{align*}

An important case of operators $\pi_f$ is the family of operators $X_{\alpha,k}$ associated to the universal functions $I_{\alpha,k}$ (see Definition~\ref{Def of I}) and to $h$, which will play a crucial role in the sequel precisely because of their universality.
\begin{equation}
\label{eq-def-Xalphak}
X_{\alpha,k} \vc \pi_{I_{\alpha,k}} \text{ for } k\in \bbN.
\end{equation}
Again, for brevity of notation on the use of $X_{\alpha,k}$, both the $h$-dependence and the summation when applied to arguments will be implicit.

From the equations \eqref{eq-I alpha k (r0, ...,r0)} and \eqref{def I0} we immediately check that
\vspace{-0.2cm}
\begin{align}
\label{eq-X on only h}
X_{\alpha,k}[\,\underbrace{h \otimes \cdots \otimes h}_{k}\,] = \tfrac{1}{k!}\, h^{-\alpha}=X_{\alpha,0}[h].
\end{align}
We also remark that for any matrices $\,b_i,\, c$ in $\calA$ such that $[c,h]=0$, we have the two following factorizations:
\begin{align}
&X_{\alpha,k}[b_1 \otimes \cdots \otimes  b_{\ell-1} \otimes b_\ell\,c \otimes b_{\ell+1}\otimes\cdots \otimes b_{k}] 
=  X_{\alpha,k}[b_1 \otimes \cdots \otimes  b_{k}] \,c, \quad \text{if } [c,b_i]=0\text{ for }\ell +1\leq i \leq k, \label{eq-X with au}  \\
&X_{\alpha,k}[b_1\otimes \cdots \otimes b_{\ell-1} \otimes c\,b_\ell \otimes b_{\ell+1}\otimes\cdots \otimes b_{k}]
= c \,X_{\alpha,k}[b_1\otimes \cdots \ \otimes b_{k}], \quad \text{if } [c,b_i]=0\text{ for } 1\leq i \leq \ell-1. \label{eq-X with ua}
\end{align}

For $a\in \calA$ and $\ell = 1, \dots, k$, let $i_{a}^{(\ell)} : \calB(\calA^{\otimes^{k}}, \calA) \to \calB(\calA^{\otimes^{k-1}}, \calA)$ be defined by  
\begin{align}
\label{def-il}
(i_{a}^{(\ell)} \bbB)[b_1 \otimes \cdots \otimes b_{k-1}] \vc \bbB[b_1\otimes \cdots \otimes b_{\ell-1} \otimes a \otimes b_{\ell} \otimes \cdots \otimes b_{k-1}],
\end{align}
which inserts $a$ at the $\ell$-th place in $\bbB$. For instance, one easily checks that
\begin{equation*}
i_{a}^{(\ell)} \, \mul \circ (a_0^R \otimes a_1^R \otimes \cdots \otimes a_{k}^R) 
= 
\mul \circ (a_0^R \otimes a_1^R \otimes \cdots \otimes \left[ a_{\ell-1} a a_{\ell} \right]^R \otimes \cdots \otimes a_{k}^R) \in \calB(\calA^{\otimes^{k-1}}, \calA).
\end{equation*}

\begin{lemma}
\label{lem-Xu}
The operators $X_{\alpha,k}$ satisfy $X_{\alpha,k} = \sum_{\ell = 1}^{k+1}  i_{h}^{(\ell)} X_{\alpha+1,k+1}$.\\
More explicitly, the following expansion holds true for any $b_\ell \in \calA$:
\begin{multline}
\label{eq-Xu}
X_{\alpha, k}[b_1 \otimes \cdots \otimes b_k] = 
\\
\!X_{\alpha+1, k+1}[h \otimes b_1 \otimes \cdots \otimes b_k] 
\!+\! X_{\alpha+1, k+1}[b_1 \otimes h \otimes \cdots \otimes b_k]
\!+ \cdots 
+ \!X_{\alpha+1, k+1}[b_1 \otimes \cdots \otimes b_k \otimes h] .
\end{multline}
\end{lemma}

\begin{proof}
This follows from equation \eqref{eq-sum-Ialpha+1k+1}.
\end{proof}

\begin{corollary}
\label{cor-Xan}
For any $a, b\in \calA$ and any $n \geq 1$, one has
\begin{align}
X_{\alpha, 1}[a] &=
n! \sum_{\substack{\ell_1, \ell_2 \geq 0\\ \ell_1 + \ell_2 = n}}
X_{\alpha, n+1}[ \underbrace{h \otimes \cdots \otimes h}_{\ell_1 \text{ times}}\otimes\, a \otimes \underbrace{h \otimes \cdots \otimes h}_{\ell_2 \text{ times}}], \label{eq:Xalpha1}
\\
X_{\alpha, 2}[a \otimes b] &=
n! \sum_{\substack{\ell_1, \ell_2, \ell_3 \geq 0\\ \ell_1 + \ell_2 + \ell_3 = n}}
X_{\alpha, n+1}[ \underbrace{h \otimes \cdots \otimes h}_{\ell_1 \text{ times}} \otimes\, a \otimes \underbrace{h \otimes \cdots \otimes h}_{\ell_2 \text{ times}} \otimes b \otimes \underbrace{h \otimes \cdots \otimes h}_{\ell_3 \text{ times}}]. \label{eq:Xalpha2}
\end{align}
\end{corollary}

\begin{proof}
We prove the first relation \eqref{eq:Xalpha1} by induction. When $n=1$, the use of Lemma~\ref{lem-Xu} yields to the desired relation: $X_{\alpha, 1}[a] = X_{\alpha, 2}[h \otimes a] + X_{\alpha, 2}[a \otimes h]$. Assuming the relation holds for $n \geq 1$, then
\begin{align*}
X_{\alpha, 1}[a] &=
n! \sum_{\substack{\ell_1, \ell_2 \geq 0\\ \ell_1 + \ell_2 = n}}
X_{\alpha, n+1}[ \underbrace{h \otimes \cdots \otimes h}_{\ell_1 \text{ times}} \otimes a \otimes \underbrace{h \otimes \cdots \otimes h}_{\ell_2 \text{ times}}],
\\
&=
n! \sum_{\substack{\ell_1, \ell_2 \geq 0\\ \ell_1 + \ell_2 = n}} \big[
(\ell_1 + 1) X_{\alpha, n+2}[ \underbrace{h \otimes \cdots \otimes h}_{\ell_1 + 1 \text{ times}} \otimes\, a \otimes \underbrace{h \otimes \cdots \otimes h}_{\ell_2 \text{ times}}]
+ (\ell_2 + 1) X_{\alpha, n+2}[ \underbrace{h \otimes \cdots \otimes h}_{\ell_1 \text{ times}} \otimes \,a \otimes \underbrace{h \otimes \cdots \otimes h}_{\ell_2 + 1 \text{ times}}]
\big]
\\
&=
 n! \sum_{\substack{\ell_1 \geq 1, \ell_2 \geq 0\\ \ell_1 + \ell_2 = n+1}}
\ell_1 X_{\alpha, n+2}[ \underbrace{h \otimes \cdots \otimes h}_{\ell_1 \text{ times}} \otimes\, a \otimes \underbrace{h \otimes \cdots \otimes h}_{\ell_2 \text{ times}}]
 + n! \sum_{\substack{\ell_1 \geq 0, \ell_2 \geq 1\\ \ell_1 + \ell_2 = n+1}}
\ell_2 X_{\alpha, n+2}[ \underbrace{h \otimes \cdots \otimes h}_{\ell_1 \text{ times}} \otimes \,a \otimes \underbrace{h \otimes \cdots \otimes }_{\ell_2 \text{ times}}]
\end{align*}
after changes of summation parameters $\ell_1 + 1 \to \ell_1$ and $\ell_2 + 1 \to \ell_2$ in the two sums. \\
Then, extending the summation ranges with $\ell_1 = 0$ and $\ell_2 = 0$ since they do not contribute, one gets
\begin{align*}
X_{\alpha, 1}[a] &=
 n! \sum_{\substack{\ell_1, \ell_2 \geq 0\\ \ell_1 + \ell_2 = n+1}}
(\ell_1 + \ell_2) X_{\alpha, n+2}[ \underbrace{h \otimes \cdots \otimes h}_{\ell_1 \text{ times}} \otimes \,a \otimes \underbrace{h \otimes \cdots \otimes h}_{\ell_2 \text{ times}}]
\\
&= (n+1)! \sum_{\substack{\ell_1, \ell_2 \geq 0\\ \ell_1 + \ell_2 = n+1}}
X_{\alpha, n+2}[ \underbrace{h \otimes \cdots \otimes h}_{\ell_1 \text{ times}} \otimes \,a \otimes \underbrace{h \otimes \cdots \otimes h}_{\ell_2 \text{ times}}],
\end{align*}
and \eqref{eq:Xalpha1} holds true. The relation \eqref{eq:Xalpha1} is proved similarly.
\end{proof}

We will also use the notion of {\it expansion}: for $ k \in \bbN^*,\,\calE_k:\,\calA^{\otimes ^k} \to\calA^{\otimes ^{k+1}}$ is defined for $b_\ell\in \calA$ by
\begin{equation}
\label{def-expansion}
\calE_k[b_1\otimes \cdots \otimes b_k] \vc h \otimes b_1 \otimes \cdots \otimes b_k +b_1 \otimes h \otimes \cdots \otimes b_k+ \cdots 
+ b_1 \otimes \cdots \otimes b_k \otimes h,
\end{equation}
and the previous lemma can be read as:
\begin{equation}
\label{eq-reduction}
X_{\alpha+1,k+1} \circ \calE_k=X_{\alpha,k}, \quad \forall k \in \bbN^*
\end{equation}
or seen as {\it a reduction process} after the expansion $\calE_k$.

\begin{lemma}
Assume $\alpha \neq 1$.\\
i) For any $b\in \calA$,
\begin{align}
X_{\alpha,1}\big[ [b,h]\big]= \tfrac{1}{1-\alpha}  \,[b,h^{1-\alpha}].   \label{eq-X with commutator3}
\end{align}
ii) For any $2\leq k,\ell \in \bbN$, $ \ell<k$, and $b_i\in \calA$, 
\begin{align}
\label{eq-X with commutator1}
X_{\alpha,k}\big[b_1\otimes \cdots\otimes b_{\ell-1}\otimes [b_\ell,h]\otimes b_{\ell+1}\otimes  \cdots \otimes b_k\big]
&=
\begin{multlined}[t]
\tfrac{1}{1-\alpha} \big(+X_{\alpha-1,k-1}[b_1\otimes\cdots\otimes b_{\ell-2}\otimes b_{\ell-1}b_\ell\otimes b_{\ell+1}\otimes \cdots \otimes b_k]
\\
- X_{\alpha-1,k-1}[b_1\otimes \cdots \otimes b_{\ell-1}\otimes b_\ell b_{\ell+1} \otimes \otimes b_{\ell+2}\cdots \otimes b_k] \big).
\end{multlined}
\end{align}
iii) For any $k \geq 2$ and any $b_i \in \calA$,
\begin{align}
& X_{\alpha,k}\big[b_1\otimes \cdots\otimes b_{k-1}\otimes [b_k,h]\big] =    
\tfrac{1}{1-\alpha} \big(+X_{\alpha-1,k-1}[b_1\otimes\cdots\otimes b_{k-2}\otimes b_{k-1}b_k]  - X_{\alpha-1,k-1}[b_1\otimes \cdots \otimes b_{k-1}] \,b_{k}\big).  \label{eq-X with commutator2}
\end{align}
\end{lemma}

\begin{proof}
i) Thanks to the definition \eqref{eq-def-Xalphak}, \eqref{eq-recursive-I-ell} and \eqref{eq def I1}, we have 
\begin{align*}
X_{\alpha,1}\big[ [b,h]\big] &=  I_{\alpha,1}(r_0,r_1)\,E_0(bh-hb)E_1=(r_1-r_0)\,\tfrac{r_1^{1-\alpha}\,-\,r_0^{1-\alpha}}{(1-\alpha)(r_1\,-\,r_0)}\,E_0bE_1=\tfrac{1}{1-\alpha}(bh^{1-\alpha}-h^{1-\alpha}b)
\end{align*}
because the implicit summation over $r_0$ (resp. $r_1$) of $E_0$ (resp. $E_1$) gives $\bbbone$.

ii) The LHS of \eqref{eq-X with commutator1} is equal to
\begin{align*}
& ( r_\ell-r_{\ell-1})\, I_{\alpha,k}(r_0,\dots, r_k) \,E_0b_1\cdots E_{\ell-2} b_{\ell-1}E_{\ell-1}b_\ell E_{\ell}\cdots E_{k-1}b_kEk \\
&\hspace{0.5cm}= \tfrac{1}{1 - \alpha} \big[
I_{\alpha-1,k-1}(r_0, \ldots, \widehat{r_{\ell - 1}}, \ldots, r_k)\,E_0b_1\cdots  b_{\ell-1}E_{\ell-1}b_\ell E_{\ell}b_{\ell+1}\cdots E_{k-1}b_kEk\\
&\hspace{2cm}
-I_{\alpha-1,k-1}(r_0, \ldots, \widehat{r_{\ell}}, \ldots, r_k)\,E_0b_1\cdots b_{\ell-1}E_{\ell-1}b_\ell E_{\ell}b_{\ell+1}\cdots E_{k-1}b_kEk
\big] \\
& \hspace{0.5cm}=\tfrac{1}{1 - \alpha} \big(X_{\alpha-1,k-1}[b_1\otimes\cdots\otimes b_{\ell-2}\otimes b_{\ell-1}b_\ell\otimes b_{\ell+1}\otimes \cdots \otimes b_k]
\\
& \hspace{2cm}- X_{\alpha-1,k-1}[b_1\otimes \cdots \otimes b_{\ell-1}\otimes b_\ell b_{\ell+1} \otimes \otimes b_{\ell+2}\cdots \otimes b_k] \big)
\end{align*}
because the missing summations in $r_{\ell-1}$ and $r_{\ell}$ implies that $E_{\ell-1}$ and $E_\ell$ are replaced by $\bbbone$. \\
iii) Similarly, the LHS of \eqref{eq-X with commutator2} is equal to
\begin{align*}
&\!\! (r_k-r_{k-1})\, I_{\alpha,k}(r_0,\dots, r_k) \,E_0b_1\cdots  b_{k-1}E_{k-1}b_k Ek 
\\
&\hspace{0.5cm}= \tfrac{1}{1 - \alpha} \big[
I_{\alpha-1,k-1}(r_0, \ldots, \widehat{r_{k- 1}}, \ldots, r_k)\,E_0b_1\cdots  b_{k-1}E_{k-1}b_kE_k
-I_{\alpha-1,k-1}(r_0, \ldots, \widehat{r_k}, \ldots, r_k)\,E_0b_1\cdots b_{k-1} E_{k-1}b_kE_k
\big] \\
& \hspace{0.5cm}=\tfrac{1}{1 - \alpha} \big(X_{\alpha-1,k-1}[b_1\otimes\cdots\otimes b_{k-1}b_k] 
- X_{\alpha-1,k-1}[b_1\otimes \cdots \otimes b_{k-1}]b_k \big).
\end{align*}
\end{proof}

As a consequence of the previous lemma, we get the following:
\begin{corollary}
\label{cor-(inverse h)ah}
For any $\alpha \neq 1,\,2\leq k,\ell \in \bbN$, $\ell<k$, and $b_i\in \calA$,
\begin{align*}
\begin{multlined}[t]
X_{\alpha,k}[b_1\otimes \cdots\otimes b_{\ell-2}\otimes h\otimes h^{-1}a h\otimes h\otimes b_{\ell+2}\otimes \cdots \otimes b_k]
-X_{\alpha,k}[b_1\otimes \cdots\otimes b_{\ell-2}\otimes h\otimes a\otimes h\otimes b_{\ell+2}\otimes \cdots \otimes b_k]
\\
= \tfrac{1}{1-\alpha} \big(X_{\alpha-1,k-1}[b_1\otimes\cdots\otimes b_{\ell-2}\otimes a\otimes h\otimes b_{\ell+2}\otimes\cdots \otimes b_k] 
- X_{\alpha-1,k-1}[b_1\otimes \cdots \otimes b_{\ell-2}\otimes h\otimes h^{-1}ah \otimes b_{\ell+2}\cdots \otimes b_k] \big).
\end{multlined}
\end{align*}
\end{corollary}
For instance, using also \eqref{eq-X with au} and \eqref{eq-X with ua},
\begin{align*}
X_{\alpha,3}[h\otimes h^{-1}a h\otimes h]-X_{\alpha,3}[h\otimes a\otimes h]
& =\tfrac{1}{1-\alpha}\, (X_{\alpha-1,2}[a \otimes h] - X_{\alpha-1,2}[ h \otimes h^{-1}a h])
\\
& 
=\tfrac{1}{1-\alpha} \,(X_{\alpha-1,2} [a \otimes h] - h^{-1} X_{\alpha-1,2}[h \otimes a]\, h).
\end{align*}
For a potential use of this corollary, see Remark~\ref{rem-N zero or not}.

\subsection{\texorpdfstring{If $h$ commutes with the $b_\ell$'s}{If h commutes with the b }}

When $h$ commutes with the arguments acted upon by the operator $X_{\alpha,k}$, the action of the latter is quite simple:

\begin{lemma}
Let $\alpha \in \bbR$, $k\in \bbN$ and $b_\ell \in \calA$ with $1\leq \ell \leq k$. If  $[h,b_\ell]=0$ for any $\ell$, then
\begin{align}
\label{eq-Xk for u diagonal}
X_{\alpha,k}[b_1 \otimes \cdots\otimes b_k] =\tfrac{1}{k!} \,h^{-\alpha} b_1\cdots b_k.
\end{align}
\end{lemma}
This is for instance the case either when $h=\bbbone$ or when $h$ and the $b_\ell$'s are diagonal matrices.

\begin{proof}
We have, using the equality \eqref{eq-I alpha k (r0, ...,r0)},
\begin{align*}
X_{\alpha,k}[b_1 \otimes \cdots\otimes b_k] &=I_{\alpha,k}(r_0,\dots, r_k) \,E_0b_1E_1 \cdots b_k E_k=I_{\alpha,k}(r_0,\dots,r_0)\, E_0b_1\cdots b_k  =\tfrac{1}{k!} \,h^{-\alpha} b_1\cdots b_k.
\end{align*}
\end{proof}
This shows that in this situation the operators $X_{\alpha,k}$ act as a polynomial in $h$ and the $b_\ell$'s.

\subsection{Action of a derivation and finite differences}

It is immediate to extend all previous definitions and results to the algebra that we consider from now on, namely
\begin{equation*}
\calA \vc C^\infty(\calU,M_N(\bbC)),
\end{equation*}
where $\calU$ is a parameter space. Later on, when $P$ will play a role, $h$ will be either $u(x)$ with $\calU=U$ (an open set in $M$) or $H(x,\xi)=g^{\mu\nu}(x)u(x)\xi_\mu\xi_\nu$ and $\calU=U\times \bbR^d$. This extension is necessary because we have derivations in the play and consequently, the operators $X_{\alpha,k}$ now depend on $x\in M$. With the definitions
\begin{align}
&f_{k,\ell}(s;r) \vc e^{-(s_\ell-s_{\ell+1})\, r}, \label{eq-fkl}\\
&\widetilde{I}_k(r_0,\dots,r_k) \vc \int_{\Delta_k} \dd s\, f_{k,0}(s;r_0)\cdots f_{k,k}(s;r_k),\label{I tilde}
\end{align}
where $s\in \Delta_k,\, r\in \bbR_+^*$, we can rewrite the functions $I_{\alpha,k}$ as
\begin{align}
&I_{\alpha,k}(r_0, \ldots, r_k)
=\Gamma(\alpha)^{-1} \, \int_{\Delta_k} \!\! \dd s\, \int_0^\infty \! \dd t \, t^{\alpha - 1} e^{-t\, [\sum_{\ell=0}^k \,(s_\ell-s_{\ell+1})\,r_\ell ]}
= \Gamma(\alpha)^{-1} \, \int_0^\infty \!\dd t \, t^{\alpha - 1}\, \widetilde{I}_k(tr_0,\dots,tr_k) \label{eq- I Laplace}
\end{align}
and the operator $f_k(\xi)$ defined in \eqref{eq-fk-def} and restricted to $\calA^{\otimes^k}$ (\textsl{i.e.} to arguments without derivatives),  is associated for $h=H(x,\xi)$ to the spectral function $\widetilde{I}_k$:
\begin{align}
\label{eq-fk as a pi}
f_k(\xi) = \pi_{\widetilde{I}_k}.
\end{align}

\smallskip
Let $\partial$ be an arbitrary derivation of the algebra $\calA$, namely a linear combination of a derivative of a $\calA$-valued function along a parameter $y\in \mathcal{U}$ and a commutator with an element of $\calA$. Then
\begin{align}
\label{eq- partial of exp}
\partial e^{-t \,h} = - \int_0^t \dd s \, e^{(s-t)\,h} \, (\partial h) \, e^{-s\,h}.
\end{align}
Recall that a proof for $\partial=\partial_y$ is based on the following relation: If $E_\epsilon(s) \vc e^{(st-t)\,h(y+\epsilon)}\,e^{-st\,h(y)}$ for $s\in [0,1]$ and $\epsilon \in \bbR$, then 
\begin{align*}
e^{-t\,h(y+\epsilon)} - e^{-t\,h(y)}=-\int_0^1 \dd s\,E_\epsilon'(s) = -\int_0^1 t\dd s\, e^{(st-t)\,h(y+\epsilon)}\,\big(h(y+\epsilon)-h(y)\big)\,e^{-st\,h(y)},
\end{align*}
see also \cite{Avra15a}.
For an inner derivation like $\partial=[\cdot,\,a]$ where $a\in \calA$, a similar argument can be applied if one begins with $E(s)\vc e^{(st-t)\,h}\,a\,e^{-st\,h}$.\\
By functional calculus on $h=r_0E_0$, we deduce from \eqref{eq- partial of exp} 
\begin{align}
\label{eq-exp in C}
\partial e^{-t \,h} 
= - \int_0^t \dd s \,e^{(s-t)\,r_0} e^{-s \,r_1} \,E_0 (\partial h) E_1 
= \tfrac{ e^{-t \,r_1}\, -\, e^{-t\, r_0}}{r_1 - r_0}\, E_0 (\partial h) E_1 ,
\end{align}
(still an implicit summation over repeated indices) where from now on we use the symbolic notation $\tfrac{f(r_1)\, - f(r_0)}{r_1 - r_0}$ of finite differences instead of $ f'(r_0)$ when $r_0=r_1$. \\
Remark that the peculiar case $\partial_{\xi_\mu} e^{-t H(x,\xi)}= -t g^{\mu\nu}\xi_\nu \,u(x)\,e^{-t H(x,\xi)}$ is compatible with \eqref{eq-exp in C}.

The relation \eqref{eq-exp in C} can be extended in the following way. Let $g$ be a Laplace transform of a Borel signed $\bbR$-valued measure $\phi$ on $\bbR_+$, \textit{i.e.} $g: \, r\in \bbR_+^* \to \int_0^\infty \dd\phi( t)\,e^{-tr}\in \bbR$. We consider the derivability of $g$ at the point $r_0\in ]r_m,r_M[$: since for any $r\in]r_m,r_M[$ and $0<\epsilon<r_m$, $\vert \partial_r (te^{-tr})\vert<\epsilon^{-1}  e^{t\epsilon}e^{-tr_m}$ and $\int_0^\infty \dd \phi(t)\,e^{-t(r_m-\epsilon)}=g(r_m-\epsilon)<\infty$, we may use the differentiation lemma for parameter dependent integrals to commute $\partial_r$ and the integral (recall also that by Bernstein's theorem, $g$ is completely monotonic, see \cite[Theorem 1.4]{Schilling}). This commutation property can be first extended when $g$ is a Laplace transform of a signed $\bbC$-valued measure $\phi$ on $\bbR_+$ and then extended again to $f=g \circ r$ for any smooth function $r: y\in \calU \to \bbR_+^*$: consider again the derivability of $f$ in an open ball $B$ around $y_0$ such that $r(y_0) \in r(B)\subset ]r_m,r_M[$ where $r_m=\inf_{y\in \overline{B}}\,r(y)$ and $r_M=\sup_{y\in \overline{B}}\,r(y)$; as before we get $\vert \partial_y e^{-tr(y)} \vert =\vert r'(y) \vert\,te^{-tr(y)}<\sup_{y\in \overline{B}} \vert r'(y) \vert \,\epsilon^{-1} e^{-t(r_m-\epsilon)}$ so that we can commute $\partial_y$ with the integral. In particular, for such $f$, 
\begin{align}
\label{eq-partial(f(h))}
\partial [f(h)]=\int_0^\infty \dd\phi(t)\,\partial (e^{-t\, r_0E_0}) = \tfrac{1}{r_1 - r_0} \,[\int_0^\infty \dd\phi(t)\, (e^{-t r_1} - e^{-t r_0} )] \,E_0 (\partial h) E_1= \tfrac{f(r_1)\, -\, f(r_0)}{r_1\, -\, r_0}\, E_0 (\partial h) E_1.
\end{align}
It is then natural to define the set of functions
\begin{align}
\label{eq-partial-1}
\calC \vc \{ f \in C^\infty(\bbR_+^*,\bbC)\,\,\vert\,\, \partial[f(h)]=\tfrac{f(r_1) \,-\, f(r_0)}{r_1\, -\, r_0}\, E_0 (\partial h) E_1\}
\end{align}
which is large because $\calC$ contains all functions of type $g\circ r$ with $g,\, r$ defined as above. \\
Since $r^{-\alpha} =  \tfrac{1}{ \Gamma(\alpha)}\int_0^\infty  \,\dd t\,t^{\alpha - 1}\,e^{-t r}$ for $\alpha>0$, the definition \eqref{def I0} shows in particular that $I_{\alpha,0} \in \mathcal{C}$ if $\alpha >0$ and since $I_{0,0} = 1$,
\begin{align}
\label{eq-I0 in C0}
I_{\alpha,0} \in \mathcal{C}, \quad \forall \alpha \geq 0.
\end{align}
With the help of \eqref{eq- I Laplace}, the functions
\begin{align}
\label{eq-I in C}
I_{\alpha,k}(r_0, \ldots, r_k)
=\Gamma(\alpha)^{-1} \, \int_{\Delta_k} \dd s\, \int_0^\infty \dd t \, t^{\alpha - 1} f_{k,0}(s;t r_0)\cdots f_{k,k}(s;tr_k)
\end{align}
for $\alpha>0$ are nothing else but integrals of products of elements in $\calC$.

Then, let us introduce the following generalization of the finite difference appearing in the RHS of \eqref{eq-partial(f(h))} to functions of several variables: for $f\in C^{\infty}((\bbR_{+}^*)^{\,k+1}, \bbC )$ and $k ,\,\ell \in \bbN$ with $0\leq \ell \leq k$, define the “partial” finite differences
\begin{align}
\label{def-Delta f}
(\Delta^{(\ell)} f)(r_0, \dots, r_{k+1}) \vc 
\begin{dcases}
 \tfrac{1}{r_{\ell+1} - r_\ell}\,\big(f(r_0, \dots, \widehat{r_\ell}, \dots, r_{k+1}) - f(r_0, \dots, \widehat{r_{\ell+1}}, \dots, r_{k+1})\big), &\text{ when }r_\ell  \neq r_{\ell+1},\\
\partial_{r_\ell}f(r_0,\cdots,r_\ell,r_\ell,\cdots,r_k),& \text{ when }r_\ell =r_{\ell+1}.
\end{dcases}
\end{align}
We can generalize \eqref{eq-partial(f(h))} to operators $\pi_f$ by defining:
\begin{align}
\label{def-Delta pi}
\Delta^{(\ell)}_{\partial} \pi_f \vc i_{\partial h}^{(\ell+1)} \,\pi_{\Delta^{(\ell)} f}, 
\end{align}
or more explicitly, $\Delta^{(\ell)}_{\partial} \pi_f = (\Delta^{(\ell)} f)(r_0, \dots, r_{k+1})  \,\mul\circ (E_0^R \otimes \cdots \otimes [E_{\ell} (\partial h) E_{\ell+1}]^R \otimes \cdots \otimes E_{k+1}^R)$.
The next proposition shows that the families of $I_{\alpha,k}$ and $X_{\alpha,k}$ are indeed invariant respectively by $\Delta^{(\ell)}$ and $\Delta^{(\ell)}_{\partial}$ modulo dilations and insertion of $\partial h$:

\begin{proposition}
For any $\alpha \neq 1$, any $k,\ell \in \bbN^*$ with $0 \leq \ell \leq k$, we have
\begin{align}
\label{eq-Delta-Ialphak}
&
\Delta^{(\ell)} I_{\alpha,k} = - \alpha\,I_{\alpha+1,k+1},
\\
\label{eq-Delta-Xalphak}
&
\Delta^{(\ell)}_{\partial}  X_{\alpha,k} = - \alpha \, i_{\partial h}^{(\ell+1)} \, X_{\alpha+1,k+1},
\\
\label{eq-Delta-tildeIalphak}
&
\Delta^{(\ell)} \widetilde{I}_k = - \widetilde{I}_{k+1},
\\
\label{eq-Delta-fkxi}
&
\Delta^{(\ell)}_{\partial}  f_k(\xi) = - i_{\partial h}^{(\ell+1)} \, f_{k+1}(\xi) \quad\text{with $h=H(x,\xi)$},
\end{align}
\end{proposition}

\begin{proof}
The first relation follows from \eqref{eq-recursive-I-ell} applied to $I_{\alpha+1,k+1}$ with $\ell +1$ instead of $\ell$ and the second one from the definition \eqref{eq-def-Xalphak}. The third one can be shown using a straightforward adaptation of the proof of  Lemma~\ref{lem-relations-Ialpha} ii) beginning with $\widetilde{I}_{k+1}(r_0, \dots, r_{k+1}) = \int_{\Delta_{k+1}} \!\dd s \, e^{- \sum_{\ell=0}^{k+1} s_\ell (r_\ell -r_{\ell-1}) }$. The last relation is a consequence of \eqref{eq-fk as a pi} and \eqref{eq-Delta-tildeIalphak}.
\end{proof}

\subsection{Propagation of derivations}

Let $\partial$ be a derivation acting on elements in $\calA$, for instance along a parameter in $\calU$. Suppose that we have a representation of the algebra $\calA$ on a vector space $\calH$ on which $\partial$ is also defined (with the same notations) in such a way that the Leibniz rule holds: $\partial (a v) = (\partial a) v + a (\partial v)$ for any $a \in \calA$ and $v \in \calH$.

\begin{proposition}
\label{prop-Propagation}
Assume that the function $f\in C^\infty((\bbR_+^*)^{\,k+1},\bbC)$ is either $f(r_0,\dots,r_k)=f_0(r_0)\cdots f_k(r_k)$ with $f_\ell \in \calC$, or $f=\widetilde{I}_k$, or $f=I_{\alpha,k}$ with $\alpha \geq 0$. Then, for any $\ell$ such that $1\leq \ell\leq k$ and any $b_\ell$ which is a $\calA$-valued differential operator in $\partial$, the following propagation rule holds true:
\begin{multline}
\pi_f[b_1 \otimes \cdots \otimes b_{\ell} \partial \otimes \cdots \otimes b_k]\, v 
 = \sum_{i=\ell+1}^{k} \pi_f[b_1 \otimes \cdots \otimes b_{i-1} \otimes (\partial b_i) \otimes \cdots \otimes b_k]\, v
+ \sum_{i=\ell}^{k} (\Delta^{(i)}_{\partial} \pi_f)[b_1 \otimes \cdots \otimes b_k] v  
\\
+ \pi_f[b_1 \otimes \cdots \otimes b_{\ell} \otimes \cdots \otimes b_k] (\partial v).  \label{eq-Ypartial-vector-dvlpt}
\end{multline}
\end{proposition}

Before looking at a proof, remark first that, using \eqref{eq-Delta-Xalphak} for $X_{\alpha,k}=\pi_{I_{\alpha,k}}$, see \eqref{eq-def-Xalphak}, (resp. $f_k(\xi)=\pi_{\widetilde{I}_k}$, see \eqref{eq-fk as a pi}), the RHS of \eqref{eq-Ypartial-vector-dvlpt} is written only in terms of the operators $X_{\alpha,k}$ and $X_{\alpha+1,k+1}$ (resp. $f_k(\xi)$ and $f_{k+1}(\xi)$, using \eqref{eq-Delta-fkxi}). This is a key point in the method exhibited in Section~\ref{The method and its simplifications}.

\begin{proof}
First case: Since $\pi_f= \mul \circ \big((f_0(r_0) \,E_0)^R \otimes \cdots \otimes (f_k(r_k) \,E_k)^R \big)$, one gets
\begin{align*}
\pi_f[b_1 \otimes \cdots \otimes b_{\ell} \partial \otimes \cdots \otimes b_k]\,v
& = f_0(r_0)\, E_0 b_1\, f_1(r_1)\, E_1 \cdots b_\ell \, \partial \big( f_{\ell}(r_{\ell})\, E_{\ell} b_{\ell+1} \cdots b_k f_k(r_k) \,E_k \,v\big).
\end{align*}
Thus, in the parenthesis, $\partial$ acts either on the $b_i$'s  or on the $f_i(r_i) E_{i}$'s or on $v$. When it acts on the $b_i$'s, this reproduces the first sum in \eqref{eq-Ypartial-vector-dvlpt} while when it acts on $v$, it gives the last term.\\
For the action on the $f_i(r_i) \,E_{i}=f_i(h)$, we can apply \eqref{eq-partial(f(h))} and get the total contribution
\begin{align*}
& \sum_{i=\ell}^k \,f_0(r_0) E_0 b_1\cdots b_i \partial \big(f_i(r_i) E_i\big)b_{i+1} \cdots b_kf_k(r_k) E_k \,v
\\
&\quad\quad  =\sum_{i=\ell}^k \,f_0(r_0) E_0 b_1\cdots b_i \tfrac{f_i(r_i)\, -\, f_i(r_{j})}{r_{i} \,- \,r_{j}}  E_i (\partial h) E_{j}b_{i+1} \cdots b_kf_k(r_k) E_k \,v
\\
&\quad\quad = \sum_{i=\ell}^k\, f_0(r_0) \cdots  \tfrac{f_i(r_i) - f_i(r_{j})}{r_{i} - r_{j}} \cdots f_k(r_k) \,
E_0 b_1\cdots b_i  E_i (\partial h) E_{j} b_{i+1} \cdots b_k E_k\,v.
\end{align*}
Swapping $r_{j}$ to $r_{i +1}$ and $r_j$ to $r_{j+1}$ for $j > i$, each spectral function in the last sum is
\begin{align*}
f_0(r_0) \cdots  \tfrac{f_i(r_i)\, -\, f_i(r_{i +1})}{r_{i} \,- \,r_{i+1}} \cdots f_k(r_{k+1})
&= 
\tfrac{1}{r_{i} \,- \,r_{i+1}} 
\big[  \begin{multlined}[t]
f_0(r_0) \cdots f_{i-1}(r_{i-1}) f_i(r_i) f_{i+1}(r_{i+2}) \cdots f_k(r_{k+1})
\\
- f_0(r_0) \cdots f_{i-1}(r_{i-1}) f_{i}(r_{i+1})  f_{i+1}(r_{i+2}) \cdots f_k(r_{k+1})
\big]
\end{multlined}
\\
& = \tfrac{1}{r_{i+1}\, - \,r_i}\,[ f(r_0, \dots, \widehat{r_i}, \dots, r_{k+1}) - f(r_0, \dots, \widehat{r_{i+1}}, \dots, r_{k+1})] = \Delta^{(i)}f(r_0,\dots,r_{k+1})
\end{align*}
which contributes to the second term of the RHS of \eqref{eq-Ypartial-vector-dvlpt}.

Case $f=\widetilde{I}_k$: From \eqref{eq-fkl}, \eqref{I tilde} and using \eqref{eq-pi integral}, we can apply the previous argument under the integration over $\Delta_k$. The only point to take care of, is the commutation of $\Delta^{(i)}$ with the integral in the second term of the RHS of \eqref{eq-Ypartial-vector-dvlpt}. 
Since $\Delta_k$ is compact and the integrand is smooth in $s$ and $r_\ell$, this commutation occurs even at coincidence points (where partial derivatives arise, see \eqref{def-Delta f}).

Case $f=I_{\alpha,k}$: For $\alpha=0$, the result is direct and the second term of \eqref{eq-Ypartial-vector-dvlpt} vanishes since $I_{0,k}$ is constant. \\
Assume now $\alpha>0$. From \eqref{eq-I in C}, once again the same argument can be applied under the integral and we only need to prove the commutation with integral along $t\in[0,\infty[$. With
\begin{align*}
 f(t;s;r_0,\dots,r_k) \vc e^{-t \sum_{\ell=0}^k\,(s_\ell-s_{\ell+1})\,r_\ell} , \quad  g(r_0,\dots,r_k) \vc \int_0^\infty dt\, t^{\alpha-1} f(t;s;r_0,\dots,r_k)=\Gamma(\alpha)\, \big[\sum_{\ell=0}^k\, s_\ell \,(r_\ell-r_{\ell-1}) \big]^{\,-\alpha},
\end{align*}
we have to show that
\begin{align*}
\big(\Delta^{(i)} \int_0^\infty dt\,& t^{\alpha-1} f(t;s;\cdot)\big)\,(r_0,\dots,r_{k+1})=\tfrac{1}{r_{i+1}\,-\,r_i} \big(g(r_0,\dots,\widehat{r_i},\dots,r_k)-g(r_0,\dots,\widehat{r_{i+1}},\dots,r_k)\big)
\end{align*}
coincides with
\begin{align*}
\int_0^\infty \dd t\,t^{\alpha-1}\,\big(\Delta^{(i)} f(t;s;\cdot) \big)\,(r_0,\dots,r_{k+1})
= \tfrac{1}{r_{i+1}\,-\,r_{i}}\, \int_0^\infty \dd t\,t^{\alpha-1}\,
\big(
e^{-t\sum_{\ell=0;\ell\neq i}^k\, s_\ell \,(r_\ell-r_{\ell-1}) }
-e^{-t\sum_{\ell=0;\ell\neq i+1}^k\, s_\ell \,(r_\ell-r_{\ell-1}) }
\big).
\end{align*}
By linearity, this is true for $r_i\neq r_{i+1}$ but at coincidence points we still have to show that the $t$-integral commutes with $\partial_{r_i}$.
For $r_i\in ]r_{\min}=\min r_j,r_{\max}=\max r_j[$ we have, using 
$t(s_i-s_{i+1})\,e^{-t\,(s_i-s_{i+1})\,r_i} \leq \epsilon^{-1} e^{-t\,(s_i-s_{i+1})\,(r_{\min}-\epsilon)}$ when $0<\epsilon<r_{\min}$,
\begin{align*}
t^{\alpha-1}\,\abs{(\partial_{r_i} f)(t;s; r_0,\dots,r_k) }
= (s_i -s_{i+1})\, t^\alpha\,f(t;s; r_0,\dots,r_k) \leq \epsilon^{-1} t^{\alpha-1} f(t;s;r_0,\dots, r_{i-1},r_{\min}-\epsilon,r_{i+1},\dots,r_k)
\end{align*}
and the RHS is $t$-integrable uniformly along $r_i$ which secures the commutation of the integral with $\partial_{r_i}$.
\end{proof}

\section{Total covariant derivative and normal coordinates}
\label{Total covariant derivative and normal coordinates}

In this section, we come back to the differential operator $P$ defined in \eqref{eq-def-P-upq} but we do not use Section~\ref{Section def of X}. Firstly, we rewrite $P$ in terms of a total covariant derivative, and since we know that the coefficients $\calR_r$ are invariant under a change of coordinates, we secondly gather the computation of several derivatives within a normal coordinate system.

\subsection{Total covariant derivative}

We need the total covariant derivative $\hnabla$, which combines the (gauge) connection $\nabla$ on $V$ with the Levi-Civita covariant derivative $\gnabla$ induced by the metric $g$. To avoid a definition of $\hnabla$ on the tensor products of $V$, $TM$ and $T^* M$ via heavy notations, and since we only need the action on $\End(V)$-valued tensors, it is sufficient to remark that $\hnabla$ satisfies
\begin{align*}
& \hnabla_\mu u = \nmu u = \pmu u + [A_\mu, u],
\label{eq-nablaLCu}
\\
&\hnabla_\mu a^\nu = 
\nmu a^\nu + \Gamma_{\mu \rho}^\nu a^\rho,
\qquad
\hnabla_\mu b_\nu 
=  \nmu b_\nu - \Gamma_{\mu\nu}^\rho b_\rho,
\\
&\hnabla_\mu g^{\alpha \beta} = 0,
\qquad
\hnabla_\mu g_{\alpha \beta} = 0,
\end{align*}
for any $\End(V)$-valued $(0,0)$-tensor $u$, $(1,0)$-tensor $a = a^\nu \pnu$, and $(0,1)$-tensor $b = b_\nu \,\dd x^\nu$. Here $A_\mu$ is the (local) gauge potential associated to $\nabla$. General formulas for $\End(V)$-valued $(r,s)$-tensors are easily obtained from these conventions.
\\
As usual, we use the short notation $\hnabla_\mu a^\nu = (\hnabla_\mu a)^\nu$ and $\hnabla_\mu b_\nu = (\hnabla_\mu b)_\nu$.

We first recall few formulae of Riemannian geometry:
\begin{align*}
&\abs{g}^{-1/2} \pmu \abs{g}^{1/2}
= \tfrac{1}{2} \pmu \ln \abs{g},
\\
&\Gamma_{\mu\nu}^\mu 
= -\tfrac{1}{2} g_{\alpha\beta} (\pnu g^{\alpha\beta}),
\\
&\Gamma^\rho=g^{\mu\nu} \Gamma_{\mu\nu}^\rho 
= \tfrac{1}{2} g^{\rho\sigma} g_{\alpha\beta} (\psigma g^{\alpha\beta}) - \psigma g^{\rho\sigma}
= - g^{\rho\sigma} \tfrac{1}{2} \psigma \ln \abs{g} - \psigma g^{\rho\sigma}.
\end{align*}
The curvature of the Levi-Civita connection $\gnabla$ is 
\begin{equation*}
R(X,Y) \vc \gnabla_X \gnabla_Y - \gnabla_Y \gnabla_X - \gnabla_{[X,Y]},
\end{equation*}
(\cite[p.~23]{Gilk01a}) and this expression is an endomorphism of the tangent bundle $TM$.\\
The Riemann tensors 
\begin{align*}
\tensor{R}{_{\mu\nu\, \rho}^\sigma} \psigma 
\vc R(\pmu, \pnu) \prho,
\qquad
R_{\mu\nu\, \rho\sigma} 
&\vc g_{\sigma\alpha} \tensor{R}{_{\mu\nu\, \rho}^\alpha}
= g( R(\pmu, \pnu) \prho, \psigma),
\end{align*}
satisfy $\tensor{R}{_{\mu\nu\, \rho}^\sigma}
=
\pmu \Gamma^\sigma_{\nu\rho} 
- \pnu \Gamma^\sigma_{\mu\rho} 
+ \Gamma^\sigma_{\mu\gamma} \Gamma^\gamma_{\nu\rho} 
- \Gamma^\sigma_{\nu\gamma} \Gamma^\gamma_{\mu\rho}$, 
with $\Gamma^\rho_{\mu\nu} \prho \vc \gnabla_{\pmu} \pnu$. \\
The Ricci tensor $\Ric_{\mu\nu}$ and the scalar curvature $\SC$ are
\begin{align}
\label{eq-def-Ricci-scalar-curv}
\Ric_{\mu\nu} \vc g^{\alpha\beta} R_{\mu\alpha\, \beta\nu} = \Ric_{\nu\mu},
\qquad
\SC &\vc g^{\mu\nu} \Ric_{\mu\nu}.
\end{align}
The Riemann tensor yields to some $\bbZ_2$-symmetries and to the first Bianchi identity:
\begin{align*}
R_{\mu\nu\, \rho\sigma} 
&= R_{\rho\sigma\, \mu\nu} = - R_{\nu\mu\, \rho\sigma} = - R_{\mu\nu\, \sigma\rho},
\qquad
R_{\mu\nu\, \rho\sigma} + R_{\mu\rho\, \sigma\nu} + R_{\mu\sigma\, \nu\rho} = 0.
\end{align*}
This also yields to some $\bbZ_2$-symmetries, to the derivative of the first Bianchi identity and the second Bianchi identity:
\begin{align}
& (\gnabla_\tau R)_{\mu\nu\, \rho\sigma}
= - (\gnabla_\tau R)_{\nu\mu\, \rho\sigma}
= - (\gnabla_\tau R)_{\mu\nu\, \sigma\rho}
= (\gnabla_\tau R)_{ \rho\sigma\,\mu\nu}, \nonumber
\\
& (\gnabla_\tau R)_{\mu\nu\, \rho\sigma} 
+ (\gnabla_\tau R)_{\nu\rho\,\mu\sigma}
+ (\gnabla_\tau R)_{\rho\mu\,\nu\sigma}
=0,  \nonumber
\\
& (\gnabla_\tau R)_{\mu\nu\, \rho\sigma}
+ (\gnabla_\rho R)_{\mu\nu\, \sigma\tau}
+ (\gnabla_\sigma R)_{\mu\nu\, \tau\rho}
= 0. \label{eq-Riemann-sym-2nd Bianchi}
\end{align}
In \eqref{eq-Riemann-sym-2nd Bianchi}, after a contraction over $\mu$ and $\rho$ first and then over $\nu$ and $\sigma$, and using the fact that $\gnabla_\tau g^{\mu\rho}=0$, one obtains $-(\gnabla_\tau \Ric)_{\nu\sigma} + (\gnabla_\rho R)^\rho\,_{\nu \,\sigma\tau} + (\gnabla_\sigma \Ric)_{\nu\tau} =0$ and $-(\gnabla_\tau \SC) + (\gnabla_\rho \Ric)^\rho\,_\tau +(\gnabla_\nu \,\Ric)^\nu\,_\tau=0$, so that 
\begin{align}
\label{div de Ricci}
g^{\mu\nu}(\gnabla_\mu \Ric)_{\nu \rho}=\tfrac{1}{2} (\gnabla_\rho \SC).
\end{align}
If we define
\begin{align}
\abs{R}^2 &\vc R_{\mu_1\mu_2\, \mu_3\mu_4} R^{\mu_1\mu_2\, \mu_3\mu_4} \label{square module of R}
\end{align}
then, by the first Bianchi identity, 
$R_{\mu_1\mu_2\, \mu_3\mu_4} R^{\mu_1\mu_3\, \mu_2\mu_4}
= - R_{\mu_1\mu_2\, \mu_3\mu_4} R^{\mu_1\mu_2\, \mu_4\mu_3}
- R_{\mu_1\mu_2\, \mu_3\mu_4} R^{\mu_1\mu_4\, \mu_3\mu_2}
= \abs{R}^2 - R_{\nu_1\nu_2\, \nu_3\nu_4} R^{\nu_1\nu_3\, \nu_2\nu_4}
$
with a relabeling for the second equality, so that
\begin{align*}
\tfrac{1}{2} \abs{R}^2
= R_{\mu_1\mu_2\, \mu_3\mu_4} R^{\mu_1\mu_3\, \mu_2\mu_4}.
\end{align*}
In the same vein, let us store
\begin{align}
\abs{\Ric}^2
& \vc \Ric_{\mu_1\mu_2} \Ric^{\mu_1\mu_2}
= \Ric_{\mu_1\mu_2} \Ric^{\mu_2\mu_1}. \label{square module of Ricci}
\end{align}
Given the field strength
\begin{align*}
F_{\mu\nu} \vc \pmu A_\nu - \pnu A_\mu + [A_\mu, A_\nu],
\end{align*}
one has for any section $u$ of $\End(V)$ and similarly, for any tensors $a = a^\nu \pmu$ and $b = b_\nu \dd x^\nu$,
\begin{align*}
 [\nabla_\mu, \nabla_\nu]\, u = [F_{\mu\nu}, u],
\qquad
[\gnabla_\mu, \gnabla_\nu] \,a^\rho
= \tensor{R}{_{\mu\nu\, \sigma}^\rho} a^\sigma,
\qquad
[\gnabla_\mu, \gnabla_\nu] \,b_\rho
&= - \tensor{R}{_{\mu\nu\, \rho}^\sigma} b_\sigma.
\end{align*}
Combining these expressions, for any $\End(V)$-valued tensors $a = a^\mu \pmu$ and $b = b_\nu\, \dd x^\nu$, one obtains
\begin{align*}
[\hnabla_\mu, \hnabla_\nu]\, u
&= [F_{\mu\nu}, u],
&
[\hnabla_\mu, \hnabla_\nu]\, a^\sigma
&= [F_{\mu\nu}, a^\rho] + \tensor{R}{_{\mu\nu\, \rho}^\sigma} a^\rho,
&
[\hnabla_\mu, \hnabla_\nu]\, b_\rho
&= [F_{\mu\nu}, b_\rho] - \tensor{R}{_{\mu\nu\, \rho}^\sigma} b_\sigma,
\end{align*}
and for any $s \in \Gamma(V)$
\begin{align}
\label{eq-commutator of hnabla}
[\hnabla_\mu,\hnabla_\nu] \,s =  F_{\mu\nu} \,s.
\end{align}
Moreover,
\begin{align}
&g^{\mu\nu} \,\hnabla_\mu \hnabla_\nu\, u 
=
g^{\mu\nu} ( \nmu \nnu \,u - \Gamma^\rho_{\mu\nu} \nrho u )
=
g^{\mu\nu} \nmu \nnu \,u - \Gamma^\mu \nmu u,
\label{eq-laplacianLCu}
\\
&\hnabla_\mu p^\mu
 = \nmu p^\mu+\Gamma^\mu_{\mu \rho} \,p^\rho
= \nmu p^\mu - \tfrac{1}{2} g_{\alpha\beta} (\pmu g^{\alpha\beta}) \,p^\mu.    \nonumber
\end{align}

\medskip
Let us now come back to the operator $P$.
Applying the definition \eqref{eq-def-Nmu}, we can rewrite $P$ in a (fully) covariant way as
\begin{align}
\label{eq-P-upq-hatnabla}
P = - ( g^{\mu\nu} \, \hnmu u \hnnu  + p^\nu \,\hnnu +q)= - (g^{\mu\nu} u \hnmu \hnnu + N^\nu\, \hnnu + q),
\end{align}
and the following result is a generalization of \cite[Lemma 1.2.1]{Gilk03a}:

\begin{lemma}
\label{lem-N=0}
Given $P$ as in \eqref{eq-P-upq-hatnabla}, there exist a connection $\hnabla'$ and a section $q'$ of $\End(V)$ such that 
\begin{align}
\label{eq-P en nabla'} 
P= -(g^{\mu\nu} u\hnabla_\mu'\hnabla_\nu' +q'),
\end{align}
given by $\hnabla'_\mu\vc \hnabla_\mu +\tfrac 12 u^{-1} N_\mu$ and $q'\vc q- \tfrac 12 g^{\mu\nu} u \hnabla_\mu(u^{-1} N_\nu) -\tfrac 14 N^\nu u^{-1} N_\nu$.
\end{lemma}

\begin{proof}
This follows from a direct computation (omitting the section applied upon):
\begin{align*}
  g^{\mu\nu} u\hnabla_\mu'\hnabla_\nu' +q' & =  g^{\mu\nu} u\,(\hnabla_\mu +\tfrac 12 u^{-1}N_\mu)\,(\hnabla_\nu +\tfrac 12 u^{-1}N_\nu) +q- \tfrac 12 g^{\mu\nu} u \hnabla_\mu(u^{-1} N_\nu) -\tfrac 14 N^\nu u^{-1} N_\nu \\
 & = +g^{\mu\nu}\big[u\hnabla_\mu \hnabla_\nu  +\tfrac 12 u \hnabla_\mu (u^{-1}N_\nu )
  +\tfrac 12 N_\nu \hnabla_\mu +\tfrac 12 N_\mu \hnabla_\nu+\tfrac 14 N_\mu u^{-1} N_\nu -\tfrac 12 u\hnabla_\mu (u^{-1} N_\nu)\big]
+q  -\tfrac 14 N^\nu u^{-1} N_\nu \\
 & = g^{\mu\nu} \, \big[u\hnabla_\mu \hnabla_\nu + N_\mu \hnabla_\nu \big] +q =-P.
\end{align*}
\end{proof}

\begin{remark}
\label{rem-N zero or not}
Compared with \eqref{eq-P-upq-hatnabla}, this rewriting of $P$ in \eqref{eq-P en nabla'} greatly simplifies the computations of $\calR_r$ because it means that we may assume $N^\nu=0$ for all $\nu$. This is the traditional way to present the results on the heat kernel coefficients, see for instance \cite{Gilk01a,Gilk03a}. In particular, $N^\nu=0$ for all $\nu$ is equivalent to $L^\mu=-\Gamma^\mu u$ in \eqref{def-PHLq}. \\
If one insists on keeping $N^\mu\neq 0$, Corollary~\ref{cor-(inverse h)ah} can be useful because $\hnabla_\mu' u=\hnabla_\mu u +\tfrac{1}{2} u^{-1}N_\mu u$.
\end{remark}

\subsection{Covariant derivatives and normal coordinates}
\label{Covariant derivatives and normal coordinates}

In this subsection, some results on the iterated covariant derivatives of $u, \,p^\mu, \dots$ are shown and they will be used for the computation of $\calR_2$. The code generates their generalizations to higher orders in derivatives for the computation of $\calR_4$.

In the following, for any $\ell \in \bbN$ and indices $\nu_1, \dots, \nu_\ell$, we use the compact notation
\begin{equation*}
\nabla^\ell_{\nu_1 \dots \nu_\ell} \vc \nabla_{\!\nu_1} \cdots \nabla_{\!\nu_\ell} \text{ and the same for } \hnabla\text{ and }\partial.
\end{equation*}
Comparing the action of $\hnabla$ and $\nabla$ on $u$, we get
\begin{align}
\hnabla_{\nu_1} u
= 
\nabla_{\nu_1} u, 
\qquad
\hnabla^2_{\nu_1\nu_2} u
= 
\nabla^2_{\nu_1\nu_2} u - \Gamma^{\sigma_1}_{\nu_1\nu_2} (\hnabla_{\sigma_1} u) .
\label{hnabla2 u}
\end{align}
We deduce from \eqref{hnabla2 u} that if $u$ is parallel for $\nabla$ (\textsl{i.e.} $\nmu u = 0$), then $u$ is also parallel for $\hnabla$.  
Similarly,
\begin{align}
\hnabla_{\nu_1} p^{\mu_1}
& = 
\nabla_{\nu_1} p^{\mu_1} + \Gamma^{\mu_1}_{\nu_1 \sigma_1} \,p^{\sigma_1}. 
\end{align}
For $L^\mu =N^\mu- \Gamma^\mu u= p^\mu + g^{\mu\nu} (\nnu u) - \Gamma^\mu u$, 
\begin{align}
\nabla_{\nu_1} L^{\mu_1}
&= \nabla_{\nu_1} N^{\mu_1} 
	- \Gamma^{\mu_1} (\nabla_{\nu_1} u)
         - (\partial_{\nu_1} \Gamma^{\mu_1})\, u ,
\label{nabla L}
\\
\nabla_{\nu_1} N^{\mu_1}
&= \nabla_{\nu_1} p^{\mu_1} 
	+ g^{\mu_1\sigma_1} (\nabla^2_{\nu_1\sigma_1} u) 
	+ (\partial_{\nu_1} g^{\mu_1\sigma_1}) (\nabla_{\sigma_1} u).\nonumber
\end{align}
To compute the RHS of these expressions, we also need to know the derivatives of $\Gamma_{\alpha\beta}^\gamma$ and $\Gamma^\gamma$:
\begin{align*}
\Gamma_{\alpha\beta}^\gamma
&= 
\tfrac{1}{2} g^{\gamma\sigma_1} (
	\partial_\alpha g_{\sigma_1\beta} 
	+ \partial_\beta g_{\sigma_1\alpha} 
	- \partial_{\sigma_1} g_{\alpha\beta}
),
\\
\partial_{\nu_1} \Gamma_{\alpha\beta}^\gamma
&= 
\tfrac{1}{2} \big[ 
(\partial_{\nu_1} g^{\gamma\sigma_1}) (
	\partial_\alpha g_{\sigma_1\beta} 
	+\partial_\beta g_{\sigma_1\alpha} 
	-\partial_{\sigma_1} g_{\alpha\beta}
)
+ g^{\gamma\sigma_1} (
	\partial^2_{\nu_1\alpha} g_{\sigma_1\beta} 
	+ \partial^2_{\nu_1\beta} g_{\sigma_1\alpha} 
	- \partial^2_{\nu_1\sigma_1} g_{\alpha\beta}
)
\big].
\end{align*}
Thus, for $\Gamma^\gamma = g^{\alpha\beta} \Gamma^\gamma_{\alpha\beta}$, 
\begin{align*}
\partial_{\nu_1} \Gamma^\gamma
&= (\partial_{\nu_1} g^{\alpha\beta})\, \Gamma^\gamma_{\alpha\beta}
+ g^{\alpha\beta} (\partial_{\nu_1} \Gamma^\gamma_{\alpha\beta}).
\end{align*} 

The swap of $\nabla$ to $\hnabla$ in (\ref{hnabla2 u}--\ref{nabla L}) can be reduced to a peculiar coordinate system and we now present  some results concerning derivatives of quantities in normal coordinates, namely a geodesic coordinate system centered at a pinned point $x \in M$.  

Let us use the notation $\tonc$ to map a quantity to its value in normal coordinates at $x$. We warn the reader that, to alleviate the presentation, we omit in the sequel the explicit dependency to $x$. \\
For the following results on the metric and its inverse or on the Christoffel symbols, see for instance \cite[Sect.~1.11.3]{Gilk01a},  \cite[p.~5]{Gilk03a}, \cite{Sakai} or \cite{Brew09a}.
\begin{align}
\label{eq-partial0gq}
g_{\alpha\beta}
&\tonc 
	\delta_{\alpha\beta},
\qquad
\partial_{\nu_1} g_{\alpha\beta}
\tonc 
	0,
\qquad 
\partial^2_{\nu_1\nu_2} g_{\alpha\beta}
\tonc 
	\sumperm_{\nu_1, \nu_2} \tfrac{1}{3} R_{\nu_1\alpha\, \nu_2\beta}
	= \tfrac{1}{3} \big[ R_{\nu_1\alpha\, \nu_2\beta} + R_{\nu_2\alpha\, \nu_1\beta} \big],
\\
\partial^3_{\nu_1\nu_2\nu_3} g_{\alpha\beta}
&\tonc 
	\sumperm_{\nu_1, \nu_2, \nu_3} 
	\tfrac{1}{3!} (\hnabla_{\nu_1} R)_{\nu_2 \alpha\, \nu_3 \beta},
\qquad
\label{eq-partial4gq}
\partial^4_{\nu_1\nu_2\nu_3\nu_4} g_{\alpha\beta}
\tonc 
	\sumperm_{\nu_1, \nu_2, \nu_3, \nu_4} \big[ 
	\tfrac{1}{20}  (\hnabla^2_{\nu_1\nu_2} R)_{\nu_3 \alpha\, \nu_4\beta}
	+ \tfrac{2}{45} R_{\nu_1\alpha\, \nu_2\sigma_1} \tensor{R}{_{\nu_3 \beta\, \nu_4}^{\sigma_1}}
	\big]
\end{align}
where 
$\sumperm_{\nu_1, \dots, \nu_n} A_{\nu_1 \dots \nu_n}
\vc
\sum_{\,\sigma \in \bbS_n} A_{\nu_{\sigma(1)} \dots \nu_{\sigma(n)}}
$ and $\bbS_n$ is the permutation group of $n$ elements. \\
Similarly,
\begin{align}
g^{\alpha\beta}
&\tonc 
	\delta^{\alpha\beta},
\qquad \label{eq-partial1ginvq}
\partial_{\nu_1} g^{\alpha\beta}
\tonc 
	0,
\\
\Gamma_{\alpha\beta}^\gamma
&\tonc 
	0,
\qquad
\partial_{\nu_1} \Gamma_{\alpha\beta}^\gamma
\tonc 
	\tfrac{1}{3} \sumperm_{\alpha, \beta} \tensor{R}{_{\nu_1\alpha\, \beta}^{\gamma}}
	= 
	\tfrac{1}{3} \big[ 
		\tensor{R}{_{\nu_1\alpha\, \beta}^{\gamma}} 
		+ \tensor{R}{_{\nu_1\beta\, \alpha}^{\gamma}}  \nonumber
		\big].
\end{align}
Thus we obtain
\begin{align}
\Gamma^{\mu_1}
&\tonc 
	0, 
\qquad
\partial_{\nu_1} \Gamma^{\mu_1}
\tonc 
	\tfrac{2}{3} \tensor{\Ric}{_{\nu_1}^{\mu_1}},  \nonumber
\\
\label{eq-hnabla1nc u}
\nabla_{\nu_1} u
&\tonc 
	\hnabla_{\nu_1} u \quad \text{(equality in any coordinate system)}, 
\\
\label{eq-hnabla2nc u}
\nabla^2_{\nu_1\nu_2} u
&\tonc 
	\hnabla^2_{\nu_1\nu_2} u,
\\
\label{eq-hnabla1nc p}
\nabla_{\nu_1} p^{\mu_1}
&\tonc 
	\hnabla_{\nu_1} p^{\mu_1}.
\end{align}
For $L^\mu$, we also get from \eqref{nabla L}:
\begin{align*}
L^{\mu_1}
&\tonc 
	p^{\mu_1} 
	+ g^{\mu_1\sigma_1} (\nabla_{\sigma_1} u),
\qquad
\nabla_{\nu_1} L^{\mu_1}
\tonc 
	\nabla_{\nu_1} p^{\mu_1} 
	+ g^{\mu_1\sigma_1} (\nabla^2_{\nu_1\sigma_1} u) 
	- \tfrac{2}{3} \,\tensor{\Ric}{_{\nu_1}^{\mu_1}}u.
\end{align*}
The corresponding expressions given in terms of $N^\mu$ are
\begin{align*}
L^{\mu_1}
&\tonc 
	N^{\mu_1},
\qquad
\nabla_{\nu_1} L^{\mu_1}
\tonc 
	\nabla_{\nu_1} N^{\mu_1}- \tfrac{2}{3}\, \tensor{\Ric}{_{\nu_1}^{\mu_1}}u.
\end{align*}
Finally, for $H^{\mu_1\mu_2}=g^{\mu_1\mu_2} u$, we obtain directly from \eqref{eq-partial1ginvq}:
\begin{align*}
H^{\mu_1\mu_2}
&\tonc
	g^{\mu_1\mu_2}\, u,
\qquad
\nabla_{\nu_1} H^{\mu_1\mu_2}
\tonc 
	g^{\mu_1\mu_2} (\nabla_{\nu_1} u),
\qquad
\nabla^2_{\nu_1\nu_2} H^{\mu_1\mu_2}
\tonc 
	g^{\mu_1\mu_2} (\nabla^2_{\nu_1\nu_2} u)
	- \tfrac{1}{3} \sumperm_{\nu_1, \nu_2}
		\tensor{R}{_{\nu_1}^{\mu_1}_{\, \nu_2}^{\mu_2}}\, u .
\end{align*}

\section{The method and its simplifications}
\label{The method and its simplifications}

In this section, we start with $h$ equal to $H(x,\xi)$, so that the definition \eqref{eq-def Pi f} is specialized to the operator $f_k(\xi)=\pi_{\widetilde{I}_k}$ as seen in \eqref{eq-fk as a pi}. With $\partial = \nnu$ acting on any local section $s$ of $V$ by $\nnu \,s = \pnu \,s + A_\nu \,s$, so that the required Leibniz rule holds, and using \eqref{eq-Delta-fkxi}, Proposition~\ref{prop-Propagation} becomes

\begin{proposition}
\label{lem-fk-nabla-vector-propagation}
Given a local trivialization $s : U \to \bbC^N$ of a section in $\Gamma(V)$, and $B_\ell$, $1 \leq i\leq k$, which are $k$ matrix-valued differential operators in $\nmu$ depending on $x$ and (linearly in) $\xi$, the following holds true:
\begin{align}
f_k(\xi)[B_1 \otimes \cdots \otimes B_i \nabla_\nu \otimes \cdots \otimes B_k]\, s
= \begin{aligned}[t]
& \sum_{j=i+1}^k f_k(\xi)[B_1 \otimes \cdots \otimes(\nnu B_j)\otimes \cdots\otimes  B_k]\, s
\\
&- \sum_{j=i}^k f_{k+1}(\xi)[B_1 \otimes \cdots \otimes B_j\otimes (\nnu H)\otimes B_{j+1}\otimes \cdots\otimes  B_k]\, s
\\
&+  f_k(\xi)[B_1 \otimes \cdots \otimes B_i \otimes \cdots \otimes B_k] (\nnu \,s).
\end{aligned}
\label{eq-fk-nabla-vector-propagation}
\end{align}
\end{proposition}

When $f_k(\xi)[B_1 \otimes \cdots \otimes B_k]$ contains more than one covariant derivative $\nabla_{\!\nu}$ to propagate, they can accumulate on $s$ as $\nabla^{\ell}_{\!\nu_1\dots \nu_\ell} s$, and this produces complicated expressions. 

We are interested in the computation of $f_k(\xi)[B_1 \otimes \cdots \otimes B_k]$ as a matrix-valued function, that is as a linear map: $v\in \bbC^N  \mapsto f_k(\xi)[B_1 \otimes \cdots \otimes B_k] \,v$. In that case, $\nnu v = A_\nu\, v$ and $\nabla^2_{\!\nu_1\nu_2} v = (\partial_{\nu_1} A_{\nu_2} + A_{\nu_1} A_{\nu_2})\, v$ because $v$ is constant.

In general, we can write the result after the propagation of some $\nabla_{\!\nu}$ as a sum of terms like $f_{k'}(\xi)[B_1 \otimes \cdots \otimes B_{k'}] \,Q[A] \,v$, where $Q[A]$ is a matrix-valued function written as a polynomial expression in the $A_\mu$ and their derivatives. \\
We begin the process with $f_k(\xi)[B_1 \otimes \cdots \otimes B_k] \,Q_0[A]\, v$ where  $Q_0[A] = \bbbone$ (the constant unital section in $\End(V)$) and after applying $\ell$ covariant derivatives, we get $f_{k'}(\xi)[B_1 \otimes \cdots \otimes B_{k'}]\, Q_\ell[A] \,v$ where
\begin{equation*}
Q_\ell[A]\, v = \nnu Q_{\ell-1}[A]\, v = \big(\pnu Q_{\ell-1}[A] + A_\nu Q_{\ell-1}[A]\big) \,v.
\end{equation*}
It is easy to establish that $Q_\ell[A]$ is an homogeneous polynomial of degree $\ell$ when counting a degree $1$ for each $\pnu$ and $A_\nu$. In the final expression, these factors $Q[A]$ generate “gauge invariant” contributions, in term of the curvature of $A_\mu$ and its derivatives.

Finally, notice that Proposition~\ref{lem-fk-nabla-vector-propagation} reduces to \cite[Lemma~2.1]{IochMass17a} for $A_\mu = 0$, \textsl{i.e.} for $\nnu = \pnu$.

\medskip
For $\mu_i\in \{1,\dots,d\}$, let $T_{\mu_1\dots \mu_{2p},k}$ be the operators from $\calA^{\otimes^k}$ to itself defined by
\begin{align*}
T_{\mu_1\dots \mu_{2p},k}
\vc g_d \, G_{\mu_1\dots \mu_{2p}} \, I_{d/2+p,k}(r_0, \dots, r_k)\, (E_0^R \otimes \cdots \otimes E_k^R),
\end{align*}
with the notations
\begin{align}
g_d (x)
&\vc \tfrac{1}{(2\pi)^{\,d}} \int_{\bbR^d}\! \dd\xi\, e^{-\abs{\xi}_{g(x)}^2}
= \tfrac{\abs{g(x)}^{1/2}}{2^{d}\,\pi^{d/2}}\,, 
\label{eq-gd}
\\ 
G_{\mu_1\dots \mu_{2p}} (x)
&\vc  \tfrac{1}{(2\pi)^{\,d}\,g_d(x)} \int_{\bbR^d} \! \dd\xi\, \xi_{\mu_1} \cdots \xi_{\mu_{2p}}\, e^{-g^{\alpha\beta}(x)\, \xi_\alpha\xi_\beta}
\nonumber
\\
&= \tfrac{1}{2^{2p}\,p!} \, \big( \sum_{\rho \in \bbS_{2p}} g_{\mu_{\rho(1)} \mu_{\rho(2)}}\cdots g_{\mu_{\rho(2p-1)} \mu_{\rho(2p)}} \big)(x)
= \tfrac{(2p)!}{2^{2p}\,p!} \,g_{(\mu_1\mu_2}\cdots g_{\mu_{2p-1}\mu_{2p})}(x),
 \label{eq-Gmu-permutations}
\end{align}
where $\bbS_{2p}$ is the symmetric group of permutations on $2p$ elements and the parenthesis in the index of $g$ is the complete symmetrization over all indices and with the convention that, when $p=0$, $G_{\mu_1\dots \mu_{2p}}$ is just 1. For instance, $G_{\mu_1\mu_2}=\tfrac{1}{2}\, g_{\mu_1\mu_2}$ and 
\begin{align}
\label{eq-G1234}
G_{\mu_1\mu_2\mu_3\mu_4} 
= \tfrac{1}{4} \, (
g_{\mu_1\mu_2}g_{\mu_3\mu_4}
+ g_{\mu_1\mu_3}g_{\mu_2\mu_4}
+ g_{\mu_1\mu_4}g_{\mu_2\mu_3}
).
\end{align}

Let us also introduce the operators
\begin{align}
\label{eq-def-Xmuk}
 X_{\alpha, k, \mu_1\dots \mu_{2p}} 
 \vc \, G_{\mu_1\dots \mu_{2p}} \, X_{\alpha ,k},
\end{align}
which are justified as
\begin{align}
\label{eq-fk-Xmuk}
 X_{d/2+p,k,\mu_1\dots \mu_{2p}} 
 = \tfrac{1}{g_d}\tfrac{1}{(2\pi)^{\,d}} \int_{\bbR^d} \!\dd\xi \, \xi_{\mu_1} \cdots \xi_{\mu_{2p}} \, f_k(\xi),
\end{align}
because, in \eqref{eq-integral of fk}, $c_{k,\mu_1,\dots,\mu_{2p}} = (2\pi)^{\,d} g_d \,G_{\mu_1,\dots,\mu_{2p}}$ 
and there is the following relations between the $f_k$ and the $X_{\alpha,k}$:
\begin{align*}
\tfrac{1}{(2\pi)^{\,d}} \int_{\bbR^d} \dd\xi \, \xi_{\mu_1} \cdots \xi_{\mu_{2p}} \, f_k(\xi)[ B_1 \otimes \cdots \otimes B_k]
&= \mul\circ T_{\mu_1\dots \mu_{2p},k}[B_1 \otimes \cdots \otimes B_k]
= g_d \, G_{\mu_1\dots \mu_{2p}} \, X_{d/2+p,k}[B_1 \otimes \cdots \otimes B_k]
\\
&
= g_d \,  X_{d/2+p,k,\mu_1\dots \mu_{2p}} [B_1 \otimes \cdots \otimes B_k].
\end{align*}
The appearance of $H(\xi)$ in \eqref{eq-fk-nabla-vector-propagation} forces to consider the $\xi$-dependence of the arguments $B_1 \otimes \cdots \otimes B_k$: each factor depends polynomially on $\xi=(\xi_1,\dots,\xi_d)$ because $B_i=\sum B^{\mu_1\dots\mu_{\ell_i}} \xi_{\mu_1}\dots \xi_{\mu_{\ell_i}}$ with $B^{\mu_1\dots\mu_{\ell_i}}(x) \in M_N(\bbC)$ independent of $\xi$. \\
Thus $B_1 \otimes \cdots \otimes B_k$ is a sum of terms like $\xi_{\mu_1}\dots \xi_{\mu_{\ell} }(B_1 \otimes \cdots \otimes B_k)^{\mu_1\dots \mu_{\ell}}$ and the symmetry of the $\xi$-integral implies that the ones which only survive are when $\ell=2p$ for some $p\in \bbN$. \\
As a consequence, each function $a_r(a, P)(x)$ is expressed formally as a sum
\begin{align}
\label{eq-a-T-rel}
a_r(a, P)(x)
&= \abs{g}^{-1/2}\sum \,\tr \,\big[ a(x) \,\mul\circ T_{\mu_1\dots\mu_{2p},k}(x)[(B_1(x) \otimes \cdots \otimes B_k(x))^{\mu_1\dots\mu_{2p}}] \big],
\end{align}
and the wanted factor $\calR_r$ is a sum
\begin{align}
\label{eq-R-Xdpk}
\calR_r 
&=\abs{g}^{-1/2}\sum \, \mul\circ T_{\mu_1\dots\mu_{2p},k}[ (B_1 \otimes \cdots \otimes B_k)^{\mu_1\dots\mu_{2p}}]
=\tfrac{1}{(4\pi)^{\,d/2}} \sum  \,X_{d/2+p,k,\mu_1\dots \mu_{2p}}[(B_1 \otimes \cdots \otimes B_k)^{\mu_1\dots \mu_{2p}}]
\end{align}
because $\abs{g}^{-1/2} g_d = \tfrac{1}{2^{d}\,\pi^{d/2}}$. \\
Before we give a precise way to compute \eqref{eq-R-Xdpk}, we now store two technical lemmas
\begin{lemma}
For any $p\in \bbN^*$, $\alpha\in \bbR$ and $k\in\bbN^*$, 
\begin{align}
\label{eq-contraction-g-Gmu}
g^{\mu_{2p-1} \mu_{2p}}\, G_{\mu_1\dots \mu_{2p}} 
&= (\tfrac{1}{2} d+ p - 1) \,G_{\mu_1\dots \mu_{2(p-1)}},
\\
\label{eq-contraction-g-Xmu}
g^{\mu_{2p-1} \mu_{2p}} \,X_{\alpha, k, \mu_1\dots \mu_{2p}} 
&= (\tfrac{1}{2}d + p - 1)\, X_{\alpha, k, \mu_1\dots \mu_{2(p-1)}}.
\end{align}
\end{lemma}

\begin{proof}
To compute the $\xi$-integration defining $G_{\mu_1\dots \mu_{2p}}$, we use spherical coordinates $\xi=s \sigma$ with $s \vc (g^{\mu\nu}\xi_\mu\xi_\nu)^{1/2}$ and $\sigma = s^{-1} \xi \in S^{d-1}_g$, the unit sphere in ($\bbR^d,g$). Then
\begin{align*}
g^{\mu_{2p-1} \mu_{2p}} \,G_{\mu_1\dots \mu_{2p}} 
&=  \tfrac{1}{(2\pi)^{\,d}\,g_d} \int_{\bbR^d}\! \dd\xi\, \xi_{\mu_1} \cdots \xi_{\mu_{2(p-1)}} (g^{\mu_{2p-1} \mu_{2p}} \xi_{\mu_{2p-1}} \xi_{\mu_{2p}}) \, e^{-g^{\alpha\beta} \xi_\alpha\xi_\beta}
\\
&= \tfrac{1}{(2\pi)^{\,d}\,g_d} \big(\int_{S^{d-1}_g}\! \dd\Omega_g(\sigma)\,\sigma_{\mu_1}\cdots \sigma_{\mu_{2p-1}}\big) \,
\big( \int_0^\infty \!\dd s\, s^{d-1+2p} e^{-s^2} \big)
\\
&= \tfrac{1}{(2\pi)^{\,d}\,g_d} \big(\int_{S^{d-1}_g} \!\dd\Omega_g(\sigma)\,\sigma_{\mu_1}\cdots \sigma_{\mu_{2p-1}}\big) \,
\tfrac{1}{2} \Gamma( \tfrac{1}{2}d + p)
\\
&= \tfrac{1}{(2\pi)^{\,d}\,g_d} \big(\int_{S^{d-1}_g}\! \dd\Omega_g(\sigma)\,\sigma_{\mu_1}\cdots \sigma_{\mu_{2p-1}}\big) \, 
\tfrac{1}{2} ( \tfrac{1}{2}d + p - 1) \Gamma( \tfrac{1}{2}d + p - 1)
\\
&= ( \tfrac{1}{2}d + p - 1 ) \,\tfrac{1}{(2\pi)^{\,d}\,g_d} \,\big(\int_{S^{d-1}_g} \!\dd\Omega_g(\sigma)\,\sigma_{\mu_1}\cdots \sigma_{\mu_{2p-1}}\big) \,
\big( \int_0^\infty \!\dd s\, \,s^{d-1+2(p-1)} \,e^{-s^2} \big)
\\
&= (\tfrac{1}{2}d + p - 1)\, G_{\mu_1\dots \mu_{2(p-1)}}
\end{align*}
and the equality \eqref{eq-contraction-g-Xmu} follows from the definition \eqref{eq-def-Xmuk}.
\end{proof}

The full method to compute $\calR_r$ consists to apply \eqref{eq-fk-nabla-vector-propagation} starting from terms of the form
\begin{equation*}
\tfrac{1}{(2\pi)^{\,d}} \int \!\dd\xi \, \xi_{\mu_1} \cdots \xi_{\mu_{2p}} \, f_k(\xi)[(B_1 \otimes \cdots \otimes B_i \nnu \otimes \cdots \otimes B_k)^{\,\mu_1\dots \mu_{2p}}]
\end{equation*}
generated by the series in \eqref{Volt} (the convention of Remark~\ref{rmk fk convention} is used). Considering the last term in \eqref{eq-fk-nabla-vector-propagation}, the most general expression to start with is (see discussion after Proposition~\ref{lem-fk-nabla-vector-propagation}):
\begin{equation*}
\tfrac{1}{(2\pi)^{\,d}} \int \!\dd\xi \, \xi_{\mu_1} \cdots \xi_{\mu_{2p}} \, f_k(\xi)[(B_1 \otimes \cdots \otimes B_i \nnu \otimes \cdots \otimes B_k)^{\,\mu_1\dots \mu_{2p}}]\, Q[A].
\end{equation*}
Then the LHS of \eqref{eq-fk-nabla-vector-propagation} produces three kinds of terms. The first ones come from the propagation of $\nabla$ on the arguments:
\begin{equation*}
\sum_{j=i+1}^k \tfrac{1}{(2\pi)^{\,d}} \int \!\dd\xi \, \xi_{\mu_1} \cdots \xi_{\mu_{2p}} \, f_k(\xi)[(B_1 \otimes \cdots \otimes (\nnu B_j)\otimes \cdots \otimes  B_k)^{\,\mu_1\dots \mu_{2p}}]\, Q[A].
\end{equation*}
The second ones consist of adding $-\nnu \,H =- ( \nnu\, H^{\alpha\beta} )\,\xi_\alpha \xi_\beta$ as an argument after $B_i$, so we get:
\begin{equation*}
- \sum_{j=i}^k \tfrac{1}{(2\pi)^{\,d}} \int \!\dd\xi \, \xi_{\mu_1} \cdots \xi_{\mu_{2p}} \xi_\alpha \xi_\beta\, f_{k+1}(\xi)[(B_1 \otimes \cdots \otimes B_j \otimes  (\nnu\, H^{\alpha\beta}) \otimes B_{j+1} \otimes \cdots \otimes  B_k)^{\,\mu_1\dots \mu_{2p}}]\, Q[A].
\end{equation*}
The third ones modify the matrix-valued polynomials $Q[A]$ as:
\begin{equation*}
\tfrac{1}{(2\pi)^{\,d}} \int \!\dd\xi \, \xi_{\mu_1} \cdots \xi_{\mu_{2p}} \, f_k(\xi)[(B_1 \otimes \cdots \otimes B_i \nnu \otimes \cdots \otimes B_k)^{\,\mu_1\dots \mu_{2p}}] \,(\pnu + A_\nu) \,Q[A].
\end{equation*}
Replacing the $f_k$'s and the integrations along $\xi$ with the $X_{d/2+p, k, \mu_1\dots \mu_{2p}}$ as in \eqref{eq-fk-Xmuk}, we finally obtain
\begin{multline}
\label{eq-Xmuk-derivation}
X_{d/2+p, k, \mu_1\dots \mu_{2p}}[(B_1 \otimes \cdots \otimes B_i \nnu \otimes \cdots \otimes B_k)^{\,\mu_1\dots \mu_{2p}}] \,Q[A]
\\
= \begin{aligned}[t]
&+ \sum_{j=i+1}^k X_{d/2+p, k, \mu_1\dots \mu_{2p}}[(B_1 \otimes \cdots \otimes (\nnu\, B_j)\otimes \cdots \otimes  B_k)^{\,\mu_1\dots \mu_{2p}}] \,Q[A]
\\
& - \sum_{j=i}^k X_{d/2+p+1, k+1, \mu_1\dots \mu_{2(p+1)}}[(B_1 \otimes \cdots \otimes B_j \otimes (\nnu \,H^{\,\mu_{2p+1}\mu_{2p+2}}) \otimes \cdots \otimes  B_k)^{\,\mu_1\dots \mu_{2p}}]\, Q[A]
\\
& + X_{d/2+p, k, \mu_1\dots \mu_{2p}}[(B_1 \otimes \cdots \otimes B_k)^{\,\mu_1\dots \mu_{2p}}] (\pnu + A_\nu) \,Q[A].
\end{aligned}
\end{multline}
This relation is the “integrated” version of \eqref{eq-fk-nabla-vector-propagation}. One has to look inside as the tensor product over the field $\bbC$ and not over functions. In other words, \emph{it is necessary to keep the functions $g^{\alpha\beta}$ and their derivatives in front of their arguments in the tensor product until all the derivations in the arguments have propagated}. Working with the matrix-valued functions $\nnu \,H^{\mu_{2p+1}\mu_{2p+2}}$ prevents this temptation.

From this result, we then get the following simplification in the computation of $\calR_r$:
\begin{proposition}
\label{prop-simplified-computation}
All the terms of \eqref{eq-R-Xdpk} can be obtained starting from their analogs in \eqref{Volt} with the replacement of operators $f_k$ by $X_{d/2+p, k, \mu_1\dots \mu_{2p}}$ as in \eqref{eq-fk-Xmuk} and iteratively applying the rule \eqref{eq-Xmuk-derivation}. At the end of this iteration, one can use \eqref{eq-contraction-g-Xmu} to deal with any remaining function $g^{\alpha\beta}$ together with normal coordinates to deal with derivatives of the metric.
\end{proposition}

In order to simplify the computation, we can omit the operators $X_{d/2+p, k, \mu_1\dots \mu_{2p}}$ in \eqref{eq-Xmuk-derivation} and work directly at the level of their arguments. We will also make use of two other useful symbolic notations. \\
Let us adopt the notation $\leadsto$ to express the development of arguments induced by \eqref{eq-Xmuk-derivation}. In order to take into account the presence of the matrix-valued polynomials $Q[A]$ (appearing in \eqref{eq-Xmuk-derivation}) multiplied on the right of the application of $X_{d/2+p, k, \mu_1\dots \mu_{2p}}$ on the arguments, we denote its presence with the separation symbol $\mid$, except if $Q[A] = \bbbone$. 

Then the computation consists to perform the propagation of all the derivatives by applying, as many times as necessary, the following symbolic rule: If $B_i$, $1\leq i\leq k$, are $k$ matrix-valued differential operators in $\nmu$ depending on $x$ and independent of $\xi$, then,
\begin{align}
\label{eq-Xmuk-arguments-derivation}
(B_1 \otimes \cdots \otimes B_i \nnu \otimes \cdots \otimes B_k)^{\,\mu_1\dots \mu_{2p}} \mid Q[A]
\leadsto
&+ \sum_{j=i+1}^k (B_1 \otimes \cdots \otimes (\nnu \,B_j)\otimes \cdots \otimes  B_k)^{\,\mu_1\dots \mu_{2p}} \mid Q[A]
\\
& - \sum_{j=i}^k (B_1 \otimes \cdots \otimes B_j \otimes (\nnu \,H^{\,\mu_{2p+1}\mu_{2p+2}}) \otimes B_{j+1} \otimes \cdots \otimes  B_k)^{\,\mu_1\dots \mu_{2p}} \mid Q[A] \nonumber
\\
& + (B_1 \otimes \cdots \otimes B_i \otimes \cdots \otimes B_k)^{\,\mu_1\dots \mu_{2p}} \mid \nnu \,Q[A]. \nonumber
\end{align}
Once this rule has been used, the operators $X_{d/2+p, k, \mu_1\dots \mu_{2p}}$ to apply on each argument in the obtained sum are uniquely determined by the number of free indices $\mu_i$ and the number of arguments in tensor products.

\begin{remark}
When $u=\bbbone$, the previous formula \eqref{eq-Xmuk-arguments-derivation} cannot be simplified because $H^{\mu\nu}=g^{\mu\nu}\,\bbbone$ commutes with matrix-valued functions but not with some differential operators $\nabla_\nu$ possibly contained in some $B_\ell$ for $1\leq \ell \leq i$. This implies in particular that we can not hope for simplifications at this computation stage even if $u=\bbbone$, and the number of terms produced by  successive applications of \eqref{eq-Xmuk-arguments-derivation} is independent of the exact form of $H^{\mu\nu}$. Only subsequent computations can use some hypothesis on $u$ for simplifications.
\end{remark}

\begin{remark}
It is tempting in the previous method to start with $P$ written in terms of $\hnabla$ as in \eqref{eq-P-upq-hatnabla} and to propagate $\hnabla_\mu$ instead of $\nabla_\mu$. But this requires to get an analogue of formula \eqref{eq-P and exp} which needs to make sense of $\hnabla_\mu \xi_\nu$: the variables $\xi_\nu$ are the Fourier dual of the variables $x^\nu$ and, even if they carry a space index, they are not the component of a tensor field on $M$. While here these variables $\xi_\nu$ are silent since only confined in the $G$-tensors (as a consequence of $\nabla_\mu \xi_\nu = 0$), they would have to remain both in the $B_j$ and $H^{\mu\nu}$  to be sensitive to the action of $\hnabla$ in \eqref{eq-Xmuk-arguments-derivation}, thus generating more terms. This would give directly the result in terms of $\hnabla$ while here, we are obliged to exchange with some efforts the $\nabla$ with $\hnabla$ (in normal coordinates) at the end.
\end{remark}

\section{\texorpdfstring{Results on $\calR_0$ and $\calR_2$}{Results on R0 and R2}}
\label{Results on R2}

After the direct result on $\calR_0$, this section is devoted to a complete computation of $\calR_2$ via the above method and is a more algebraic version than the similar result obtained in \cite[Thm~2.4]{IochMass17c}.

The computation of $\calR_0$ is straighforward thanks to \eqref{eq-Xk for u diagonal}:

\begin{lemma}
\begin{align*}
(4\pi)^{\,d/2}\calR_0 = X_{d/2, 0}[\bbbone] = u^{-d/2}.
\end{align*}
\end{lemma}

We plane to follow Proposition~\ref{prop-simplified-computation} to compute $\calR_2$ in $a_2(a,P)(x) = \tr\,[ a(x) \,\calR_2(x) ]$ and start with
\begin{align*}
(4\pi)^{\,d/2}\calR_2
&= 
X_{d/2+1, 2, \mu_1\mu_2}[K^{\mu_1} \otimes K^{\mu_2}] - X_{d/2, 1}[P]
.
\end{align*}
We perform the computation at the level of arguments using \eqref{eq-Xmuk-arguments-derivation}. Notice that all the used spectral operators $X_{\alpha, k}$ (and those appearing in the final result) are in the series $X_{d/2, 1}, X_{d/2+1, 2}, \dots, X_{d/2+k-1, k}$.

Consider first $-X_{d/2, 1}[P]$:
\begin{align*}
-P
& = 
H^{\nu_1\nu_2} \nabla^2_{\nu_1\nu_2} 
+ L^{\nu_1} \nabla_{\nu_1} 
+ q
\\
&\leadsto
\begin{aligned}[t]
&
- H^{\nu_1\nu_2} \nabla_{\!\nu_1} \otimes (\nabla_{\!\nu_2} H^{\mu_1\mu_2})
+ H^{\nu_1\nu_2} \nabla_{\!\nu_1} \mid A_{\nu_2}
- L^{\nu_1} \otimes (\nabla_{\nu_1} H^{\mu_1\mu_2})
+ L^{\nu_1} \mid A_{\nu_2}
+ q
\end{aligned}
\\
&\leadsto
\begin{aligned}[t]
&
- H^{\nu_1\nu_2} \otimes (\nabla^2_{\!\nu_1\nu_2} H^{\mu_1\mu_2})
+ H^{\nu_1\nu_2} \otimes (\nabla_{\!\nu_1} H^{\mu_3\mu_4}) \otimes (\nabla_{\!\nu_2} H^{\mu_1\mu_2})
+ H^{\nu_1\nu_2} \otimes (\nabla_{\!\nu_2} H^{\mu_1\mu_2}) \otimes (\nabla_{\!\nu_1} H^{\mu_3\mu_4})
\\
&
- H^{\nu_1\nu_2} \otimes (\nabla_{\!\nu_2} H^{\mu_1\mu_2}) \mid A_{\nu_1}
- H^{\nu_1\nu_2} \otimes (\nabla_{\!\nu_1} H^{\mu_1\mu_2}) \mid A_{\nu_2}
+ H^{\nu_1\nu_2} \mid (\partial_{\!\nu_1} A_{\nu_2} + A_{\nu_1} A_{\nu_2})
\\
&
- L^{\nu_1} \otimes (\nabla_{\!\nu_1} H^{\mu_1\mu_2})
+ L^{\nu_1} \mid A_{\nu_1}
+ q
\end{aligned}
\\
&\tonc
\begin{aligned}[t]
&
- g^{\nu_1\nu_2} g^{\mu_1\mu_2}\, u \otimes (\nabla^2_{\!\nu_1\nu_2} u)
+ \tfrac{1}{3} g^{\nu_1\nu_2} \big( \sumperm_{\nu_1, \nu_2} \tensor{R}{_{\nu_1}^{\mu_1}_{\, \nu_2}^{\mu_2}} \big)\, u \otimes u
+ 2 g^{\nu_1\nu_2} g^{\mu_1\mu_2} g^{\mu_3\mu_4}  \, u \otimes (\nabla_{\! \nu_1} u) \otimes (\nabla_{\! \nu_2} u)
\\
&
- 2 g^{\nu_1\nu_2} g^{\mu_1\mu_2} \, u \otimes (\nabla_{\!\nu_1} u) \mid A_{\nu_2}
+ g^{\nu_1\nu_2} \, u \mid (\partial_{\nu_1} A_{\nu_2} + A_{\nu_1} A_{\nu_2})
\\
&
- g^{\mu_1\mu_2} \, p^{\nu_1} \otimes (\nabla_{\!\nu_1} u)
- g^{\nu_1\sigma_1} g^{\mu_1\mu_2} \, (\nabla_{\!\sigma_1} u) \otimes (\nabla_{\!\nu_1} u)
+L^{\nu_1} \mid A_{\nu_1}
+ q.
\end{aligned}
\end{align*}

Similarly for $X_{d/2+1, 2, \mu_1\mu_2}[K^{\mu_1} \otimes K^{\mu_2}]$:
\begin{align*}
K^{\mu_1} \otimes K^{\mu_2}
& = 
- (L^{\mu_1} + 2 H^{\mu_1\nu_1} \nabla_{\!\nu_1}) \otimes (L^{\mu_2} + 2 H^{\mu_2\nu_2} \nabla_{\!\nu_2})
\\
& =
\begin{aligned}[t]
&
- L^{\mu_1} \otimes L^{\mu_2}
- 2 L^{\mu_1} \otimes H^{\mu_2\nu_2} \nabla_{\nu_2}
- 2 H^{\mu_1\nu_1} \nabla_{\!\nu_1} \otimes L^{\mu_2}
- 4 H^{\mu_1\nu_1} \nabla_{\!\nu_1} \otimes H^{\mu_2\nu_2} \nabla_{\!\nu_2}
\end{aligned}
\\
&\leadsto
\begin{aligned}[t]
&
- L^{\mu_1} \otimes L^{\mu_2}
+ 2 L^{\mu_1} \otimes H^{\mu_2\nu_2} \otimes (\nabla_{\!\nu_2} H^{\mu_3\mu_4})
- 2 L^{\mu_1} \otimes H^{\mu_2\nu_2} \mid A_{\nu_2}
\\
&- 2 H^{\mu_1\nu_1} \otimes (\nabla_{\!\nu_1} L^{\mu_2})
+ 2 H^{\mu_1\nu_1} \otimes (\nabla_{\!\nu_1} H^{\mu_3\mu_4}) \otimes L^{\mu_2}
\\
&+ 2 H^{\mu_1\nu_1} \otimes L^{\mu_2} \otimes (\nabla_{\!\nu_1} H^{\mu_3\mu_4})
- 2 H^{\mu_1\nu_1} \otimes L^{\mu_2} \mid A_{\nu_1}
\\
&
+ 4 H^{\mu_1\nu_1} \nabla_{\!\nu_1} \otimes H^{\mu_2\nu_2} \otimes (\nabla_{\! \nu_2} H^{\mu_3\mu_4})
- 4 H^{\mu_1\nu_1} \nabla_{\!\nu_1} \otimes H^{\mu_2\nu_2} \mid A_{\nu_2}
\end{aligned}
\\
&\leadsto
\begin{aligned}[t]
&
- L^{\mu_1} \otimes L^{\mu_2}
+ 2 L^{\mu_1} \otimes H^{\mu_2\nu_2} \otimes (\nabla_{\!\nu_2} H^{\mu_3\mu_4})
- 2 L^{\mu_1} \otimes H^{\mu_2\nu_2} \mid A_{\nu_2}
\\
&
- 2 H^{\mu_1\nu_1} \otimes (\nabla_{\!\nu_1} L^{\mu_2})
+ 2 H^{\mu_1\nu_1} \otimes (\nabla_{\!\nu_1} H^{\mu_3\mu_4}) \otimes L^{\mu_2}
+ 2 H^{\mu_1\nu_1} \otimes L^{\mu_2} \otimes (\nabla_{\!\nu_1} H^{\mu_3\mu_4})
- 2 H^{\mu_1\nu_1} \otimes L^{\mu_2} \mid A_{\nu_1}
\\
&
+ 4 H^{\mu_1\nu_1} \otimes (\nabla_{\!\nu_1} H^{\mu_2\nu_2}) \otimes (\nabla_{\!\nu_2} H^{\mu_3\mu_4}) %
+ 4 H^{\mu_1\nu_1} \otimes H^{\mu_2\nu_2} \otimes (\nabla^2_{\!\nu_1\nu_2} H^{\mu_3\mu_4}) %
\\
&
- 4 H^{\mu_1\nu_1} \otimes (\nabla_{\!\nu_1} H^{\mu_5\mu_6}) \otimes H^{\mu_2\nu_2} \otimes  (\nabla_{\!\nu_2} H^{\mu_3\mu_4}) 
- 4 H^{\mu_1\nu_1} \otimes H^{\mu_2\nu_2} \otimes (\nabla_{\!\nu_1} H^{\mu_5\mu_6}) \otimes  (\nabla_{\!\nu_2} H^{\mu_3\mu_4}) 
\\
&
- 4 H^{\mu_1\nu_1} \otimes H^{\mu_2\nu_2} \otimes  (\nabla_{\!\nu_2} H^{\mu_3\mu_4}) \otimes (\nabla_{\!\nu_1} H^{\mu_5\mu_6}) 
+ 4 H^{\mu_1\nu_1} \otimes H^{\mu_2\nu_2} \otimes  (\nabla_{\!\nu_2} H^{\mu_3\mu_4}) \mid A_{\nu_1}
\\
&
- 4 H^{\mu_1\nu_1} \otimes (\nabla_{\!\nu_1} H^{\mu_2\nu_2}) \mid A_{\nu_2}
+ 4 H^{\mu_1\nu_1} \otimes (\nabla_{\!\nu_1} H^{\mu_3\mu_4}) \otimes H^{\mu_2\nu_2} \mid A_{\nu_2}
\\
&
+ 4 H^{\mu_1\nu_1} \otimes H^{\mu_2\nu_2} \otimes (\nabla_{\!\nu_1} H^{\mu_3\mu_4}) \mid A_{\nu_2}
- 4 H^{\mu_1\nu_1} \otimes H^{\mu_2\nu_2} \mid (\partial_{\!\nu_1} A_{\nu_2} + A_{\nu_1} A_{\nu_2})
\end{aligned}
\\
&\tonc
\begin{aligned}[t]
& 
- p^{\mu_1} \otimes p^{\mu_2}
- g^{\mu_2\sigma_2} \, p^{\mu_1} \otimes (\nabla_{\!\sigma_2} u)
- g^{\mu_1\sigma_1} \, (\nabla_{\sigma_1} u) \otimes p^{\mu_2}
- g^{\mu_1\sigma_1} g^{\mu_2\sigma_2} \,(\nabla_{\!\sigma_1} u) \otimes (\nabla_{\!\sigma_2} u)
\\
&
+ 2 g^{\mu_2\nu_2} g^{\mu_3\mu_4}\, p^{\mu_1} \otimes u \otimes (\nabla_{\!\nu_2} u)
+ 2 g^{\mu_1\sigma_1} g^{\mu_2\nu_2} g^{\mu_3\mu_4}\, (\nabla_{\!\sigma_1} u) \otimes u \otimes (\nabla_{\!\nu_2} u)
- 2 g^{\mu_2\nu_2} \, L^{\mu_1} \otimes u \mid A_{\nu_2}
\\
& 
- 2 g^{\mu_1\nu_1} \, u \otimes (\nabla_{\!\nu_1} p^{\mu_2})
- 2 g^{\mu_1\nu_1} g^{\mu_2\sigma_2} \, u \otimes (\nabla^2_{\!\nu_1\sigma_2} u)
+ \tfrac{4}{3} g^{\mu_1\nu_1} \tensor{\Ric}{_{\nu_1}^{\mu_2}} \, u \otimes u
\\
&
+ 2 g^{\mu_1\nu_1} g^{\mu_3\mu_4} \, u \otimes (\nabla_{\!\nu_1} u) \otimes p^{\mu_2}
+ 2 g^{\mu_1\nu_1} g^{\mu_3\mu_4} g^{\mu_2\sigma_2} \, u \otimes (\nabla_{\!\nu_1} u) \otimes (\nabla_{\!\sigma_2} u)
\\
&
+ 2 g^{\mu_1\nu_1} g^{\mu_3\mu_4} \, u \otimes p^{\mu_2} \otimes (\nabla_{\!\nu_1} u)
+ 2 g^{\mu_1\nu_1} g^{\mu_2\sigma_2} g^{\mu_3\mu_4}\, u \otimes (\nabla_{\!\sigma_2} u) \otimes (\nabla_{\!\nu_1} u)
- 2 g^{\mu_1\nu_1} \, u \otimes L^{\mu_2} \mid A_{\nu_1}
\\
&
+ 4 g^{\mu_1\nu_1} g^{\mu_2\nu_2} g^{\mu_3\mu_4}\, u \otimes (\nabla_{\!\nu_1} u) \otimes (\!\nabla_{\nu_2} u)
+ 4 g^{\mu_1\nu_1} g^{\mu_2\nu_2} g^{\mu_3\mu_4} \, u \otimes u \otimes (\nabla^2_{\!\nu_1\nu_2} u)
\\
&
- \tfrac{4}{3} g^{\mu_1\nu_1} g^{\mu_2\nu_2} \big( \sumperm_{\nu_1, \nu_2} \tensor{R}{_{\nu_1}^{\mu_3}_{\, \nu_2}^{\mu_4}} \big) \, u \otimes u \otimes u
- 4 g^{\mu_1\nu_1} g^{\mu_5\mu_6} g^{\mu_2\nu_2} g^{\mu_3\mu_4} \, u \otimes (\nabla_{\!\nu_1} u) \otimes u \otimes  (\nabla_{\!\nu_2} u)
\\
&
- 4 g^{\mu_1\nu_1} g^{\mu_2\nu_2} g^{\mu_5\mu_6} g^{\mu_3\mu_4} \, u \otimes u \otimes (\nabla_{\!\nu_1} u) \otimes  (\nabla_{\!\nu_2} u)
- 4 g^{\mu_1\nu_1} g^{\mu_2\nu_2} g^{\mu_3\mu_4} g^{\mu_5\mu_6} \, u \otimes u \otimes  (\nabla_{\!\nu_2} u) \otimes (\nabla_{\!\nu_1} u)
\\
&
+ 4 g^{\mu_1\nu_1} g^{\mu_2\nu_2} g^{\mu_3\mu_4} \, u \otimes u \otimes  (\nabla_{\!\nu_2} u) \mid A_{\nu_1}
-4 g^{\mu_1\nu_1} g^{\mu_2\nu_2} \, u \otimes (\nabla_{\!\nu_1} u) \mid A_{\nu_2}
\\
&
+ 4 g^{\mu_1\nu_1} g^{\mu_3\mu_4} g^{\mu_2\nu_2} \, u \otimes (\nabla_{\!\nu_1} u) \otimes u \mid A_{\nu_2}
+ 4 g^{\mu_1\nu_1} g^{\mu_2\nu_2} g^{\mu_3\mu_4} \, u \otimes u \otimes  (\nabla_{\!\nu_1} u) \mid A_{\nu_2}
\\
&
- 4 g^{\mu_1\nu_1} g^{\mu_2\nu_2} \, u \otimes u \mid (\partial_{\nu_1} A_{\nu_2} + A_{\nu_1} A_{\nu_2}).
\end{aligned}
\end{align*}

We can now apply the operators $X_{d/2, 1}$, $X_{d/2+1, 2, \mu_1\mu_2}$, $X_{d/2+2, 3, \mu_1\mu_2\mu_3\mu_4}$ and $X_{d/2+3, 4, \mu_1\mu_2\mu_3\mu_4\mu_5\mu_6}$ according to the obtained arguments, use \eqref{eq-contraction-g-Xmu} when necessary, and finally collect all the terms.

Let us first collect the terms with $Q[A] \neq \bbbone$ and show that they all cancel thanks to \eqref{eq-Xu}. \\
The terms with $(\partial_{\nu_1} A_{\nu_2} + A_{\nu_1} A_{\nu_2})$ do not contribute:
\begin{align*}
g^{\nu_1\nu_2} X_{d/2, 1}[u] - 4 g^{\mu_1\nu_1} g^{\mu_2\nu_2} X_{d/2+1, 2, \mu_1\mu_2}[u \otimes u]
&=
g^{\nu_1\nu_2} \big( X_{d/2, 1}[u] - 2 X_{d/2+1, 2}[u \otimes u] \big)
= 0.
\end{align*}
The sum of terms containing $L^\mu$ and $A_\nu$ vanishes:
\begin{multline*}
\big( 
X_{d/2, 1}[L^{\nu_1}]  
- 2 g^{\mu_2\nu_1} X_{d/2+1, 2, \mu_1\mu_2}[L^{\mu_1} \otimes u] 
- 2 g^{\mu_1\nu_1} X_{d/2+1, 2, \mu_1\mu_2}[u \otimes L^{\mu_2}] 
\big) \,A_{\nu_1}
\\
\begin{aligned}[t]
& = \big( 
X_{d/2, 1}[L^{\nu_1}]
- X_{d/2+1, 2}[L^{\nu_1} \otimes u]
- X_{d/2+1, 2}[u \otimes L^{\nu_1}]
\big) \,A_{\nu_1}
= 0.
\end{aligned}
\end{multline*}
Finally the contribution of terms containing $\nnu \,u$ and $A_\nu$ also vanishes:
\begin{multline*}
\begin{multlined}[t]
\big(
 - 2 g^{\nu_1\nu_2} g^{\mu_1\mu_2} X_{d/2+1, 2, \mu_1\mu_2}[u \otimes (\nabla_{\!\nu_1} u)] 
- 4 g^{\mu_1\nu_1} g^{\mu_2\nu_2} X_{d/2+1, 2, \mu_1\mu_2}[u \otimes (\nabla_{\!\nu_1} u)] 
\\
 + 4 g^{\mu_1\nu_1} g^{\mu_2\nu_2} g^{\mu_3\mu_4} X_{d/2+2, 3, \mu_1\mu_2\mu_3\mu_4}[u \otimes (\nabla_{\!\nu_1} u) \otimes u]
 + 8 g^{\mu_1\nu_1} g^{\mu_2\nu_2} g^{\mu_3\mu_4} X_{d/2+2, 3, \mu_1\mu_2\mu_3\mu_4}[u \otimes u \otimes (\nabla_{\!\nu_1} u)]
\big)\, A_{\nu_2}\!
\end{multlined}
\\
\begin{aligned}[t]
= g^{\nu_1\nu_2} (d + 2) \big(
		\begin{aligned}[t]
			& -  X_{d/2+1, 2}[u \otimes (\nabla_{\!\nu_1} u)]
			 + X_{d/2+2, 3}[u \otimes (\nabla_{\nu_1} u )\otimes u]
				+ 2 X_{d/2+2, 3}[u \otimes u \otimes (\nabla_{\!\nu_1} u)]
			\big)\, A_{\nu_2}= 0.
		\end{aligned}	
\end{aligned}
\end{multline*}
There are 3 terms which contain tensor products of $u$ only:
\begin{align*}
&+\tfrac{1}{3} \,g^{\nu_1\nu_2} \big( \sumperm_{\nu_1, \nu_2} \tensor{R}{_{\nu_1}^{\mu_1}_{\, \nu_2}^{\mu_2}} \big) \,
X_{d/2+1, 2, \mu_1\mu_2}[u \otimes u]
= \tfrac{1}{6} \,g_{\mu_1\mu_2} g^{\nu_1\nu_2} \big( \sumperm_{\nu_1, \nu_2} \tensor{R}{_{\nu_1}^{\mu_1}_{\, \nu_2}^{\mu_2}} \big) \, X_{d/2+1, 2}[u \otimes u]
= -\tfrac{1}{3}\, \SC \, X_{d/2+1, 2}[u \otimes u],
\\[3mm]
&+\tfrac{4}{3} \,g^{\mu_1\nu_1} \tensor{\Ric}{_{\nu_1}^{\mu_2}} \, X_{d/2+1, 2, \mu_1\mu_2}[u \otimes u]
= \tfrac{2}{3}\, g_{\mu_1\mu_2}  g^{\mu_1\nu_1} \tensor{\Ric}{_{\nu_1}^{\mu_2}} \, X_{d/2+1, 2}[u \otimes u]
= \tfrac{2}{3}\, \SC \, X_{d/2+1, 2}[u \otimes u],
\\[3mm]
&- \tfrac{4}{3} \,g^{\mu_1\nu_1} g^{\mu_2\nu_2} \big( \sumperm_{\nu_1, \nu_2} \tensor{R}{_{\nu_1}^{\mu_3}_{\, \nu_2}^{\mu_4}} \big) \,X_{d/2+2, 3, \mu_1\mu_2\mu_3\mu_4}[u \otimes u \otimes u]
= - \tfrac{4}{3} \,G_{\mu_1\mu_2\mu_3\mu_4} \big( \sumperm_{\mu_1, \mu_2} \tensor{R}{^{\mu_1\mu_3\, \mu_2\mu_4}} \big) \, X_{d/2+2,3}[u \otimes u \otimes u]
= 0,
\end{align*}
by the complete symmetry of $G_{\mu_1\mu_2\mu_3\mu_4}$ and the skew symmetry of the first and second couples of indices in $\tensor{R}{^{\mu_1\mu_3\, \mu_2\mu_4}}$. \\
Since by \eqref{eq-X on only h}, $X_{d/2+1, 2}[u \otimes u] = \tfrac{1}{2}\, u^{-d/2+1}$, this amounts to $\tfrac{1}{6}\, \SC \, u^{-d/2+1} = \tfrac{1}{6}\, \SC \, X_{d/2,1}[u]$.
\\
The only term in $q$ is $X_{d/2, 1}[q]$. \\
The sum of 3 terms containing $\nabla^2_{\!\nu_1\nu_2} u$ is
\begin{align*}
&\begin{aligned}[t]
& - g^{\nu_1\nu_2} g^{\mu_1\mu_2} \, X_{d/2+1, 2, \mu_1\mu_2}[u \otimes (\nabla^2_{\!\nu_1\nu_2} u)]
- 2 g^{\mu_1\nu_1} g^{\mu_2\nu_2} \, X_{d/2+1, 2, \mu_1\mu_2}[u \otimes (\nabla^2_{\!\nu_1\nu_2} u)]
\\
&\quad\quad
+ 4 g^{\mu_1\nu_1} g^{\mu_2\nu_2} g^{\mu_3\mu_4} \, X_{d/2+2, 3, \mu_1\mu_2\mu_3\mu_4}[u \otimes u \otimes ( \nabla^2_{\!\nu_1\nu_2} u)]
\end{aligned}
\\[1mm]
&\quad\quad\quad\quad\quad\quad
=\begin{aligned}[t]
& 
- \tfrac{1}{2}d g^{\nu_1\nu_2} \, X_{d/2+1, 2}[u \otimes (\nabla^2_{\!\nu_1\nu_2} u)]
- g^{\nu_1\nu_2} \, X_{d/2+1, 2}[u \otimes (\nabla^2_{\!\nu_1\nu_2} u)]
\\
&\quad\quad
+ (d+2) g^{\nu_1\nu_2} \, X_{d/2+2, 3}[u \otimes u \otimes (\nabla^2_{\!\nu_1\nu_2} u)]
\end{aligned}
\\[1mm]
&\quad\quad\quad\quad\quad\quad
=\begin{aligned}[t]
&
- \tfrac{1}{2}(d+2) g^{\nu_1\nu_2} \, X_{d/2+1,2}[u \otimes (\nabla^2_{\!\nu_1\nu_2} u)]
+ (d+2) g^{\nu_1\nu_2} \, X_{d/2+2,3}[u \otimes u \otimes (\nabla^2_{\!\nu_1\nu_2} u)].
\end{aligned}
\end{align*}
The sum of 10 terms in $\nabla_{\!\nu_1} u$ and $\nabla_{\!\nu_2} u$ is:
\begin{align*}
&\begin{aligned}[t]
&
2 g^{\nu_1\nu_2} g^{\mu_1\mu_2} g^{\mu_3\mu_4}  \,X_{d/2+2, 3, \mu_1\mu_2\mu_3\mu_4}[u \otimes (\nabla_{\!\nu_1} u) \otimes (\nabla_{\nu_2} u)]
- g^{\nu_1\nu_2} g^{\mu_1\mu_2} \, X_{d/2+1, 2, \mu_1\mu_2}[(\nabla_{\!\nu_1} u) \otimes (\nabla_{\!\nu_2} u)]
\\
&\quad
- g^{\mu_1\nu_1} g^{\mu_2\nu_2} \, X_{d/2+1, 2, \mu_1\mu_2}[(\nabla_{\!\nu_1} u) \otimes (\nabla_{\nu_2} u)]
+ 2 g^{\mu_1\nu_1} g^{\mu_2\nu_2} g^{\mu_3\mu_4} \, X_{d/2+2, 3, \mu_1\mu_2\mu_3\mu_4}[(\nabla_{\!\nu_1} u) \otimes u \otimes (\nabla_{\!\nu_2} u)]
\\
&\quad
+ 2 g^{\mu_1\nu_1} g^{\mu_2\nu_2} g^{\mu_3\mu_4} \, X_{d/2+2, 3, \mu_1\mu_2\mu_3\mu_4}[u \otimes ( \nabla_{\!\nu_1} u )\otimes (\nabla_{\!\nu_2} u)]
+ 2 g^{\mu_1\nu_1} g^{\mu_2\nu_2} g^{\mu_3\mu_4} \, X_{d/2+2, 3, \mu_1\mu_2\mu_3\mu_4}[u \otimes (\nabla_{\!\nu_2} u) \otimes (\nabla_{\!\nu_1} u)]
\\
&\quad
+ 4 g^{\mu_1\nu_1} g^{\mu_2\nu_2} g^{\mu_3\mu_4} \, X_{d/2+2, 3, \mu_1\mu_2\mu_3\mu_4}[u \otimes (\nabla_{\!\nu_1} u) \otimes ( \nabla_{\!\nu_2} u)]
\\
&\quad
- 4 g^{\mu_1\nu_1} g^{\mu_2\nu_2} g^{\mu_3\mu_4} g^{\mu_5\mu_6} \, X_{d/2+3, 4, \mu_1\mu_2\mu_3\mu_4\mu_5\mu_6}[u \otimes (\nabla_{\!\nu_1} u) \otimes u \otimes  (\nabla_{\!\nu_2} u)]
\\
&\quad
- 4 g^{\mu_1\nu_1} g^{\mu_2\nu_2} g^{\mu_3\mu_4} g^{\mu_5\mu_6} \, X_{d/2+3, 4, \mu_1\mu_2\mu_3\mu_4\mu_5\mu_6}[u \otimes u \otimes (\nabla_{\!\nu_1} u) \otimes  (\nabla_{\!\nu_2} u)]
\\
&\quad
- 4 g^{\mu_1\nu_1} g^{\mu_2\nu_2} g^{\mu_3\mu_4} g^{\mu_5\mu_6} \, X_{d/2+3, 4, \mu_1\mu_2\mu_3\mu_4\mu_5\mu_6}[u \otimes u \otimes (\nabla_{\!\nu_2} u )\otimes (\nabla_{\!\nu_1} u)]
\end{aligned}
\\[1mm]
&\quad\quad\quad
= \begin{aligned}[t]
&\tfrac{1}{2}d(d+2)g^{\nu_1\nu_2} \, X_{d/2+2, 3}[u \otimes (\nabla_{\!\nu_1} u) \otimes (\nabla_{\!\nu_2} u)]
- \tfrac{1}{2}d g^{\nu_1\nu_2} \, X_{d/2+1, 2}[(\nabla_{\!\nu_1} u) \otimes (\nabla_{\!\nu_2} u)]
\\
&\quad
- \tfrac{1}{2} g^{\nu_1\nu_2} \, X_{d/2+1, 2}[(\nabla_{\nu_1} u) \otimes (\nabla_{\!\nu_2} u)]
+ \tfrac{1}{2} (d+2) g^{\nu_1\nu_2} \, X_{d/2+2, 3}[(\nabla_{\nu_1} u) \otimes u \otimes (\nabla_{\nu_2} u)]
\\
&\quad
+ 2(d+2) g^{\nu_1\nu_2} \, X_{d/2+2, 3}[u \otimes (\nabla_{\!\nu_1} u) \otimes ( \nabla_{\nu_2} u)]
- \tfrac{1}{2}(d+2)(d+4)  g^{\nu_1\nu_2} \, X_{d/2+3, 4}[u \otimes (\nabla_{\!\nu_1} u) \otimes u \otimes  (\nabla_{\!\nu_2} u)]
\\
&\quad
- (d+2)(d+4) g^{\nu_1\nu_2} \, X_{d/2+3, 4}[u \otimes u \otimes (\nabla_{\!\nu_1} u) \otimes  (\nabla_{\!\nu_2} u)]
\end{aligned}
\\[1mm]
&\quad\quad\quad
= \begin{aligned}[t]
&
- \tfrac{1}{2}(d+1) g^{\nu_1\nu_2} \, X_{d/2+1, 2}[(\nabla_{\!\nu_1} u) \otimes (\nabla_{\!\nu_2} u)]
+ \tfrac{1}{2} (d+2)(d+4)g^{\nu_1\nu_2} \, X_{d/2+2, 3}[u \otimes (\nabla_{\!\nu_1} u) \otimes (\nabla_{\!\nu_2} u)]
\\
&\quad
+ \tfrac{1}{2}(d+2)  g^{\nu_1\nu_2} \, X_{d/2+2, 3}[(\nabla_{\!\nu_1} u) \otimes u \otimes (\nabla_{\!\nu_2} u)]
- \tfrac{1}{2}(d+2)(d+4)  g^{\nu_1\nu_2} \, X_{d/2+3, 4}[u \otimes (\nabla_{\!\nu_1} u) \otimes u \otimes (\nabla_{\!\nu_2} u)]
\\
&\quad
- (d+2)(d+4) g^{\nu_1\nu_2} \, X_{d/2+3, 4}[u \otimes u \otimes (\nabla_{\!\nu_1} u) \otimes (\nabla_{\!\nu_2} u)].
\end{aligned}
\end{align*}
The only term in $\nnu\, p^{\mu}$ is
\begin{align*}
- 2 g^{\mu_1\nu_1} \, X_{d/2+1, 2, \mu_1\mu_2}[u \otimes (\nabla_{\!\nu_1} p^{\mu_2})]
&= - X_{d/2+1, 2}[u \otimes (\nmu p^{\mu})].
\end{align*}
The 4 terms in $p^\mu$ and $\nnu \,u$ are
\begin{align*}
& - g^{\mu_1\mu_2} \, X_{d/2+1, 2, \mu_1\mu_2}[p^{\nu_1} \otimes (\nabla_{\!\nu_1} u)]
- g^{\mu_2\nu_1} \, X_{d/2+1, 2, \mu_1\mu_2}[p^{\mu_1} \otimes (\nabla_{\!\nu_1} u)]
\\
&\quad
+ 2 g^{\mu_2\nu_2} g^{\mu_3\mu_4} \, X_{d/2+2, 3, \mu_1\mu_2\mu_3\mu_4}[p^{\mu_1} \otimes u \otimes (\nabla_{\!\nu_2} u)]
+ 2 g^{\mu_1\nu_1} g^{\mu_3\mu_4} \, X_{d/2+2, 3, \mu_1\mu_2\mu_3\mu_4}[u \otimes p^{\mu_2} \otimes (\nabla_{\!\nu_1} u)]
\\[1mm]
&\quad\quad
= 
\begin{aligned}[t]
& 
- \tfrac{1}{2}d \, X_{d/2+1, 2}[p^{\nu_1} \otimes (\nabla_{\!\nu_1} u)]
- \tfrac{1}{2} \, X_{d/2+1, 2}[p^{\nu_1} \otimes (\nabla_{\!\nu_1} u)]
\\
&\quad\quad
+ \tfrac{1}{2}(d+2) \, X_{d/2+2, 3}[p^{\nu_1} \otimes u \otimes (\nabla_{\!\nu_1} u)]
+ \tfrac{1}{2}(d+2) \, X_{d/2+2, 3}[u \otimes p^{\nu_1} \otimes (\nabla_{\!\nu_1} u)]
\end{aligned}
\\[1mm]
&\quad\quad
= 
- \tfrac{1}{2}(d+1) \, X_{d/2+1, 2}[p^{\nu_1} \otimes (\nabla_{\!\nu_1} u)]
+ \tfrac{1}{2}(d+2) \, X_{d/2+2, 3}[p^{\nu_1} \otimes u \otimes (\nabla_{\!\nu_1} u)]
+ \tfrac{1}{2}(d+2) \, X_{d/2+2, 3}[u \otimes p^{\nu_1} \otimes (\nabla_{\!\nu_1} u)].
\end{align*}
The 2 terms in $\nnu u$ and $p^\mu$ are
\begin{align*}
&\begin{aligned}[t]
& - g^{\mu_1\nu_1} \, X_{d/2+1, 2, \mu_1\mu_2}[(\nabla_{\!\nu_1} u) \otimes p^{\mu_2}]
+ 2 g^{\mu_1\nu_1} g^{\mu_3\mu_4} \, X_{d/2+2, 3, \mu_1\mu_2\mu_3\mu_4}[u \otimes (\nabla_{\!\nu_1} u) \otimes p^{\mu_2}]
\end{aligned}
\\[1mm]
&\quad\quad\quad\quad\quad\quad
= \begin{aligned}[t]
&
- \tfrac{1}{2} \, X_{d/2+1, 2}[(\nabla_{\!\nu_1} u) \otimes p^{\nu_1}]
+ \tfrac{1}{2}(d+2) \, X_{d/2+2, 3}[u \otimes (\nabla_{\!\nu_1} u) \otimes p^{\nu_1}].
\end{aligned}
\end{align*}
Finally, the term in $p^{\mu_1}$ and $p^{\mu_2}$ is
\begin{align*}
- X_{d/2+1, 2, \mu_1\mu_2}[p^{\mu_1} \otimes p^{\mu_2}]
&= - \tfrac{1}{2} g_{\mu_1\mu_2} \, X_{d/2+1, 2}[p^{\mu_1} \otimes p^{\mu_2}].
\end{align*}

\smallskip
Since the computation has been performed in normal coordinates, using \eqref{eq-hnabla1nc u}, \eqref{eq-hnabla2nc u} and \eqref{eq-hnabla1nc p}, one can replace the gauge covariant derivative $\nmu$ by the total derivative $\hnmu$ to get a fully covariant expression:
\begin{align}
\label{eq-R-upq-operators}
(4\pi)^{\,d/2}&\calR_2
= 
 +\tfrac{1}{6} \SC \, X_{d/2, 1}[u]
+ X_{d/2, 1}[q]
 - \tfrac{1}{2} (d+2)\, g^{\mu\nu} \, X_{d/2+1, 2}[u \otimes (\hnabla^2_{\mu\nu} u)]
\\
&+ (d+2) \, g^{\mu\nu} \, X_{d/2+2, 3}[u \otimes u \otimes (\hnabla^2_{\mu\nu} u)]
 - X_{d/2+1, 2}[u \otimes (\hnmu p^\mu)]
- \tfrac{1}{2}(d+1) \, g^{\mu\nu} \, X_{d/2+1, 2}[(\hnmu u )\otimes (\hnnu u)]  \nonumber
\\
&
+ \tfrac{1}{2}(d+2)(d+4)\, g^{\mu\nu} \, X_{d/2+2, 3}[u \otimes (\hnmu u) \otimes (\hnnu u)]
 + \tfrac{1}{2}(d+2) \, g^{\mu\nu} \, X_{d/2+2, 3}[(\hnmu u) \otimes u \otimes (\hnnu u)]  \nonumber
\\
& - (d+2)(d+4)\, g^{\mu\nu} \, X_{d/2+3, 4}[u \otimes u \otimes (\hnmu u) \otimes (\hnnu u)]
 - \tfrac{1}{2}(d+2)(d+4) \, g^{\mu\nu} \, X_{d/2+3, 4}[u \otimes (\hnmu u \otimes u) \otimes (\hnnu u)]  \nonumber
\\
& - \tfrac{1}{2}(d+1) \, X_{d/2+1, 2}[p^\mu \otimes (\hnmu u)]
+ \tfrac{1}{2}(d+2) \, X_{d/2+2, 3}[u \otimes p^\mu \otimes (\hnmu u)]
 + \tfrac{1}{2}(d+2) \, X_{d/2+2, 3}[p^\mu \otimes u \otimes (\hnmu u)]  \nonumber
 \\
& - \tfrac{1}{2} \, X_{d/2+1, 2}[(\hnmu u) \otimes p^\mu ]
+ \tfrac{1}{2}(d+2) \, X_{d/2+2, 3}[u \otimes (\hnmu u) \otimes p^\mu]
- \tfrac{1}{2} \, g_{\mu\nu} \, X_{d/2+1, 2}[p^\mu \otimes p^\nu ] .  \nonumber
\end{align}
In this expression the contribution of a few terms can be simplified using Lemma~\ref{lem-Xu} like:
\begin{align*}
&
- \tfrac{1}{2} (d+2)\, g^{\mu\nu} \, X_{d/2+1, 2}[u \otimes (\hnabla^2_{\mu\nu} u)]
+ (d+2) \, g^{\mu\nu} \, X_{d/2+2, 3}[u \otimes u \otimes (\hnabla^2_{\mu\nu} u)] \\
&\quad\quad\quad\quad\quad\quad
=  - \tfrac{1}{2}(d+2) \, g^{\mu\nu} \, X_{d/2+2, 3}[ u \otimes (\hnabla^2_{\mu\nu} u) \otimes u] ,
\\[3mm]
&
+\tfrac{1}{2}(d+2)(d+4) \, g^{\mu\nu} \, X_{d/2+2, 3}[u \otimes (\hnmu u) \otimes (\hnnu u)] 
- (d+2)(d+4) \, g^{\mu\nu} \, X_{d/2+3, 4}[u \otimes u \otimes (\hnmu u) \otimes (\hnnu u)] 
\\
&\quad\quad(
- \tfrac{1}{2} (d+2)(d+4)\, g^{\mu\nu} \, X_{d/2+3, 4}[u \otimes (\hnmu u )\otimes u \otimes (\hnnu u)] 
\\
&\quad\quad\quad\quad\quad\quad
= \tfrac{1}{2}(d+2)(d+4) \, g^{\mu\nu} \, X_{d/2+3, 4}[u \otimes (\hnmu u)\otimes (\hnnu u) \otimes u],
\\[3mm]
& 
- \tfrac{1}{2}(d+1) \, X_{d/2+1, 2}[p^\mu \otimes (\hnmu u)]  
+ \tfrac{1}{2} (d+2)\, X_{d/2+2, 3}[u \otimes p^\mu \otimes (\hnmu u)] 
+ \tfrac{1}{2}(d+2) \, X_{d/2+2, 3}[p^\mu \otimes u \otimes (\hnmu u)] 
\\
 &\quad\quad\quad\quad\quad\quad 
 = \tfrac{1}{2} \, X_{d/2+1, 2}[p^\mu \otimes (\hnmu u)] 
 - \tfrac{1}{2}(d+2) \, X_{d/2+2, 3}[ p^\mu \otimes (\hnmu u) \otimes u] .
\end{align*}
Finally, we get:

\begin{theorem}
The section $\calR_2$ of $\End(V)$ is
\begin{align}
\label{eq-R-upq-operators-simplified}
(4\pi)^{\,d/2} \calR_2= 
& +\tfrac{1}{6} \SC \, X_{d/2, 1}[u]
+ X_{d/2, 1}[q] 
 - \tfrac{1}{2}(d+2) \, g^{\mu\nu} \, X_{d/2+2, 3}[ u \otimes (\hnabla^2_{\mu\nu} u) \otimes u ]
\\
& 
- \tfrac{1}{2}(d+1) \, g^{\mu\nu} \, X_{d/2+1, 2}[(\hnmu u) \otimes (\hnnu u)]
+ \tfrac{1}{2}(d+2) \, g^{\mu\nu} \, X_{d/2+2, 3}[(\hnmu u) \otimes u \otimes (\hnnu u)]  \nonumber
\\
&+\tfrac{1}{2}(d+2)(d+4) \, g^{\mu\nu} \, X_{d/2+3, 4}[u \otimes (\hnmu u) \otimes (\hnnu u) \otimes u]
+\tfrac{1}{2} \,X_{d/2+1, 2}[p^\mu \otimes (\hnmu u)]  
- X_{d/2+1, 2}[u \otimes (\hnmu p^\mu)]  \nonumber
\\
& 
-\tfrac{1}{2}(d+2)\, X_{d/2+2, 3}[ p^\mu \otimes (\hnmu u) \otimes u]
- \tfrac{1}{2} \, X_{d/2+1, 2}([\hnmu u) \otimes p^\mu ]
+ \tfrac{1}{2}(d+2) \, X_{d/2+2, 3}[u \otimes (\hnmu u) \otimes p^\mu]   \nonumber
\\
& 
- \tfrac{1}{2} \, g_{\mu\nu} \, X_{d/2+1, 2}[p^\mu \otimes p^\nu ]
.    \nonumber
\end{align}
\end{theorem}
A lengthy computation shows that this is compatible with \cite[Thm~2.4]{IochMass17c}.

According to Lemma~\ref{lem-Xu}, the writing of \eqref{eq-R-upq-operators-simplified} is not unique. 
For instance, using
\begin{align*}
&
- \tfrac{1}{2}(d+1) \, g^{\mu\nu} \, X_{d/2+1, 2}[(\hnmu u) \otimes (\hnnu u)]
+ \tfrac{1}{2}(d+2) \, g^{\mu\nu} \, X_{d/2+2, 3}[(\hnmu u) \otimes u \otimes (\hnnu u)]
\\
&\qquad
\begin{aligned}[t]
&
\quad = \tfrac{1}{2} \, g^{\mu\nu} \, X_{d/2+1, 2}[(\hnmu u) \otimes (\hnnu u)]
- \tfrac{1}{2} (d+2)\, g^{\mu\nu} \, X_{d/2+2, 3}[u \otimes (\hnmu u) \otimes (\hnnu u)]
\\
&
\quad\quad
- \tfrac{1}{2}(d+2) \, g^{\mu\nu} \, X_{d/2+2, 3}[(\hnmu u) \otimes (\hnnu u) \otimes u],
\end{aligned}
\end{align*}
this expression can be factorized as: 
\begin{corollary}
\label{cor-resultR2-factorized}
\begin{align*}
(4\pi)^{\,d/2} \calR_2
&=
\begin{aligned}[t]
& +\tfrac{1}{6} \SC \, X_{d/2, 1}[u]
+ X_{d/2, 1}[q]
- X_{d/2+1, 2}[ u \otimes (\hnmu p^\mu) ]
+ \tfrac{1}{2} \, g^{\mu\nu} \, X_{d/2+1, 2}[ (\hnmu u + p_\mu) \otimes (\hnnu \,u - p_\nu) ]
\\
& - \tfrac{1}{2}(d+2) \, g^{\mu\nu} \, X_{d/2+2, 3}[ u \otimes (\hnabla^2_{\mu\nu} u) \otimes u ]
- \tfrac{1}{2}(d+2) g^{\mu\nu} \, X_{d/2+2, 3}[ u \otimes (\hnmu u) \otimes (\hnnu u - p_\nu) ]
\\
& - \tfrac{1}{2}(d+2) g^{\mu\nu} \, X_{d/2+2, 3}[ (\hnmu u + p_\mu) \otimes (\hnnu\, u)  \otimes u ]
\\
& 
+ \tfrac{1}{2}(d+2)(d+4) \, g^{\mu\nu} \, X_{d/2+3, 4}[ u \otimes (\hnmu u) \otimes (\hnnu u) \otimes u ]
.
\end{aligned}
\end{align*}
\end{corollary}
We do not know if this proposed factorization has some structural origin. \\
As explained just after Lemma~\ref{lem-N=0}, it can be useful to rewrite this result in terms of $N^\nu$:

\begin{corollary}
\label{cor-resultR2-code}
\begin{align*}
(4\pi)^{\,d/2} \calR_2 = {} 
 &+\tfrac{1}{6} \, \SC \, X_{d/2, 1}[ u ] +X_{d/2, 1}[ q ] +g^{\mu\nu} \,X_{d/2+1, 2}[ u \otimes (\hnabla^2_{\!\mu\nu} u) ]-\tfrac{1}{2}(d+2) \, g^{\mu\nu} \, X_{d/2+2, 3}[ u \otimes (\hnabla^2_{\!\mu\nu} u) \otimes u ] 
  \\
 & -(d +2) \, g^{\mu\nu} \, X_{d/2+2, 3}[ u \otimes (\hnabla_{\!\mu} u) \otimes (\hnabla_{\!\nu} u) ] + \tfrac{1}{2}(d+2)(d+4)\, g^{\mu\nu} \, X_{d/2+3, 4}[ u \otimes (\hnabla_{\!\mu} u) \otimes (\hnabla_{\nu} u) \otimes u ]   \\
 & +X_{d/2+1, 2}[ N^{\mu} \otimes (\hnabla_{\mu} u) ] 
  -\tfrac{1}{2}(d+2) \, X_{d/2+2, 3}[ N^{\mu} \otimes (\hnabla_{\!\mu}u) \otimes u ] 
  +\tfrac{1}{2}(d +2) \,X_{d/2+2, 3}[ u \otimes (\hnabla_{\!\mu} u) \otimes N^{\mu} ] \\
 & -\tfrac{1}{2} \, g_{\mu\nu} \, X_{d/2+1, 2}[ N^{\mu} \otimes N^{\nu} ] 
  -X_{d/2+1, 2}[ u \otimes (\hnabla_{\mu} N^{\mu}) ]  
  .
\end{align*}
\end{corollary}

\section{The code}
\label{The code}

The computation of $\calR_2$ exposed in Section~\ref{Results on R2} shows that the simplified method summarized in  Proposition~\ref{prop-simplified-computation} consists to apply a set of (mainly algebraic) rules at the level of the arguments of operators. This is to be contrasted with other methods based on more analytical properties of the heat coefficients, where for instance all possible expressions (based on the theory of invariants) are given by hand and their respective weights are computed (see \cite{Gilk95a} for instance). While these latter methods cannot be easily managed with a computer, the present method, being algebraic, can be translated into an algorithm.

The first step in the computation of $\calR_2$ makes appear a collection of fewer than 30 terms: they can be managed by hand. However, the same part of the computation for $\calR_4$ produces thousands of terms. This is why a computer is needed to perform this computation. 

Let us describe in this non-technical section the main characteristics of the computer code elaborated to make possible this computation. 

\smallskip
Some computer algebraic systems (CAS) have been evaluated as a possible basis for this code. But, to our best knowledge, none of them was able to manage, in a easy way and without adding external modules, all the complexity of this computation. Indeed, formal manipulations have to be performed, to list a few, on commutative and noncommutative objects, on derivations ($\nabla_\mu$, $\hnabla_\mu$, $\pmu$), on Riemannian structures (metric, Christoffel symbols, Ricci and Riemann tensors…), on contraction of tensors, on gauge structures ($A_\mu$ and its curvature $F_{\mu\nu}$), on tensor products, on polynomials (in the dimension parameter $d$)… Starting a formal computational code from the very beginning, as we did, has the following two main advantages: its purpose is to manipulate the necessary structures encountered in the computation, and only these structures; its internal model is based on the “mathematical model” that the method reveals.

This last point is a strong motivation to use an \emph{object oriented language} in order to internally reproduce and manipulate, in a “natural” way, all the mathematical structures describing the key ingredient in which the method (and so the code) focuses: the “arguments” of the operators $X_{\alpha,k}$, as explained in Proposition~\ref{prop-simplified-computation}. So, the code is built from the beginning on objects such as polynomials, commutative and noncommutative “elements” (for instance Riemannian tensors or matrix-valued functions), derivations (which can be applied, in a repetitive way, on the previously mentioned elements), products of elements (respecting commutative and noncommutative rules), tensor products, and finally the “arguments” of the $X_{\alpha,k}$ operators with collected commutative elements in front and the presence of the $Q[A]$ matrix-valued polynomials “on the right”.

The object oriented language selected is \texttt{JavaScript}, the powerful language used in web browsers. This choice relies on various motivations. One of us was familiar with this language, and this helped to produce an efficient code quickly. The \texttt{Node} runtime\footnote{\url{https://nodejs.org}} permits to execute \texttt{JavaScript} as a scripting language in a terminal and it makes possible to read and write files.\footnote{Since it uses an extension of the “strict” \texttt{JavaScript} language used in web browsers.} Moreover, the execution relies on the open source version of the very optimized \texttt{JavaScript} engine \texttt{V8}:\footnote{On which a lot of software engineers are working in a big private company…} benchmarks are very favorably compared to \texttt{Python} for instance (a language that would have been a good choice also). All the results are saved in files using the (open and native) format \texttt{JSON}\footnote{“JavaScript Object Notation”, a very convenient human readable structure format for data, that any modern language can read and write.} and these results can be read as inputs for further computations. The translation of the code could be done into any other modern object oriented language.

On top of the main objects that the code can manipulate (with “natural methods” from a mathematical point of view), specific (higher level) functions have been coded to reproduce mathematical rules, like for instance contractions of Riemannian tensors, raising of indices, some simplifications… Substitutions of “elements” by more complicated structures are also made possible: these permits to reproduce the steps described in Section~\ref{Results on R2}, where the computation is first done on the mathematical objects $H^{\mu\nu}$, $L^\mu$, and $q$ and then, in a second step, these objects (and their derivatives) are substituted using rules given in Section~\ref{Covariant derivatives and normal coordinates} (into normal coordinates). Substitutions rules can be hard coded or computed.

One of the main challenges when constructing such a code is to be able to simplify expressions to collect similar terms. This has required quite a lot of work to construct a “normalized” internal representation of terms (taking into account ambiguities on commutativity of elements, ambiguities on labeling indices in tensor contractions…). In that respect, the code may not compare to more mature CAS. A way to bypass this weakness was to make the exportation possible to inject expressions into another CAS: Our choice was to use \texttt{Mathematica} to perform formal computations (mainly to simplify the results at the very end of the computation) and to inject back the obtained results in the code. 

The code can also export the generated expressions in \LaTeX, and all the results presented in this paper are those directly obtained in this way. Only the final layout has been adapted to reduce space.

\smallskip
Let us now explain the main steps of the computation of $\calR_4$. One of the main ingredients in the computation is the substitution rule described in Section~\ref{Covariant derivatives and normal coordinates}. In order to avoid as much as possible any transcription errors, the choice was made to only hard code the substitutions \eqref{eq-partial0gq} and \eqref{eq-partial4gq} for the metric and its derivatives (up to four) to normal coordinates. So, a preliminary step consists in computing (and save for later use) all the necessary substitutions of covariant derivatives of $H^{\mu\nu}$, $L^\mu$, and $q$ (up to the necessary number of derivations for the computation of $\calR_4$) in terms of Riemannian tensors and derivations of $u$, $N^\mu$ and $q$. The subsequent preliminary step is to compute the replacement of the covariant derivative $\nabla_\mu$ by the total covariant derivative $\hnabla_\mu$ (up to the necessary number of derivations) on these latter elements.

Then, after these preliminary results are stored, the main computation starts with the propagation of covariant derivations as given by the rule \eqref{eq-Xmuk-arguments-derivation}. Terms are next collected according to $Q[A]$ in order to apply, sequentially to these reduced numbers of terms, the following steps: 
\begin{itemize}[itemsep=2pt, parsep=0pt, topsep=2pt]
\item Substitutions are performed to generate expressions in normal coordinates.
\item The necessary contractions with the tensors $G_{\mu_1\dots \mu_{2p}}$ (computed on the fly) are performed.
\item A series of rules (contractions of tensors, raising of indices…) is applied to all the terms as long as these rules can be applied.
\item The “expansion” rules of Lemma~\ref{lem-Xu} are applied to leverage terms with same patterns\footnote{A “pattern” of an argument is the reduced ordered list of elements appearing in this argument (forgetting the polynomials in dimension $d$ in front of it) where the $u$ elements (without applied derivation) are omitted.} to a common number $k$ of arguments, so that they can be compared. 
\item A full simplification of the obtained sum is performed by adding similar terms.
\item The results are then saved in \texttt{JSON} file format for later use and in \LaTeX\ for human reading.
\end{itemize}

Terms with same patterns are then collected in partial sums since they can produce simpler expressions, thanks to the use of Lemma~\ref{lem-Xu}, this time to reduce, as much as possible, the number of arguments $k$ and the number of terms. There is no uniqueness in this simplification procedure and a balance has to be found to produce these simplified expressions. This process has been mainly done in \texttt{Mathematica} after exportation of these partial sums. Once simplified expressions are obtained, they are written (by hand) in files that the code uses as inputs to compare them to their original (non simplified) versions. This series of checks, the last step of the computation, also exports the simplified expressions in \LaTeX: they are the expressions presented in this paper.

The code produces new results for $\calR_4$ which, to our best knowledge, never appeared before. So, some comparisons against known results have to be performed to confirm the validity of the code (at least partially). Three tests have been successful:
\begin{enumerate}[itemsep=2pt, parsep=0pt, topsep=2pt]
\item The computation of $\calR_2$ reproduces the result obtained by hand (this is the result given in Corollary~\ref{cor-resultR2-code}).

\item The computation of $\calR_2$ for the $2$-dimensional noncommutative torus produces results in agreement with  those in \cite{IochMass17c} (once translated back into spectral functions): the interested reader will find this result in \LaTeX\ files accompanying the open source code  \cite{IochMass19a}.

\item The case $u$ parallel (Section~\ref{u parallel}) is a special case for $\calR_4$ in which all the derivations of $u$ are put to zero. With the further specifications $u=\bbbone$ and $N^\mu = 0$, the result agrees with \cite[Theorem 3.3.1]{Gilk03a}.
\end{enumerate}
Notice that the consistency of gauge invariant expressions in the $Q[A]$-part of the results (see Section~\ref{sec R4k k non zero}) is also a strong requirement for the global validity of the code.

\smallskip
The code has been written with the method in mind, not for the computation of $\calR_4$ in particular. This means that it can be used to compute $\calR_r$ for $r\geq 6$ (and then the number of generated terms will be huge!) and it can also be appropriate in situations where some elements take specific values (for instance, in the $2$-dimensional noncommutative torus case, all the elements are written in terms of a single positive element in the noncommutative algebra). This makes the code flexible enough for further computations of heat coefficients for non minimal Laplace type operators.

Anyone can contribute to the open-source code (GNU GPL v3 License) we produced \cite{IochMass19a} and can use it as a starting point for his/her projects as far as the required computations are accessible by the method exposed in this paper.

\section{\texorpdfstring{Results on $\calR_4$ produced by computer}{Results on R4 produced by computer}}
\label{Results on R4}

To compute $\calR_4$ in $a_4(a,P)(x) = \tr[ a(x) \calR_4(x) ]$, we start with
\begin{align*}
(4\pi)^{\,d/2} \calR_4
&= 
\begin{aligned}[t]
& X_{d/2, 2}[P \otimes P]
- X_{d/2+1, 3, \mu_1\mu_2}[K^{\mu_1} \otimes K^{\mu_2} \otimes P]
 - X_{d/2+1, 3, \mu_1\mu_2}[K^{\mu_1} \otimes P \otimes K^{\mu_2}]
 \\
 &
- X_{d/2+1, 3, \mu_1\mu_2}[P \otimes K^{\mu_1} \otimes K^{\mu_2}]
+ X_{d/2+2, 4, \mu_1\mu_2\mu_3\mu_4}[K^{\mu_1} \otimes K^{\mu_2} \otimes K^{\mu_3} \otimes K^{\mu_4}]
.
\end{aligned}
\end{align*}
Here, the series of spectral operators $X_{\alpha, k}$ appearing in the computation are $X_{d/2-1, 1}, X_{d/2, 2}, X_{d/2+1, 3}, \dots, X_{d/2+k-2, k}$ for $k =1,\dots, 6$. Thus, in this section we adopt the shorthands $X_k$ for $k =1,\dots, 6$.

Application of \eqref{eq-Xmuk-arguments-derivation} is done using a computer because it gives too many terms. Actually, after simplification, it produces thousands of terms, that can be sorted according to only 5 values of $Q[A]\, v = \nabla^k v$ produced for $v \in \bbC^N$: $v$, $\nabla_{\!\nu_1} v$, $\nabla^2_{\nu_1\nu_2} v$, $\nabla^3_{\nu_1\nu_2\nu_3} v$, and $\nabla^4_{\nu_1\nu_2\nu_3\nu_4} v$. These five sums are denoted $\calR_{4,k}$ for $k\in\{0,1,2,3,4\}$, thus
\begin{align*}
\calR_4(x)=\sum_{k=0}^4\calR_{4,k}(x).
\end{align*}

Since the factors in front of $Q[A]$ have homogeneous gauge transformations, the result should be written as (explicitly) gauge homogeneous expressions to ensure that $\calR_4(x)$ transforms homogeneously. 

To simplify, the results are presented with $N^\mu = 0$, but the interested reader will find the general case in \LaTeX\ files accompanying the open source code  \cite{IochMass19a}.

\newlength{\interblock}
\setlength{\interblock}{2mm}

\setcounter{termscounter}{0}

\subsection{\texorpdfstring{Computation of $\calR_{4,k}$ for $k \neq 0$}{Computation of R4k}}
\label{sec R4k k non zero}

When $Q[A] \,v= \nabla^4_{\nu_1\nu_2\nu_3\nu_4} v$, the factor can be computed by hand from the very beginning because only few terms contribute. The only possible gauge homogeneous expression is $F^{\nu_1\nu_2} F_{\nu_1\nu_2}$ and indeed the computer returns directly only one term:
\begin{align*}
(4\pi)^{\,d/2}\calR_{4,4} ={}
\input{RESULTS/R4-4-one+u.tex}.
\end{align*}

For $Q[A]\,v = \nabla^3_{\nu_1\nu_2\nu_3} v$, the computer produces terms that can be sorted following a repeating pattern:
 \begin{align*}
 (g^{\nu_1\nu_2}b^{\nu_3}-2g^{\nu_1\nu_3}b^{\nu_2}+g^{\nu_2\nu_3}b^{\nu_1})  \, (\nabla^3_{\nu_{1}\nu_{2}\nu_{3}} v).
 \end{align*}
In this expression, changing the summation variables and using the symmetry of metric indices, the gauge homogeneous expression $\nabla_{\!\nu_1} F_{\nu_2\nu_3}$ appears automatically:
\begin{align*}
g^{\nu_{1}\nu_{2}} \, b^{\nu_{3}} \big(
\nabla^3_{\nu_{1}\nu_{2}\nu_{3}} v 
- \nabla^3_{\nu_{1}\nu_{3}\nu_{2}} v 
+ \nabla^3_{\nu_{3}\nu_{2}\nu_{1}} v 
- \nabla^3_{\nu_{2}\nu_{3}\nu_{1}} v
\big)
&=
 g^{\nu_{1}\nu_{2}} \, b^{\nu_{3}} \, (\nabla_{\!\nu_1} F_{\nu_2\nu_3}) v,
\end{align*}
because using \eqref{eq-commutator of hnabla},
\begin{align*}
(\nabla_{\!\nu_1} F_{\nu_2\nu_3})\, v & 
= \nabla_{\!\nu_1} (F_{\nu_2\nu_3} v) - F_{\nu_2\nu_3} \nabla_{\!\nu_1} v 
= \nabla^3_{\nu_{1}\nu_{2}\nu_{3}} v
- \nabla^3_{\nu_{1}\nu_{3}\nu_{2}} v
+ \nabla^3_{\nu_{3}\nu_{2}\nu_{1}} v 
- \nabla^3_{\nu_{2}\nu_{3}\nu_{1}} v.
\end{align*}
More precisely, this gives, since $\nabla_{\nu_1} F_{\nu_2\nu_3} \tonc \hnabla_{\nu_1} F_{\nu_2\nu_3}$,
\begin{align*}
(4\pi)^{\,d/2}\calR_{4,3} ={}
\input{RESULTS/R4-3.tex}.
\end{align*}

When $Q[A]\,v=\nabla^2_{\nu_1\nu_2} v$, the symmetric part with respect to $\nu_1,\nu_2$ does not produce a gauge homogeneous term, thus must be zero. This is checked by the computer which only returns the skew symmetric part (the $F_{\nu_1\nu_2}$ tensor):
\begin{align*}
(4\pi)^{\,d/2}&\calR_{4,2} ={}\\
\input{RESULTS/R4-2-one-1.tex}
\\[\interblock]
\input{RESULTS/R4-2-one-2.tex}.
\end{align*}

The contribution $\calR_{4,1}$ with $Q[A] \, v= \nabla_{\!\nu_1} \!v$ does not transform homogeneously because $\nabla_{\!\nu_1} \!v = A_{\nu_1} \!v$, and one cannot produce a gauge homogeneous expression as a polynomial of $A_\nu$ of degree 1 with no derivations. As a consequence, $\calR_{4,1}$ should vanish and this is what the computer returns:
\begin{align*}
\calR_{4,1} = 0.
\end{align*}

\subsection{\texorpdfstring{Computation of $\calR_{4,0}$}{Computation of R40}}

Here we use the following notations: $\Sumgigi^{\nu_1\nu_2\nu_3\nu_4}
\vc \tfrac{1}{4} \left( g^{\nu_1\nu_2}g^{\nu_3\nu_4} + g^{\nu_1\nu_3}g^{\nu_2\nu_4} + g^{\nu_1\nu_4}g^{\nu_2\nu_3} \right)$ and
\begin{align*}
\widehat{\Delta}_{\mu\nu} \vc  \tfrac{1}{2}(\hnabla_\mu \hnabla_\nu + \hnabla_\nu \hnabla_\mu), \qquad  \anticomHLHN_{\nu} \vc \hDelta\, \hnabla_{\nu} + \hnabla_{\nu} \,\hDelta,
\end{align*}
so that $\hnabla_{\mu\nu}^2 =\widehat{\Delta}_{\mu\nu} + \tfrac{1}{2} F_{\mu\nu}$ and $
\widehat{\Delta}=g^{\mu\nu}\,\hnabla_{\mu\nu}^2 =g^{\mu\nu}\, \hDelta_{\mu\nu}$.

The computer produces around 400 terms for $\calR_{4,0}$, but after simplifications based on \eqref{eq-reduction}, this reduces to the following 180 terms collected according to their pattern:
\begin{align*}
(&4\pi)^{\,d/2}\calR_{4,0} ={}\\
\input{RESULTS/R4-0.tex}.
\end{align*}
As explained for $\calR_{2}$, this presentation is not unique.

\begin{remark}
This series of terms may quite seem rather useless \textit{per se}. But the attentive reader will see that these simplified expressions show some repeated structures, for instance the use of the inclusion operator $i^{(\ell)}_{q}$ in the form $\sum_{\ell=1}^k i^{(\ell)}_{q} \SO_k$ and the possible factorization by common polynomials in the dimension parameter $d$ for a large number of terms sharing the same pattern. As can be seen in Corollary~\ref{cor-resultR2-factorized}, the results can also be presented in a more factorized way (something we have not really tried to get for $\calR_4$). Splitting the results into $u$-universal operators $\SO_{\alpha,k}$ and $P$-dependent arguments on which they are applied was a strong motivation to get a better structural perception of heat coefficients. Hitherto, we are not in position to offer more perspectives in that regard but further investigations may reveal structures hidden so far. The computation of  higher-order heat coefficients via the present method could also offer the detection of  possible (hidden) structures. We encourage insightful readers to take an interest in this problem.
\end{remark}

\subsection{\texorpdfstring{If $u$ is parallel for $\nabla$}{If u is parallel for nabla}}
\label{u parallel}

We first remark that when $u = \bbbone$, the contribution of $\calR_{4,4}$ reduces to $(4\pi)^{\,d/2}\calR_{4,4} =  \tfrac{1}{12} \, F^{\nu_{1}\nu_{2}} F_{\nu_{1}\nu_{2}}$  (even if $N^\nu$ is non zero) while $\calR_{4,3} $ vanishes (even if $N^\nu$ is non zero) and the contribution of $\calR_{4,2}$ is 
\begin{equation*}
(4\pi)^{\,d/2} \calR_{4,2} = \big(\tfrac{1}{12}N^{\nu_1}N^{\nu_2}-\tfrac{1}{6} g^{\nu_2\nu_3} (\nabla_{\!\nu_3} N^{\nu_1})\big)\, F_{\nu_1\nu_2}.
\end{equation*}
More generally, when $u$ is parallel for the connection $\nabla$, we get the following 55 terms for $\calR_{4} $; recall from \eqref{eq-G1234} that $\Sumgg_{\nu_1\nu_2\nu_3\nu_4} \vc \tfrac{1}{4} \big( g_{\nu_1\nu_2}g_{\nu_3\nu_4} + g_{\nu_1\nu_3}g_{\nu_2\nu_4} + g_{\nu_1\nu_4}g_{\nu_2\nu_3})$:
\setcounter{termscounter}{0}
\begin{align*}
(4\pi&)^{\,d/2}\calR_{4} ={}\\
\input{RESULTS/u-parallel-full.tex}.
\end{align*}

\setcounter{termscounter}{0}

With $N^\mu = 0$, this shrinks to 8 terms:
\begin{align*}
(4\pi)^{\,d/2}\calR_{4} ={}
\input{RESULTS/u-parallel-N0.tex},
\end{align*}
and agrees with \cite[Theorem 3.3.1]{Gilk03a} when $u = \bbbone$ using \eqref{eq-Xk for u diagonal}.

\section{Conclusion}

In this work we have developed for a non minimal Laplace type operator a new method to compute any $\calR_r$ in terms of universal operators $X_{\alpha,k}$ with explicit details of the case $r=2$. A computer code issued by this method has already produced new results for $\calR_4$. This code is ready for such computations when $r \geq 6$ and moreover it can be particularized and adapted to more specific situations like the (rational) noncommutative torus, see for instance \cite{ConnesFath16}, or in quantum field theory as described in \cite[Section 5.4]{IochMass17a}.
Among possible perspectives, the method could also be generalized to operators $P$ acting on operator algebras, for instance constructed from spectral triples.

\section*{Acknowledgments}

The authors are strongly indebted to Laurent Raymond for helpful discussions concerning some aspects of the computer code at the early stage of this work.

\bibliographystyle{plainnat}
\bibliography{Heat-trace-biblio}

\end{document}

%% file: RESULTS/R4-4-one+u.tex
& +\tfrac{1}{12} \, \SO_{1}[ u ]  \, F^{\nu_{1}\nu_{2}} F_{\nu_{1}\nu_{2}} \EqLine{1} 

%% file: RESULTS/R4-3.tex
& +\tfrac{1}{6}(d - 2) \, g^{\nu_{1}\nu_{2}} g^{\nu_{3}\nu_{4}} \, \SO_{2}[ u \otimes (\hnabla_{\nu_{4}} u) ]  \, (\hnabla_{\nu_{1}} F_{\nu_{2}\nu_{3}}) \EqLine{1} 
-d \, g^{\nu_{1}\nu_{2}} g^{\nu_{3}\nu_{4}} \, \SO_{3}[ u \otimes u \otimes (\hnabla_{\nu_{4}} u) ]  \, (\hnabla_{\nu_{1}} F_{\nu_{2}\nu_{3}}) \EqLine{2} 
\\ & 
+2 (d + 2) \, g^{\nu_{1}\nu_{2}} g^{\nu_{3}\nu_{4}} \, \SO_{4}[ u \otimes u \otimes u \otimes (\hnabla_{\nu_{4}} u) ]  \, (\hnabla_{\nu_{1}} F_{\nu_{2}\nu_{3}}) \EqLine{3} 

%% file: RESULTS/R4-2-one-1.tex
& -d \, \SO_{3}[ u \otimes (\hnabla_{\nu_{1}} u) \otimes (\hnabla_{\nu_{2}} u) ]  \, F^{\nu_{1}\nu_{2}} \EqLine{1} 
 -4 \, \SO_{4}[ u \otimes u \otimes (\hnabla_{\nu_{1}} u) \otimes (\hnabla_{\nu_{2}} u) ]  \, F^{\nu_{1}\nu_{2}} \EqLine{2} 
 \\ & 
 +\tfrac{1}{2}(d + 2) (d + 4) \, \SO_{5}[ u \otimes (\hnabla_{\nu_{1}} u) \otimes u \otimes (\hnabla_{\nu_{2}} u) \otimes u ]  \, F^{\nu_{1}\nu_{2}} \EqLine{3} 
 +2 (d + 4) \, \SO_{5}[ u \otimes u \otimes (\hnabla_{\nu_{1}} u) \otimes (\hnabla_{\nu_{2}} u) \otimes u ]  \, F^{\nu_{1}\nu_{2}} \EqLine{4} 
 \\ & 
 +4 (d + 4) \, \SO_{5}[ u \otimes u \otimes u \otimes (\hnabla_{\nu_{1}} u) \otimes (\hnabla_{\nu_{2}} u) ]  \, F^{\nu_{1}\nu_{2}} \EqLine{5} 
 +(d + 4) (d + 6) \, \SO_{6}[ u \otimes u \otimes (\hnabla_{\nu_{1}} u) \otimes u \otimes u \otimes (\hnabla_{\nu_{2}} u) ]  \, F^{\nu_{1}\nu_{2}} \EqLine{6} 
 \\ & 
 +(d + 4) (d + 6) \, \SO_{6}[ u \otimes u \otimes u \otimes (\hnabla_{\nu_{1}} u) \otimes u \otimes (\hnabla_{\nu_{2}} u) ]  \, F^{\nu_{1}\nu_{2}} \EqLine{7} 

%% file: RESULTS/R4-2-one-2.tex
& -\tfrac{1}{2}d \, \SO_{3}[ u \otimes u \otimes \comFu_{\nu_{1}\nu_{2}} ]  \, F^{\nu_{1}\nu_{2}} \EqLine{1} 
 +(d + 2) \, \SO_{4}[ u \otimes u \otimes u \otimes \comFu_{\nu_{1}\nu_{2}} ]  \, F^{\nu_{1}\nu_{2}} \EqLine{2} 
\\ & 
 -(d + 4) \, \SO_{5}[ u \otimes u \otimes u \otimes u \otimes \comFu_{\nu_{1}\nu_{2}} ]  \, F^{\nu_{1}\nu_{2}} \EqLine{3} 

%% file: RESULTS/R4-0.tex
& +\tfrac{1}{180} \, \abs{R}^2 \, \SO_{1}[ u ]  \EqLine{1}
-\tfrac{1}{180} \, \abs{\Ric}^2 \, \SO_{1}[ u ]  \EqLine{1} 
+\tfrac{1}{30} \, (\hDelta \SC) \, \SO_{1}[ u ]  \EqLine{1} 
+\tfrac{1}{72} \, \SC^2 \, \SO_{1}[ u ]  \EqLine{1} 
\\[\interblock]
& -\tfrac{1}{12}(d - 2) \, g^{\nu_{1}\nu_{2}} (\hnabla_{\nu_{2}} \SC) \, \SO_{2}[ u \otimes (\hnabla_{\nu_{1}} u) ]  \EqLine{1} 
 +\tfrac{1}{2}(d + 2) \, g^{\nu_{1}\nu_{2}} (\hnabla_{\nu_{2}} \SC) \, \SO_{4}[ u \otimes u \otimes u \otimes (\hnabla_{\nu_{1}} u) ]  \EqLine{2} 
\\[\interblock]
& -\tfrac{1}{3}d \, \Ric^{\nu_{1}\nu_{2}} \, \SO_{3}[ u \otimes (\hnabla_{\nu_{1}} u) \otimes (\hnabla_{\nu_{2}} u) ]  \EqLine{1} 
 +\tfrac{1}{3}(d + 2) (d + 3) \, \Ric^{\nu_{1}\nu_{2}} \, \SO_{4}[ u \otimes u \otimes (\hnabla_{\nu_{1}} u) \otimes (\hnabla_{\nu_{2}} u) ]  \EqLine{2} 
\\ & 
 -\tfrac{1}{12}d (d + 2) \, \Ric^{\nu_{1}\nu_{2}} \, \SO_{4}[ u \otimes (\hnabla_{\nu_{1}} u) \otimes u \otimes (\hnabla_{\nu_{2}} u) ]  \EqLine{3} 
 -(d + 2) (d + 4) \, \Ric^{\nu_{1}\nu_{2}} \, \SO_{5}[ u \otimes u \otimes u \otimes (\hnabla_{\nu_{1}} u) \otimes (\hnabla_{\nu_{2}} u) ]  \EqLine{4} 
\\ & 
 +\tfrac{1}{4}(d + 2) (d + 4) \, \Ric^{\nu_{1}\nu_{2}} \, \SO_{5}[ u \otimes (\hnabla_{\nu_{1}} u) \otimes u \otimes (\hnabla_{\nu_{2}} u) \otimes u ]  \EqLine{5} 
\\
& -\tfrac{1}{6}d \, g^{\nu_{1}\nu_{2}} \SC \, \SO_{3}[ u \otimes (\hnabla_{\nu_{1}} u) \otimes (\hnabla_{\nu_{2}} u) ]  \EqLine{1} 
 +\tfrac{1}{12}d (d + 2) \, g^{\nu_{1}\nu_{2}} \SC \, \SO_{4}[ u \otimes (\hnabla_{\nu_{1}} u) \otimes (\hnabla_{\nu_{2}} u) \otimes u ]  \EqLine{2} 
\\[\interblock]
& 
-2 (d + 4) (d^{2} + 10d + 28) \, \Sumgigi^{\nu_{1}\nu_{2}\nu_{3}\nu_{4}} \, \SO_{5}[ u \otimes (\hnabla_{\nu_{1}} u) \otimes (\hnabla_{\nu_{2}} u) \otimes (\hnabla_{\nu_{3}} u) \otimes (\hnabla_{\nu_{4}} u) ]  \EqLine{1} 
\\ & 
 +4 (d + 4) (d + 6) \, \Sumgigi^{\nu_{1}\nu_{2}\nu_{3}\nu_{4}} \, \SO_{6}[ u \otimes (\hnabla_{\nu_{1}} u) \otimes (\hnabla_{\nu_{2}} u) \otimes (\hnabla_{\nu_{3}} u) \otimes (\hnabla_{\nu_{4}} u) \otimes u ]  \EqLine{2} 
\\ & 
 +2 (d + 4) (d + 6)^{2} \, \Sumgigi^{\nu_{1}\nu_{2}\nu_{3}\nu_{4}} \, \SO_{6}[ u \otimes (\hnabla_{\nu_{1}} u) \otimes (\hnabla_{\nu_{2}} u) \otimes (\hnabla_{\nu_{3}} u) \otimes u \otimes (\hnabla_{\nu_{4}} u) ]  \EqLine{3} 
\\ & 
 +2 (d + 4)^{2} (d + 6) \, \Sumgigi^{\nu_{1}\nu_{2}\nu_{3}\nu_{4}} \, \SO_{6}[ u \otimes (\hnabla_{\nu_{1}} u) \otimes (\hnabla_{\nu_{2}} u) \otimes u \otimes (\hnabla_{\nu_{3}} u) \otimes (\hnabla_{\nu_{4}} u) ]  \EqLine{4} 
\\ & 
 -4 (d + 4) (d^{2} + 8d + 28) \, \Sumgigi^{\nu_{1}\nu_{2}\nu_{3}\nu_{4}} \, \SO_{6}[ u \otimes u \otimes (\hnabla_{\nu_{1}} u) \otimes (\hnabla_{\nu_{2}} u) \otimes (\hnabla_{\nu_{3}} u) \otimes (\hnabla_{\nu_{4}} u) ]  \EqLine{5} 
\\ & 
 -2 (d + 2) (d + 4) (d + 6) \, \Sumgigi^{\nu_{1}\nu_{2}\nu_{3}\nu_{4}} \, \SO_{6}[ u \otimes (\hnabla_{\nu_{1}} u) \otimes u \otimes (\hnabla_{\nu_{2}} u) \otimes (\hnabla_{\nu_{3}} u) \otimes (\hnabla_{\nu_{4}} u) ]  \EqLine{6} 
\\ & 
 +8 (d + 4) (d^{2} + 10d + 32) \, \Sumgigi^{\nu_{1}\nu_{2}\nu_{3}\nu_{4}} \, \SO_{7}[ u \otimes u \otimes (\hnabla_{\nu_{1}} u) \otimes (\hnabla_{\nu_{2}} u) \otimes u \otimes (\hnabla_{\nu_{3}} u) \otimes (\hnabla_{\nu_{4}} u) ]  \EqLine{7} 
\\ & 
 +4 (d + 4) (d^{2} + 6d + 16) \, \Sumgigi^{\nu_{1}\nu_{2}\nu_{3}\nu_{4}} \, \SO_{7}[ u \otimes u \otimes (\hnabla_{\nu_{1}} u) \otimes (\hnabla_{\nu_{2}} u) \otimes (\hnabla_{\nu_{3}} u) \otimes u \otimes (\hnabla_{\nu_{4}} u) ]  \EqLine{8} 
\\ & 
-32 (d + 4)^{2} \, \Sumgigi^{\nu_{1}\nu_{2}\nu_{3}\nu_{4}} \, \SO_{7}[ u \otimes u \otimes (\hnabla_{\nu_{1}} u) \otimes (\hnabla_{\nu_{2}} u) \otimes (\hnabla_{\nu_{3}} u) \otimes (\hnabla_{\nu_{4}} u) \otimes u ]  \EqLine{9} 
\\ & 
 +8 (d + 4)^{2} (d + 6) \, \Sumgigi^{\nu_{1}\nu_{2}\nu_{3}\nu_{4}} \, \SO_{7}[ u \otimes (\hnabla_{\nu_{1}} u) \otimes u \otimes u \otimes (\hnabla_{\nu_{2}} u) \otimes (\hnabla_{\nu_{3}} u) \otimes (\hnabla_{\nu_{4}} u) ]  \EqLine{10} 
\\ & 
 +4 (d + 4)^{2} (d + 6) \, \Sumgigi^{\nu_{1}\nu_{2}\nu_{3}\nu_{4}} \, \SO_{7}[ u \otimes (\hnabla_{\nu_{1}} u) \otimes u \otimes (\hnabla_{\nu_{2}} u) \otimes u \otimes (\hnabla_{\nu_{3}} u) \otimes (\hnabla_{\nu_{4}} u) ]  \EqLine{11} 
\\ & 
 +2 d (d + 4) (d + 6) \, \Sumgigi^{\nu_{1}\nu_{2}\nu_{3}\nu_{4}} \, \SO_{7}[ u \otimes (\hnabla_{\nu_{1}} u) \otimes u \otimes (\hnabla_{\nu_{2}} u) \otimes (\hnabla_{\nu_{3}} u) \otimes u \otimes (\hnabla_{\nu_{4}} u) ]  \EqLine{12} 
\\ & 
 -16 (d + 4) (d + 6) \, \Sumgigi^{\nu_{1}\nu_{2}\nu_{3}\nu_{4}} \, \SO_{7}[ u \otimes (\hnabla_{\nu_{1}} u) \otimes u \otimes (\hnabla_{\nu_{2}} u) \otimes (\hnabla_{\nu_{3}} u) \otimes (\hnabla_{\nu_{4}} u) \otimes u ]  \EqLine{13} 
\\ & 
 +2 (d + 4) (d + 6) (d + 8) \, \Sumgigi^{\nu_{1}\nu_{2}\nu_{3}\nu_{4}} \, \SO_{7}[ u \otimes (\hnabla_{\nu_{1}} u) \otimes (\hnabla_{\nu_{2}} u) \otimes u \otimes (\hnabla_{\nu_{3}} u) \otimes (\hnabla_{\nu_{4}} u) \otimes u ]  \EqLine{14} 
\\ & 
 +24 (d + 4) (d^{2} + 10d + 32) \, \Sumgigi^{\nu_{1}\nu_{2}\nu_{3}\nu_{4}} \, \SO_{7}[ u \otimes u \otimes u \otimes (\hnabla_{\nu_{1}} u) \otimes (\hnabla_{\nu_{2}} u) \otimes (\hnabla_{\nu_{3}} u) \otimes (\hnabla_{\nu_{4}} u) ]  \EqLine{15} 
\\ & 
 +16 (d + 4) (d^{2} + 10d + 28) \, \Sumgigi^{\nu_{1}\nu_{2}\nu_{3}\nu_{4}} \, \SO_{7}[ u \otimes u \otimes (\hnabla_{\nu_{1}} u) \otimes u \otimes (\hnabla_{\nu_{2}} u) \otimes (\hnabla_{\nu_{3}} u) \otimes (\hnabla_{\nu_{4}} u) ]  \EqLine{16} 
\\ & 
 +4 (d + 4) (d + 6) (d + 8) (d + 10) \, \Sumgigi^{\nu_{1}\nu_{2}\nu_{3}\nu_{4}} \, \SO_{8}[ u \otimes u \otimes (\hnabla_{\nu_{1}} u) \otimes (\hnabla_{\nu_{2}} u) \otimes (\hnabla_{\nu_{3}} u) \otimes (\hnabla_{\nu_{4}} u) \otimes u \otimes u ]  \EqLine{17} 
\\ & 
 +2 (d + 4) (d + 6) (d + 8) (d + 10) \, \Sumgigi^{\nu_{1}\nu_{2}\nu_{3}\nu_{4}} \, \SO_{8}[ u \otimes (\hnabla_{\nu_{1}} u) \otimes u \otimes (\hnabla_{\nu_{2}} u) \otimes (\hnabla_{\nu_{3}} u) \otimes (\hnabla_{\nu_{4}} u) \otimes u \otimes u ]  \EqLine{18} 
\\ & 
 +(d + 4) (d + 6) (d + 8) (d + 10) \, \Sumgigi^{\nu_{1}\nu_{2}\nu_{3}\nu_{4}} \, \SO_{8}[ u \otimes (\hnabla_{\nu_{1}} u) \otimes u \otimes (\hnabla_{\nu_{2}} u) \otimes (\hnabla_{\nu_{3}} u) \otimes u \otimes (\hnabla_{\nu_{4}} u) \otimes u ]  \EqLine{19} 
\\ & 
 +2 (d + 4) (d + 6) (d + 8) (d + 10) \, \Sumgigi^{\nu_{1}\nu_{2}\nu_{3}\nu_{4}} \, \SO_{8}[ u \otimes u \otimes (\hnabla_{\nu_{1}} u) \otimes (\hnabla_{\nu_{2}} u) \otimes (\hnabla_{\nu_{3}} u) \otimes u \otimes (\hnabla_{\nu_{4}} u) \otimes u ]  \EqLine{20} 
\\
& -\tfrac{1}{2}(d + 2) (d + 4)^{2} \, g^{\nu_{1}\nu_{2}} g^{\nu_{3}\nu_{4}} \, \SO_{5}[ u \otimes (\hnabla_{\nu_{1}} u) \otimes (\hnabla_{\nu_{2}} u) \otimes (\hnabla_{\nu_{3}} u) \otimes (\hnabla_{\nu_{4}} u) ]  \EqLine{1} 
\\ & 
 +\tfrac{1}{2}(d + 2) (d + 4) (d + 6) \, g^{\nu_{1}\nu_{2}} g^{\nu_{3}\nu_{4}} \, \SO_{6}[ u \otimes (\hnabla_{\nu_{1}} u) \otimes (\hnabla_{\nu_{2}} u) \otimes (\hnabla_{\nu_{3}} u) \otimes u \otimes (\hnabla_{\nu_{4}} u) ]  \EqLine{2} 
\\ & 
 +\tfrac{1}{2}(d + 2) (d + 4) (d + 6) \, g^{\nu_{1}\nu_{2}} g^{\nu_{3}\nu_{4}} \, \SO_{6}[ u \otimes (\hnabla_{\nu_{1}} u) \otimes u \otimes (\hnabla_{\nu_{2}} u) \otimes (\hnabla_{\nu_{3}} u) \otimes (\hnabla_{\nu_{4}} u) ]  \EqLine{3} 
\\ & 
 +(d + 2) (d + 4) (d + 6) \, g^{\nu_{1}\nu_{2}} g^{\nu_{3}\nu_{4}} \, \SO_{6}[ u \otimes u \otimes (\hnabla_{\nu_{1}} u) \otimes (\hnabla_{\nu_{2}} u) \otimes (\hnabla_{\nu_{3}} u) \otimes (\hnabla_{\nu_{4}} u) ]  \EqLine{4} 
\\ & 
 +\tfrac{1}{4}(d + 2) (d + 4) (d + 6) (d + 8) \, g^{\nu_{1}\nu_{2}} g^{\nu_{3}\nu_{4}} \, \SO_{7}[ u \otimes (\hnabla_{\nu_{1}} u) \otimes (\hnabla_{\nu_{2}} u) \otimes u \otimes (\hnabla_{\nu_{3}} u) \otimes (\hnabla_{\nu_{4}} u) \otimes u ]  \EqLine{5} 
\\[\interblock]
& -\tfrac{1}{12}(d - 2) \, \SC \, \SO_{2}[ u \otimes (\hDelta u) ]  \EqLine{1} 
 +\tfrac{1}{6}d \, \SC \, \SO_{3}[ u \otimes u \otimes (\hDelta u) ]  \EqLine{2} 
\\
& -\tfrac{1}{3}d \, \Ric^{\nu_{1}\nu_{2}} \, \SO_{3}[ u \otimes u \otimes (\hDelta_{\nu_{1}\nu_{2}} u) ]  \EqLine{1} 
 +(d + 2) \, \Ric^{\nu_{1}\nu_{2}} \, \SO_{4}[ u \otimes u \otimes u \otimes (\hDelta_{\nu_{1}\nu_{2}} u) ]  \EqLine{2} 
\\[\interblock]
& +\tfrac{1}{2}(d + 2)^{2} \, g^{\nu_{1}\nu_{2}} \, \SO_{4}[ u \otimes (\hDelta u) \otimes (\hnabla_{\nu_{1}} u) \otimes (\hnabla_{\nu_{2}} u) ]  \EqLine{1} 
 -\tfrac{1}{2}(d + 2) (d + 4) \, g^{\nu_{1}\nu_{2}} \, \SO_{5}[ u \otimes (\hDelta u) \otimes (\hnabla_{\nu_{1}} u) \otimes u \otimes (\hnabla_{\nu_{2}} u) ]  \EqLine{2} 
 \\ & 
 -(d + 4)^{2} \, g^{\nu_{1}\nu_{2}} \, \SO_{5}[ u \otimes u \otimes (\hDelta u) \otimes (\hnabla_{\nu_{1}} u) \otimes (\hnabla_{\nu_{2}} u) ]  \EqLine{3} 
\\ & 
 -\tfrac{1}{4}(d + 2) (d + 4) (d + 6) \, g^{\nu_{1}\nu_{2}} \, \SO_{6}[ u \otimes (\hDelta u) \otimes u \otimes (\hnabla_{\nu_{1}} u) \otimes (\hnabla_{\nu_{2}} u) \otimes u ]  \EqLine{4} 
\\ & 
+(d + 4) (d + 6) \, g^{\nu_{1}\nu_{2}} \, \SO_{6}[ u \otimes u \otimes (\hDelta u) \otimes (\hnabla_{\nu_{1}} u) \otimes u \otimes (\hnabla_{\nu_{2}} u) ]  \EqLine{5} 
\\ & 
-(d + 2) (d + 4) (d + 6) \, g^{\nu_{1}\nu_{2}} \, \SO_{7}[ u \otimes u \otimes (\hDelta u) \otimes (\hnabla_{\nu_{1}} u) \otimes (\hnabla_{\nu_{2}} u) \otimes u \otimes u ]  \EqLine{6} 
\\ & 
-\tfrac{1}{2}(d + 2) (d + 4) (d + 6) \, g^{\nu_{1}\nu_{2}} \, \SO_{7}[ u \otimes u \otimes (\hDelta u) \otimes (\hnabla_{\nu_{1}} u) \otimes u \otimes (\hnabla_{\nu_{2}} u) \otimes u ]  \EqLine{7} 
\\ & 
+3 (d + 4) (d + 6) \, g^{\nu_{1}\nu_{2}} \, \SO_{7}[ u \otimes u \otimes (\hDelta u) \otimes u \otimes (\hnabla_{\nu_{1}} u) \otimes (\hnabla_{\nu_{2}} u) \otimes u ]  \EqLine{8} 
\\ & 
+9 (d + 4) (d + 6) \, g^{\nu_{1}\nu_{2}} \, \SO_{7}[ u \otimes u \otimes u \otimes (\hDelta u) \otimes (\hnabla_{\nu_{1}} u) \otimes (\hnabla_{\nu_{2}} u) \otimes u ]  \EqLine{9} 
\\
& -(d + 2) \, g^{\nu_{1}\nu_{2}} \, \SO_{4}[ u \otimes (\hnabla_{\nu_{1}} u) \otimes (\hDelta u) \otimes (\hnabla_{\nu_{2}} u) ]  \EqLine{1} 
 +\tfrac{1}{2}(d + 4) (3d + 14) \, g^{\nu_{1}\nu_{2}} \, \SO_{5}[ u \otimes (\hnabla_{\nu_{1}} u) \otimes (\hDelta u) \otimes (\hnabla_{\nu_{2}} u) \otimes u ]  \EqLine{2} 
\\ & 
 -2 (d + 4) (d + 6) \, g^{\nu_{1}\nu_{2}} \, \SO_{6}[ u \otimes (\hnabla_{\nu_{1}} u) \otimes (\hDelta u) \otimes (\hnabla_{\nu_{2}} u) \otimes u \otimes u ]  \EqLine{3} 
\\ & 
 -(d + 4) (d + 6) \, g^{\nu_{1}\nu_{2}} \, \SO_{6}[ u \otimes (\hnabla_{\nu_{1}} u) \otimes (\hDelta u) \otimes u \otimes (\hnabla_{\nu_{2}} u) \otimes u ]  \EqLine{4} 
\\ & 
 +\tfrac{1}{2}(d + 4) (d + 6) \, g^{\nu_{1}\nu_{2}} \, \SO_{6}[ u \otimes (\hnabla_{\nu_{1}} u) \otimes u \otimes (\hDelta u) \otimes u \otimes (\hnabla_{\nu_{2}} u) ]  \EqLine{5} 
\\ & 
 +(d + 4) (d + 6) \, g^{\nu_{1}\nu_{2}} \, \SO_{6}[ u \otimes u \otimes (\hnabla_{\nu_{1}} u) \otimes (\hDelta u) \otimes u \otimes (\hnabla_{\nu_{2}} u) ]  \EqLine{6}
\\ & 
 -\tfrac{1}{2}(d + 4) (d + 6) (d + 8) \, g^{\nu_{1}\nu_{2}} \, \SO_{7}[ u \otimes (\hnabla_{\nu_{1}} u) \otimes u \otimes (\hDelta u) \otimes (\hnabla_{\nu_{2}} u) \otimes u \otimes u ]  \EqLine{7} 
\\ & 
 -\tfrac{1}{4}(d + 4) (d + 6) (d + 8) \, g^{\nu_{1}\nu_{2}} \, \SO_{7}[ u \otimes (\hnabla_{\nu_{1}} u) \otimes u \otimes (\hDelta u) \otimes u \otimes (\hnabla_{\nu_{2}} u) \otimes u ]  \EqLine{8} 
\\ & 
 -(d + 4) (d + 6) (d + 8) \, g^{\nu_{1}\nu_{2}} \, \SO_{7}[ u \otimes u \otimes (\hnabla_{\nu_{1}} u) \otimes (\hDelta u) \otimes (\hnabla_{\nu_{2}} u) \otimes u \otimes u ]  \EqLine{9} 
\\ & 
 -\tfrac{1}{2}(d + 4) (d + 6) (d + 8) \, g^{\nu_{1}\nu_{2}} \, \SO_{7}[ u \otimes u \otimes (\hnabla_{\nu_{1}} u) \otimes (\hDelta u) \otimes u \otimes (\hnabla_{\nu_{2}} u) \otimes u ]  \EqLine{10} 
\\
& -(d + 2) \, g^{\nu_{1}\nu_{2}} \, \SO_{4}[ u \otimes (\hnabla_{\nu_{1}} u) \otimes (\hnabla_{\nu_{2}} u) \otimes (\hDelta u) ]  \EqLine{1} 
 +\tfrac{1}{2}(d + 4) (3d + 10) \, g^{\nu_{1}\nu_{2}} \, \SO_{5}[ u \otimes (\hnabla_{\nu_{1}} u) \otimes (\hnabla_{\nu_{2}} u) \otimes (\hDelta u) \otimes u ]  \EqLine{2} 
\\ & 
 +\tfrac{1}{2}(d + 2) (d + 4) \, g^{\nu_{1}\nu_{2}} \, \SO_{5}[ u \otimes (\hnabla_{\nu_{1}} u) \otimes (\hnabla_{\nu_{2}} u) \otimes u \otimes (\hDelta u) ]  \EqLine{3} 
\\ & 
 -(d + 4) (d + 6) \, g^{\nu_{1}\nu_{2}} \, \SO_{6}[ u \otimes (\hnabla_{\nu_{1}} u) \otimes (\hnabla_{\nu_{2}} u) \otimes (\hDelta u) \otimes u \otimes u ]  \EqLine{4} 
\\ & 
 +\tfrac{1}{2}(d + 4) (d + 6) \, g^{\nu_{1}\nu_{2}} \, \SO_{6}[ u \otimes (\hnabla_{\nu_{1}} u) \otimes (\hnabla_{\nu_{2}} u) \otimes u \otimes (\hDelta u) \otimes u ]  \EqLine{5} 
\\ & 
 -\tfrac{1}{2}(d + 4) (d + 6) (d + 8) \, g^{\nu_{1}\nu_{2}} \, \SO_{7}[ u \otimes (\hnabla_{\nu_{1}} u) \otimes (\hnabla_{\nu_{2}} u) \otimes u \otimes u \otimes (\hDelta u) \otimes u ]  \EqLine{6} 
\\ & 
 -\tfrac{1}{2}(d + 4) (d + 6) (d + 8) \, g^{\nu_{1}\nu_{2}} \, \SO_{7}[ u \otimes (\hnabla_{\nu_{1}} u) \otimes (\hnabla_{\nu_{2}} u) \otimes u \otimes (\hDelta u) \otimes u \otimes u ]  \EqLine{7} 
\\ & 
 -\tfrac{1}{2}(d + 4) (d + 6) (d + 8) \, g^{\nu_{1}\nu_{2}} \, \SO_{7}[ u \otimes (\hnabla_{\nu_{1}} u) \otimes u \otimes (\hnabla_{\nu_{2}} u) \otimes (\hDelta u) \otimes u \otimes u ]  \EqLine{8} 
\\ & 
 -\tfrac{1}{4}(d + 4) (d + 6) (d + 8) \, g^{\nu_{1}\nu_{2}} \, \SO_{7}[ u \otimes (\hnabla_{\nu_{1}} u) \otimes u \otimes (\hnabla_{\nu_{2}} u) \otimes u \otimes (\hDelta u) \otimes u ]  \EqLine{9} 
\\ & 
 -(d + 4) (d + 6) (d + 8) \, g^{\nu_{1}\nu_{2}} \, \SO_{7}[ u \otimes u \otimes (\hnabla_{\nu_{1}} u) \otimes (\hnabla_{\nu_{2}} u) \otimes (\hDelta u) \otimes u \otimes u ]  \EqLine{10} 
\\ & 
 -\tfrac{1}{2}(d + 4) (d + 6) (d + 8) \, g^{\nu_{1}\nu_{2}} \, \SO_{7}[ u \otimes u \otimes (\hnabla_{\nu_{1}} u) \otimes (\hnabla_{\nu_{2}} u) \otimes u \otimes (\hDelta u) \otimes u ]  \EqLine{11} 
\\
& +2 (d + 4) (3d + 16) \, \SO_{5}[ u \otimes u \otimes (\hDelta^{\nu_{1}\nu_{2}} u) \otimes (\hnabla_{\nu_{1}} u) \otimes (\hnabla_{\nu_{2}} u) ]  \EqLine{1} 
\\ & 
 -4 (d + 4) (d + 6) \, \SO_{6}[ u \otimes u \otimes (\hDelta^{\nu_{1}\nu_{2}} u) \otimes (\hnabla_{\nu_{1}} u) \otimes u \otimes (\hnabla_{\nu_{2}} u) ]  \EqLine{2} 
\\ & 
 -6 (d + 4) (d + 6) \, \SO_{6}[ u \otimes u \otimes (\hDelta^{\nu_{1}\nu_{2}} u) \otimes u \otimes (\hnabla_{\nu_{1}} u) \otimes (\hnabla_{\nu_{2}} u) ]  \EqLine{3} 
\\ & 
 -18 (d + 4) (d + 6) \, \SO_{6}[ u \otimes u \otimes u \otimes (\hDelta^{\nu_{1}\nu_{2}} u) \otimes (\hnabla_{\nu_{1}} u) \otimes (\hnabla_{\nu_{2}} u) ]  \EqLine{4} 
\\ & 
 -2 (d + 4) (d + 6) (d + 8) \, \SO_{7}[ u \otimes u \otimes (\hDelta^{\nu_{1}\nu_{2}} u) \otimes (\hnabla_{\nu_{1}} u) \otimes (\hnabla_{\nu_{2}} u) \otimes u \otimes u ]  \EqLine{5} 
\\ & 
 -(d + 4) (d + 6) (d + 8) \, \SO_{7}[ u \otimes u \otimes (\hDelta^{\nu_{1}\nu_{2}} u) \otimes (\hnabla_{\nu_{1}} u) \otimes u \otimes (\hnabla_{\nu_{2}} u) \otimes u ]  \EqLine{6} 
\\
& -2 (d + 2) \, \SO_{4}[ u \otimes (\hnabla_{\nu_{1}} u) \otimes (\hDelta^{\nu_{1}\nu_{2}} u) \otimes (\hnabla_{\nu_{2}} u) ]  \EqLine{1} 
 +(d + 4) (3d + 14) \, \SO_{5}[ u \otimes (\hnabla_{\nu_{1}} u) \otimes (\hDelta^{\nu_{1}\nu_{2}} u) \otimes (\hnabla_{\nu_{2}} u) \otimes u ]  \EqLine{2} 
\\ & 
 -4 (d + 4) (d + 6) \, \SO_{6}[ u \otimes (\hnabla_{\nu_{1}} u) \otimes (\hDelta^{\nu_{1}\nu_{2}} u) \otimes (\hnabla_{\nu_{2}} u) \otimes u \otimes u ]  \EqLine{3} 
\\ & 
 -2 (d + 4) (d + 6) \, \SO_{6}[ u \otimes (\hnabla_{\nu_{1}} u) \otimes (\hDelta^{\nu_{1}\nu_{2}} u) \otimes u \otimes (\hnabla_{\nu_{2}} u) \otimes u ]  \EqLine{4} 
\\ & 
 +(d + 4) (d + 6) \, \SO_{6}[ u \otimes (\hnabla_{\nu_{1}} u) \otimes u \otimes (\hDelta^{\nu_{1}\nu_{2}} u) \otimes u \otimes (\hnabla_{\nu_{2}} u) ]  \EqLine{5} 
\\ & 
 +2 (d + 4) (d + 6) \, \SO_{6}[ u \otimes u \otimes (\hnabla_{\nu_{1}} u) \otimes (\hDelta^{\nu_{1}\nu_{2}} u) \otimes u \otimes (\hnabla_{\nu_{2}} u) ]  \EqLine{6} 
\\ & 
 -(d + 4) (d + 6) (d + 8) \, \SO_{7}[ u \otimes (\hnabla_{\nu_{1}} u) \otimes u \otimes (\hDelta^{\nu_{1}\nu_{2}} u) \otimes (\hnabla_{\nu_{2}} u) \otimes u \otimes u ]  \EqLine{7} 
\\ & 
 -\tfrac{1}{2}(d + 4) (d + 6) (d + 8) \, \SO_{7}[ u \otimes (\hnabla_{\nu_{1}} u) \otimes u \otimes (\hDelta^{\nu_{1}\nu_{2}} u) \otimes u \otimes (\hnabla_{\nu_{2}} u) \otimes u ]  \EqLine{8} 
\\ & 
 -2 (d + 4) (d + 6) (d + 8) \, \SO_{7}[ u \otimes u \otimes (\hnabla_{\nu_{1}} u) \otimes (\hDelta^{\nu_{1}\nu_{2}} u) \otimes (\hnabla_{\nu_{2}} u) \otimes u \otimes u ]  \EqLine{9} 
\\ & 
 -(d + 4) (d + 6) (d + 8) \, \SO_{7}[ u \otimes u \otimes (\hnabla_{\nu_{1}} u) \otimes (\hDelta^{\nu_{1}\nu_{2}} u) \otimes u \otimes (\hnabla_{\nu_{2}} u) \otimes u ]  \EqLine{10} 
\\
& +\tfrac{1}{2}(d + 2)^{2} \, \SO_{4}[ u \otimes (\hnabla_{\nu_{1}} u) \otimes \comFUu^{\nu_{1}\nu_{2}} \otimes (\hnabla_{\nu_{2}} u) ]  \EqLine{1} 
 -\tfrac{1}{2}(d + 2) (d + 4) \, \SO_{5}[ u \otimes (\hnabla_{\nu_{1}} u) \otimes \comFUu^{\nu_{1}\nu_{2}} \otimes u \otimes (\hnabla_{\nu_{2}} u) ]  \EqLine{2} 
\\ & 
 -(d + 2) (d + 4) \, \SO_{5}[ u \otimes u \otimes (\hnabla_{\nu_{1}} u) \otimes \comFUu^{\nu_{1}\nu_{2}} \otimes (\hnabla_{\nu_{2}} u) ]  \EqLine{3} 
\\ & 
 -\tfrac{1}{4}(d + 2) (d + 4) (d + 6) \, \SO_{6}[ u \otimes (\hnabla_{\nu_{1}} u) \otimes u \otimes \comFUu^{\nu_{1}\nu_{2}} \otimes (\hnabla_{\nu_{2}} u) \otimes u ]  \EqLine{4} 
\\
& +(d + 4) (d + 6) \, \SO_{4}[ u \otimes (\hnabla_{\nu_{1}} u) \otimes (\hnabla_{\nu_{2}} u) \otimes (\hDelta^{\nu_{1}\nu_{2}} u) ]  \EqLine{1} 
 -4 (d + 4) \, \SO_{5}[ u \otimes (\hnabla_{\nu_{1}} u) \otimes (\hnabla_{\nu_{2}} u) \otimes (\hDelta^{\nu_{1}\nu_{2}} u) \otimes u ]  \EqLine{2} 
\\ & 
 -2 (d + 4) (d + 6) \, \SO_{5}[ u \otimes (\hnabla_{\nu_{1}} u) \otimes (\hnabla_{\nu_{2}} u) \otimes u \otimes (\hDelta^{\nu_{1}\nu_{2}} u) ]  \EqLine{3} 
\\ & 
 +2 (d + 4) (d + 6) \, \SO_{6}[ u \otimes (\hnabla_{\nu_{1}} u) \otimes (\hnabla_{\nu_{2}} u) \otimes u \otimes u \otimes (\hDelta^{\nu_{1}\nu_{2}} u) ]  \EqLine{4} 
\\ & 
 -2 (d + 4) (d + 6) \, \SO_{6}[ u \otimes (\hnabla_{\nu_{1}} u) \otimes u \otimes u \otimes (\hnabla_{\nu_{2}} u) \otimes (\hDelta^{\nu_{1}\nu_{2}} u) ]  \EqLine{5} 
\\ & 
 -4 (d + 4) (d + 6) \, \SO_{6}[ u \otimes u \otimes (\hnabla_{\nu_{1}} u) \otimes u \otimes (\hnabla_{\nu_{2}} u) \otimes (\hDelta^{\nu_{1}\nu_{2}} u) ]  \EqLine{6} 
\\ & 
 -6 (d + 4) (d + 6) \, \SO_{6}[ u \otimes u \otimes u \otimes (\hnabla_{\nu_{1}} u) \otimes (\hnabla_{\nu_{2}} u) \otimes (\hDelta^{\nu_{1}\nu_{2}} u) ]  \EqLine{7} 
\\ & 
 -(d + 4) (d + 6) (d + 8) \, \SO_{7}[ u \otimes (\hnabla_{\nu_{1}} u) \otimes u \otimes (\hnabla_{\nu_{2}} u) \otimes (\hDelta^{\nu_{1}\nu_{2}} u) \otimes u \otimes u ]  \EqLine{8} 
\\ & 
 -2 (d + 4) (d + 6) (d + 8) \, \SO_{7}[ u \otimes u \otimes (\hnabla_{\nu_{1}} u) \otimes (\hnabla_{\nu_{2}} u) \otimes (\hDelta^{\nu_{1}\nu_{2}} u) \otimes u \otimes u ]  \EqLine{9} 
\\
& +\tfrac{1}{2}d (d + 2) \, \SO_{4}[ u \otimes (\hnabla_{\nu_{1}} u) \otimes (\hnabla_{\nu_{2}} u) \otimes \comFUu^{\nu_{1}\nu_{2}} ]  \EqLine{1} 
 -\tfrac{1}{4}(d + 2) (d + 4)^{2} \, \SO_{5}[ u \otimes (\hnabla_{\nu_{1}} u) \otimes u \otimes (\hnabla_{\nu_{2}} u) \otimes \comFUu^{\nu_{1}\nu_{2}} ]  \EqLine{2} 
\\ & 
 -(d + 2) (d + 4) \, \SO_{5}[ u \otimes u \otimes (\hnabla_{\nu_{1}} u) \otimes (\hnabla_{\nu_{2}} u) \otimes \comFUu^{\nu_{1}\nu_{2}} ]  \EqLine{3} 
\\ & 
 +\tfrac{1}{2}(d + 2) (d + 4) (d + 6) \, \SO_{6}[ u \otimes (\hnabla_{\nu_{1}} u) \otimes u \otimes u \otimes (\hnabla_{\nu_{2}} u) \otimes \comFUu^{\nu_{1}\nu_{2}} ]  \EqLine{4} 
\\ & 
 +\tfrac{1}{2}(d + 2) (d + 4) (d + 6) \, \SO_{6}[ u \otimes u \otimes (\hnabla_{\nu_{1}} u) \otimes u \otimes (\hnabla_{\nu_{2}} u) \otimes \comFUu^{\nu_{1}\nu_{2}} ]  \EqLine{5} 
\\[\interblock]
& -\tfrac{1}{2}(3d + 8) \, \SO_{3}[ u \otimes (\hDelta u) \otimes (\hDelta u) ]  \EqLine{1} 
 +(d + 4) \, \SO_{4}[ u \otimes (\hDelta u) \otimes u \otimes (\hDelta u) ]  \EqLine{2} 
+(3d + 10) \, \SO_{4}[ u \otimes u \otimes (\hDelta u) \otimes (\hDelta u) ]  \EqLine{3} 
\\ & 
+\tfrac{1}{4}(d + 4) (d + 6) \, \SO_{5}[ u \otimes (\hDelta u) \otimes u \otimes (\hDelta u) \otimes u ]  \EqLine{4} 
 +2 (d + 4) \, \SO_{5}[ u \otimes (\hDelta u) \otimes (\hDelta u) \otimes u \otimes u ]  \EqLine{5} 
\\ & 
+(d + 4) (d + 6) \, \SO_{6}[ u \otimes u \otimes (\hDelta u) \otimes (\hDelta u) \otimes u \otimes u ]  \EqLine{6} 
\\
& -4 (d + 4) \, \SO_{5}[ u \otimes u \otimes (\hnabla^2_{\nu_{1}\nu_{2}} u) \otimes (\hDelta^{\nu_{1}\nu_{2}} u) \otimes u ]  \EqLine{1} 
+2 (d + 4) (d + 6) \, \SO_{6}[ u \otimes u \otimes (\hnabla^2_{\nu_{1}\nu_{2}} u) \otimes (\hDelta^{\nu_{1}\nu_{2}} u) \otimes u \otimes u ]  \EqLine{2} 
\\
& +\tfrac{1}{2}d (d + 2) \, \SO_{4}[ u \otimes u \otimes (\hnabla^2_{\nu_{1}\nu_{2}} u) \otimes \comFUu^{\nu_{1}\nu_{2}} ]  \EqLine{1} 
-(d + 2) (d + 4) \, \SO_{5}[ u \otimes u \otimes u \otimes (\hnabla^2_{\nu_{1}\nu_{2}} u) \otimes \comFUu^{\nu_{1}\nu_{2}} ]  \EqLine{2} 
\\[\interblock]
& +2 (d + 2) \, g^{\nu_{1}\nu_{3}} g^{\nu_{2}\nu_{4}} \, \SO_{4}[ u \otimes u \otimes (\hnabla_{\nu_{1}} u) \otimes (\hnabla^3_{\nu_{2}\nu_{3}\nu_{4}} u) ]  \EqLine{1} 
-2 (d + 4) \, g^{\nu_{1}\nu_{3}} g^{\nu_{2}\nu_{4}} \, \SO_{5}[ u \otimes (\hnabla_{\nu_{1}} u) \otimes u \otimes u \otimes (\hnabla^3_{\nu_{2}\nu_{3}\nu_{4}} u) ]  \EqLine{2} 
\\ & 
-2 (d + 4) \, g^{\nu_{1}\nu_{3}} g^{\nu_{2}\nu_{4}} \, \SO_{5}[ u \otimes u \otimes u \otimes (\hnabla_{\nu_{1}} u) \otimes (\hnabla^3_{\nu_{2}\nu_{3}\nu_{4}} u) ]  \EqLine{3} 
\\ & 
+(d + 4) (d + 6) \, g^{\nu_{1}\nu_{3}} g^{\nu_{2}\nu_{4}} \, \SO_{6}[ u \otimes (\hnabla_{\nu_{1}} u) \otimes u \otimes u \otimes u \otimes (\hnabla^3_{\nu_{2}\nu_{3}\nu_{4}} u) ]  \EqLine{4} 
\\ & 
-(d + 4) (d + 6) \, g^{\nu_{1}\nu_{3}} g^{\nu_{2}\nu_{4}} \, \SO_{6}[ u \otimes u \otimes (\hnabla_{\nu_{1}} u) \otimes u \otimes (\hnabla^3_{\nu_{2}\nu_{3}\nu_{4}} u) \otimes u ]  \EqLine{5} 
\\ & 
-(d + 4) (d + 6) \, g^{\nu_{1}\nu_{3}} g^{\nu_{2}\nu_{4}} \, \SO_{6}[ u \otimes u \otimes u \otimes (\hnabla_{\nu_{1}} u) \otimes (\hnabla^3_{\nu_{2}\nu_{3}\nu_{4}} u) \otimes u ]  \EqLine{6} 
\\ & 
-(d + 4) (d + 6) \, g^{\nu_{1}\nu_{3}} g^{\nu_{2}\nu_{4}} \, \SO_{6}[ u \otimes u \otimes u \otimes (\hnabla_{\nu_{1}} u) \otimes u \otimes (\hnabla^3_{\nu_{2}\nu_{3}\nu_{4}} u) ]  \EqLine{7} 
\\
& +2 (d + 4) \, g^{\nu_{1}\nu_{3}} g^{\nu_{2}\nu_{4}} \, \SO_{5}[ u \otimes u \otimes u \otimes (\hnabla^3_{\nu_{1}\nu_{2}\nu_{3}} u) \otimes (\hnabla_{\nu_{4}} u) ]  \EqLine{1} 
\\ & 
-(d + 4) (d + 6) \, g^{\nu_{1}\nu_{3}} g^{\nu_{2}\nu_{4}} \, \SO_{6}[ u \otimes u \otimes u \otimes (\hnabla^3_{\nu_{1}\nu_{2}\nu_{3}} u) \otimes (\hnabla_{\nu_{4}} u) \otimes u ]  \EqLine{2} 
\\
& +\tfrac{1}{6}(d^{2} + 4d + 24) \, g^{\nu_{1}\nu_{2}} \, \SO_{3}[ u \otimes (\hnabla_{\nu_{1}} u) \otimes (\hnabla_{\nu_{2}} (\hDelta u)) ]  \EqLine{1} 
 -\tfrac{1}{2}(d^{2} + 10d + 32) \, g^{\nu_{1}\nu_{2}} \, \SO_{4}[ u \otimes u \otimes (\hnabla_{\nu_{1}} u) \otimes (\hnabla_{\nu_{2}} (\hDelta u)) ]  \EqLine{2} 
\\ & 
 +(d + 2) \, g^{\nu_{1}\nu_{2}} \, \SO_{4}[ u \otimes (\hnabla_{\nu_{1}} u) \otimes u \otimes (\hnabla_{\nu_{2}} (\hDelta u)) ]  \EqLine{3} 
-\tfrac{1}{2}(d + 4) (d + 6) \, g^{\nu_{1}\nu_{2}} \, \SO_{5}[ u \otimes (\hnabla_{\nu_{1}} u) \otimes u \otimes u \otimes (\hnabla_{\nu_{2}} (\hDelta u)) ]  \EqLine{4} 
\\ & 
+\tfrac{1}{2}(d + 4) (d + 14) \, g^{\nu_{1}\nu_{2}} \, \SO_{5}[ u \otimes u \otimes u \otimes (\hnabla_{\nu_{1}} u) \otimes (\hnabla_{\nu_{2}} (\hDelta u)) ]  \EqLine{5} 
-2 (d + 4) \, g^{\nu_{1}\nu_{2}} \, \SO_{5}[ u \otimes (\hnabla_{\nu_{1}} u) \otimes u \otimes (\hnabla_{\nu_{2}} (\hDelta u)) \otimes u ]  \EqLine{6} 
\\ & 
-(d + 4) (d + 6) \, g^{\nu_{1}\nu_{2}} \, \SO_{6}[ u \otimes (\hnabla_{\nu_{1}} u) \otimes (\hnabla_{\nu_{2}} (\hDelta u)) \otimes u \otimes u \otimes u ]  \EqLine{7} 
\\
& -\tfrac{1}{12}(d + 4) (d + 6) \, g^{\nu_{1}\nu_{2}} \, \SO_{3}[ u \otimes (\hnabla_{\nu_{1}} u) \otimes (\hDelta (\hnabla_{\nu_{2}} u)) ]  \EqLine{1} 
 +\tfrac{1}{2}(d + 2) (d + 4) \, g^{\nu_{1}\nu_{2}} \, \SO_{4}[ u \otimes u \otimes (\hnabla_{\nu_{1}} u) \otimes (\hDelta (\hnabla_{\nu_{2}} u)) ]  \EqLine{2} 
\\ & 
+(d + 2) \, g^{\nu_{1}\nu_{2}} \, \SO_{4}[ u \otimes (\hnabla_{\nu_{1}} u) \otimes u \otimes (\hDelta (\hnabla_{\nu_{2}} u)) ]  \EqLine{3} 
+\tfrac{1}{2}(d + 4) (d + 6) \, g^{\nu_{1}\nu_{2}} \, \SO_{5}[ u \otimes (\hnabla_{\nu_{1}} u) \otimes (\hDelta (\hnabla_{\nu_{2}} u)) \otimes u \otimes u ]  \EqLine{4} 
\\ & 
-(d + 2) (d + 4) \, g^{\nu_{1}\nu_{2}} \, \SO_{5}[ u \otimes u \otimes u \otimes (\hnabla_{\nu_{1}} u) \otimes (\hDelta (\hnabla_{\nu_{2}} u)) ]  \EqLine{5} 
\\ & 
-(d + 4) (d + 6) \, g^{\nu_{1}\nu_{2}} \, \SO_{6}[ u \otimes (\hnabla_{\nu_{1}} u) \otimes (\hDelta (\hnabla_{\nu_{2}} u)) \otimes u \otimes u \otimes u ]  \EqLine{6} 
\\
& -(3d + 10) \, g^{\nu_{1}\nu_{2}} \, \SO_{4}[ u \otimes u \otimes (\anticomHLHN_{\nu_{1}} u) \otimes (\hnabla_{\nu_{2}} u) ]  \EqLine{1} 
+\tfrac{1}{2}(d + 4) (d + 6) \, g^{\nu_{1}\nu_{2}} \, \SO_{5}[ u \otimes u \otimes (\anticomHLHN_{\nu_{1}} u) \otimes (\hnabla_{\nu_{2}} u) \otimes u ]  \EqLine{2} 
\\ & 
+2 (d + 4) \, g^{\nu_{1}\nu_{2}} \, \SO_{5}[ u \otimes u \otimes (\anticomHLHN_{\nu_{1}} u) \otimes u \otimes (\hnabla_{\nu_{2}} u) ]  \EqLine{3} 
+8 (d + 4) \, g^{\nu_{1}\nu_{2}} \, \SO_{5}[ u \otimes u \otimes u \otimes (\anticomHLHN_{\nu_{1}} u) \otimes (\hnabla_{\nu_{2}} u) ]  \EqLine{4} 
\\ & 
 -(d + 4) (d + 6) \, g^{\nu_{1}\nu_{2}} \, \SO_{6}[ u \otimes u \otimes u \otimes (\anticomHLHN_{\nu_{1}} u) \otimes (\hnabla_{\nu_{2}} u) \otimes u ]  \EqLine{5} 
\\[\interblock]
& -2 (d + 4) \, g^{\nu_{1}\nu_{3}} g^{\nu_{2}\nu_{4}} \, \SO_{5}[ u \otimes u \otimes u \otimes u \otimes (\hnabla^4_{\nu_{1}\nu_{2}\nu_{3}\nu_{4}} u) ]  \EqLine{1} 
\\
& +(d + 2) \, g^{\nu_{1}\nu_{2}} \, \SO_{4}[ u \otimes u \otimes u \otimes (\hnabla_{\nu_{1}} (\hDelta (\hnabla_{\nu_{2}} u))) ]  \EqLine{1} 
 -2 (d + 4) \, g^{\nu_{1}\nu_{2}} \, \SO_{5}[ u \otimes u \otimes u \otimes u \otimes (\hnabla_{\nu_{1}} (\hDelta (\hnabla_{\nu_{2}} u))) ]  \EqLine{2} 
\\
& -\tfrac{1}{2}d \, \SO_{3}[ u \otimes u \otimes (\hDelta (\hDelta u)) ]  \EqLine{1} 
 +2 (d + 2) \, \SO_{4}[ u \otimes u \otimes u \otimes (\hDelta (\hDelta u)) ]  \EqLine{2} 
 -2 (d + 4) \, \SO_{5}[ u \otimes u \otimes u \otimes u \otimes (\hDelta (\hDelta u)) ]  \EqLine{3} 
\\[\interblock]
& +\tfrac{1}{6} \, \SC \, \SO_{1}[ q ]  \EqLine{1} 
+\SO_{2}[ q \otimes q ]  \EqLine{1} 
+\SO_{3}[ u \otimes (\hDelta q) \otimes u ]  \EqLine{1} 
\\[\interblock]
& +\tfrac{1}{2}(d + 2)^{2} \, g^{\nu_{1}\nu_{2}} \, \SO_{4}[ u \otimes (\hnabla_{\nu_{1}} u) \otimes (\hnabla_{\nu_{2}} u) \otimes q ]  \EqLine{1} 
-\tfrac{1}{2}(d + 2) (d + 4) \, g^{\nu_{1}\nu_{2}} \, \SO_{5}[ u \otimes (\hnabla_{\nu_{1}} u) \otimes u \otimes (\hnabla_{\nu_{2}} u) \otimes q ]  \EqLine{2} 
\\ & 
 -(d + 2) (d + 4) \, g^{\nu_{1}\nu_{2}} \, \SO_{5}[ u \otimes u \otimes (\hnabla_{\nu_{1}} u) \otimes (\hnabla_{\nu_{2}} u) \otimes q ]  \EqLine{3} 
\\
& -(d + 2) \, g^{\nu_{1}\nu_{2}} \, \SO_{4}[ u \otimes (\hnabla_{\nu_{1}} u) \otimes q \otimes (\hnabla_{\nu_{2}} u) ]  \EqLine{1} 
+\tfrac{1}{2}(d + 2) (d + 4) \, g^{\nu_{1}\nu_{2}} \, \SO_{5}[ u \otimes (\hnabla_{\nu_{1}} u) \otimes q \otimes (\hnabla_{\nu_{2}} u) \otimes u ]  \EqLine{2} 
\\
& -(d + 2) \, g^{\nu_{1}\nu_{2}} \, \SO_{4}[ u \otimes q \otimes (\hnabla_{\nu_{1}} u) \otimes (\hnabla_{\nu_{2}} u) ]  \EqLine{1} 
-(d + 2) \, g^{\nu_{1}\nu_{2}} \, \SO_{4}[ q \otimes u \otimes (\hnabla_{\nu_{1}} u) \otimes (\hnabla_{\nu_{2}} u) ]  \EqLine{2} 
\\ & 
 +\tfrac{1}{2}(d + 2) (d + 4) \, g^{\nu_{1}\nu_{2}} \, \SO_{5}[ u \otimes q \otimes (\hnabla_{\nu_{1}} u) \otimes (\hnabla_{\nu_{2}} u) \otimes u ]  \EqLine{3} 
 +\tfrac{1}{2}(d + 2) (d + 4) \, g^{\nu_{1}\nu_{2}} \, \SO_{5}[ q \otimes u \otimes (\hnabla_{\nu_{1}} u) \otimes (\hnabla_{\nu_{2}} u) \otimes u ]  \EqLine{4} 
\\
& +\sum_{\ell=1}^{3}(i^{(\ell)}_{q} \SO_{3})[ u \otimes (\hDelta u) ]  \EqLine{1} 
 -\tfrac{1}{2}(d + 2) \, \sum_{\ell=1}^{4}(i^{(\ell)}_{q} \SO_{4})[ u \otimes (\hDelta u) \otimes u ]  \EqLine{2} 
\\[\interblock]
& +2 \, g^{\nu_{1}\nu_{2}} \, \SO_{3}[ u \otimes (\hnabla_{\nu_{1}} q) \otimes (\hnabla_{\nu_{2}} u) ]  \EqLine{1} 
 -(d + 2) \, g^{\nu_{1}\nu_{2}} \, \SO_{4}[ u \otimes (\hnabla_{\nu_{1}} q) \otimes (\hnabla_{\nu_{2}} u) \otimes u ]  \EqLine{2} 
\\
& -(d + 2) \, g^{\nu_{1}\nu_{2}} \, \SO_{4}[ u \otimes (\hnabla_{\nu_{1}} u) \otimes (\hnabla_{\nu_{2}} q) \otimes u ]  \EqLine{1} 

%% file: RESULTS/u-parallel-full.tex
&
+\tfrac{1}{180} \, \abs{R}^2 \, \SO_{1}[ u ]  \EqLine{1} 
-\tfrac{1}{180} \, \abs{\Ric}^2 \, \SO_{1}[ u ]  \EqLine{1} 
+\tfrac{1}{30} \, (\hDelta \SC) \, \SO_{1}[ u ]  \EqLine{1} 
+\tfrac{1}{72} \, \SC^2 \, \SO_{1}[ u ]  \EqLine{1} 
+\tfrac{1}{6} \, \SC \, \SO_{1}[ q ]  \EqLine{1} 
+\SO_{2}[ q \otimes q ]  \EqLine{1} 
+\SO_{3}[ u \otimes (\hDelta q) \otimes u ]  \EqLine{1} 
\\
&
+\tfrac{1}{12} \, (\hnabla_{\nu_{1}} \SC) \, \SO_{1}[ N^{\nu_{1}} ]  \EqLine{1} 
 -\tfrac{1}{2} \, (\hnabla_{\nu_{1}} \SC) \, \SO_{3}[ u \otimes u \otimes N^{\nu_{1}} ]  \EqLine{2} 
+\tfrac{1}{6} \, \Ric_{\nu_{1}\nu_{2}} \, \SO_{2}[ N^{\nu_{1}} \otimes N^{\nu_{2}} ]  \EqLine{1} 
 -\tfrac{1}{2} \, \Ric_{\nu_{1}\nu_{2}} \, \SO_{3}[ N^{\nu_{1}} \otimes u \otimes N^{\nu_{2}} ]  \EqLine{2} 
\\
&
-\tfrac{1}{12} \, g_{\nu_{1}\nu_{2}} \SC \, \SO_{2}[ N^{\nu_{1}} \otimes N^{\nu_{2}} ]  \EqLine{1} 
-\tfrac{1}{2} \, g_{\nu_{1}\nu_{2}} \, \sum_{\ell=1}^{3}(i^{(\ell)}_{q} \SO_{3})[ N^{\nu_{1}} \otimes N^{\nu_{2}} ]  \EqLine{1} 
+\Sumgg_{\nu_{1}\nu_{2}\nu_{3}\nu_{4}} \, \SO_{4}[ N^{\nu_{1}} \otimes N^{\nu_{2}} \otimes N^{\nu_{3}} \otimes N^{\nu_{4}} ]  \EqLine{1} 
\\
&
-\SO_{3}[ u \otimes (\hnabla_{\nu_{1}} q) \otimes N^{\nu_{1}} ]  \EqLine{1} 
+\SO_{3}[ N^{\nu_{1}} \otimes (\hnabla_{\nu_{1}} q) \otimes u ]  \EqLine{1} 
\\
&
+\tfrac{1}{3} \, g^{\nu_{1}\nu_{3}} \Ric_{\nu_{2}\nu_{3}} \, \SO_{2}[ u \otimes (\hnabla_{\nu_{1}} N^{\nu_{2}}) ]  \EqLine{1} 
 -g^{\nu_{1}\nu_{3}} \Ric_{\nu_{2}\nu_{3}} \, \SO_{3}[ u \otimes u \otimes (\hnabla_{\nu_{1}} N^{\nu_{2}}) ]  \EqLine{2} 
-\tfrac{1}{6} \, \SC \, \SO_{2}[ u \otimes (\hnabla_{\nu_{1}} N^{\nu_{1}}) ]  \EqLine{1} 
\\
&
-\sum_{\ell=1}^{3}(i^{(\ell)}_{q} \SO_{3})[ u \otimes (\hnabla_{\nu_{1}} N^{\nu_{1}}) ]  \EqLine{1} 
\\
&
+\tfrac{1}{2} \, g_{\nu_{1}\nu_{2}} \, \SO_{3}[ N^{\nu_{1}} \otimes N^{\nu_{2}} \otimes (\hnabla_{\nu_{3}} N^{\nu_{3}}) ]  \EqLine{1} 
 -\tfrac{1}{2} \, g_{\nu_{1}\nu_{2}} \, \SO_{4}[ N^{\nu_{1}} \otimes N^{\nu_{2}} \otimes (\hnabla_{\nu_{3}} N^{\nu_{3}}) \otimes u ]  \EqLine{2} 
\\
&
-\tfrac{1}{2} \, g_{\nu_{1}\nu_{3}} \, \SO_{4}[ N^{\nu_{1}} \otimes N^{\nu_{2}} \otimes (\hnabla_{\nu_{2}} N^{\nu_{3}}) \otimes u ]  \EqLine{1} 
-\tfrac{1}{2} \, g_{\nu_{2}\nu_{3}} \, \SO_{4}[ N^{\nu_{1}} \otimes N^{\nu_{2}} \otimes (\hnabla_{\nu_{1}} N^{\nu_{3}}) \otimes u ]  \EqLine{1} 
\\
&
+\tfrac{1}{2} \, g_{\nu_{1}\nu_{3}} \, \SO_{4}[ N^{\nu_{1}} \otimes u \otimes (\hnabla_{\nu_{2}} N^{\nu_{2}}) \otimes N^{\nu_{3}} ]  \EqLine{1} 
 +\tfrac{1}{2} \, g_{\nu_{1}\nu_{3}} \, \SO_{4}[ u \otimes N^{\nu_{1}} \otimes (\hnabla_{\nu_{2}} N^{\nu_{2}}) \otimes N^{\nu_{3}} ]  \EqLine{2} 
\\
&
+\tfrac{1}{2} \, g_{\nu_{1}\nu_{3}} \, \SO_{4}[ N^{\nu_{1}} \otimes u \otimes (\hnabla_{\nu_{2}} N^{\nu_{3}}) \otimes N^{\nu_{2}} ]  \EqLine{1} 
 +\tfrac{1}{2} \, g_{\nu_{1}\nu_{3}} \, \SO_{4}[ u \otimes N^{\nu_{1}} \otimes (\hnabla_{\nu_{2}} N^{\nu_{3}}) \otimes N^{\nu_{2}} ]  \EqLine{2} 
\\
&
-\tfrac{1}{2} \, g_{\nu_{2}\nu_{3}} \, \SO_{4}[ N^{\nu_{1}} \otimes (\hnabla_{\nu_{1}} N^{\nu_{2}}) \otimes N^{\nu_{3}} \otimes u ]  \EqLine{1} 
 -\tfrac{1}{2} \, g_{\nu_{2}\nu_{3}} \, \SO_{4}[ N^{\nu_{1}} \otimes (\hnabla_{\nu_{1}} N^{\nu_{2}}) \otimes u \otimes N^{\nu_{3}} ]  \EqLine{2} 
\\
&
+\tfrac{1}{2} \, g_{\nu_{2}\nu_{3}} \, \SO_{4}[ u \otimes (\hnabla_{\nu_{1}} N^{\nu_{1}}) \otimes N^{\nu_{2}} \otimes N^{\nu_{3}} ]  \EqLine{1} 
+\tfrac{1}{2} \, g_{\nu_{2}\nu_{3}} \, \SO_{4}[ u \otimes (\hnabla_{\nu_{1}} N^{\nu_{2}}) \otimes N^{\nu_{1}} \otimes N^{\nu_{3}} ]  \EqLine{1} 
\\
&
+\tfrac{1}{2} \, g_{\nu_{2}\nu_{3}} \, \SO_{4}[ u \otimes (\hnabla_{\nu_{1}} N^{\nu_{2}}) \otimes N^{\nu_{3}} \otimes N^{\nu_{1}} ]  \EqLine{1} 
\\
&
-\SO_{4}[ u \otimes (\hnabla_{\nu_{1}} N^{\nu_{2}}) \otimes (\hnabla_{\nu_{2}} N^{\nu_{1}}) \otimes u ]  \EqLine{1} 
-g_{\nu_{2}\nu_{4}} g^{\nu_{1}\nu_{3}} \, \SO_{4}[ u \otimes (\hnabla_{\nu_{1}} N^{\nu_{2}}) \otimes (\hnabla_{\nu_{3}} N^{\nu_{4}}) \otimes u ]  \EqLine{1} 
\\
&
+\SO_{4}[ u \otimes (\hnabla_{\nu_{1}} N^{\nu_{1}}) \otimes u \otimes (\hnabla_{\nu_{2}} N^{\nu_{2}}) ]  \EqLine{1} 
 +2 \,  \, \SO_{4}[ u \otimes u \otimes (\hnabla_{\nu_{1}} N^{\nu_{1}}) \otimes (\hnabla_{\nu_{2}} N^{\nu_{2}}) ]  \EqLine{2} 
\\
&
+\tfrac{1}{2} \,  \, \SO_{3}[ u \otimes N^{\nu_{1}} \otimes \comFN^{\nu_{2}}_{\nu_{1}\nu_{2}} ]  \EqLine{1} 
 -\tfrac{1}{2} \,  \, \SO_{3}[ N^{\nu_{1}} \otimes u \otimes \comFN^{\nu_{2}}_{\nu_{1}\nu_{2}} ]  \EqLine{2} 
\\
&
-\SO_{3}[ N^{\nu_{1}} \otimes (\hDelta_{\nu_{1}\nu_{2}} N^{\nu_{2}}) \otimes u ]  \EqLine{1} 
 +2 \,  \, \SO_{4}[ N^{\nu_{1}} \otimes (\hDelta_{\nu_{1}\nu_{2}} N^{\nu_{2}}) \otimes u \otimes u ]  \EqLine{2} 
+2 \,  \, \SO_{4}[ u \otimes u \otimes (\hDelta_{\nu_{1}\nu_{2}} N^{\nu_{1}}) \otimes N^{\nu_{2}} ]  \EqLine{1} 
\\
&
-\tfrac{1}{2} \, g_{\nu_{1}\nu_{2}} \, \SO_{3}[ u \otimes (\hDelta N^{\nu_{1}}) \otimes N^{\nu_{2}} ]  \EqLine{1} 
 +g_{\nu_{1}\nu_{2}} \, \SO_{4}[ u \otimes u \otimes (\hDelta N^{\nu_{1}}) \otimes N^{\nu_{2}} ]  \EqLine{2} 
\\
&
-\tfrac{1}{2} \, g_{\nu_{1}\nu_{2}} \, \SO_{3}[ N^{\nu_{1}} \otimes (\hDelta N^{\nu_{2}}) \otimes u ]  \EqLine{1} 
 +g_{\nu_{1}\nu_{2}} \, \SO_{4}[ N^{\nu_{1}} \otimes (\hDelta N^{\nu_{2}}) \otimes u \otimes u ]  \EqLine{2} 
+2 \, g^{\nu_{1}\nu_{3}} \, \SO_{4}[ u \otimes u \otimes u \otimes (\hnabla^3_{\nu_{1}\nu_{2}\nu_{3}} N^{\nu_{2}}) ]  \EqLine{1} 
\\
&
-\SO_{3}[ u \otimes u \otimes (\anticomHLHN_{\nu_{1}} N^{\nu_{1}}) ]  \EqLine{1} 
 +2 \,  \, \SO_{4}[ u \otimes u \otimes u \otimes (\anticomHLHN_{\nu_{1}} N^{\nu_{1}}) ]  \EqLine{2} 
\\
&
+\tfrac{1}{2} \,  \, \SO_{2}[ N^{\nu_{1}} \otimes N^{\nu_{2}} ]  \, F_{\nu_{1}\nu_{2}} \EqLine{1} 
 -\SO_{3}[ N^{\nu_{1}} \otimes u \otimes N^{\nu_{2}} ]  \, F_{\nu_{1}\nu_{2}} \EqLine{2} 
-g^{\nu_{2}\nu_{3}} \, \SO_{3}[ u \otimes (\hnabla_{\nu_{3}} N^{\nu_{1}}) \otimes u ]  \, F_{\nu_{1}\nu_{2}} \EqLine{1} 
\\
&
+\tfrac{1}{6} \, g^{\nu_{1}\nu_{3}} \, \SO_{1}[ N^{\nu_{2}} ]  \, (\hnabla_{\nu_{1}} F_{\nu_{2}\nu_{3}}) \EqLine{1} 
 -g^{\nu_{1}\nu_{3}} \, \SO_{3}[ u \otimes N^{\nu_{2}} \otimes u ]  \, (\hnabla_{\nu_{1}} F_{\nu_{2}\nu_{3}}) \EqLine{2} 
+\tfrac{1}{12} \,  \, \SO_{1}[ u ]  \, F^{\nu_{1}\nu_{2}} F_{\nu_{1}\nu_{2}} \EqLine{1} 

%% file: RESULTS/u-parallel-N0.tex
&
+\tfrac{1}{180} \, \abs{R}^2 \, \SO_{1}[ u ]  \EqLine{1} 
-\tfrac{1}{180} \, \abs{\Ric}^2 \, \SO_{1}[ u ]  \EqLine{1} 
+\tfrac{1}{30} \, (\hDelta \SC) \, \SO_{1}[ u ]  \EqLine{1} 
+\tfrac{1}{72} \, \SC^2 \, \SO_{1}[ u ]  \EqLine{1} 
+\tfrac{1}{12} \,  \, \SO_{1}[ u ]  \, F^{\nu_{1}\nu_{2}} F_{\nu_{1}\nu_{2}} \EqLine{1} 
\\
&
+\tfrac{1}{6} \, \SC \, \SO_{1}[ q ]  \EqLine{1} 
+\SO_{2}[ q \otimes q ]  \EqLine{1} 
+\SO_{3}[ u \otimes (\hDelta q) \otimes u ]  \EqLine{1} 